\documentclass[twoside,11pt,reqno]{amsart}
\usepackage{amssymb}
\usepackage{mathabx}
\usepackage[cmtip,arrow,all]{xy}
\usepackage{pb-diagram,pb-xy}
\usepackage{epic}
\usepackage{bm}
\usepackage[pdftex,colorlinks,hypertexnames=false,linkcolor=blue,urlcolor=blue,citecolor=blue]{hyperref}

\usepackage{wasysym}
\SetSymbolFont{wasy}{bold}{U}{wasy}{m}{n}

\numberwithin{equation}{section}
\newcommand{\arxiv}[1]{{\tt arXiv:#1}}
\newtheorem{Proposition}{Proposition}[section]
\newtheorem{Lemma}[Proposition]{Lemma}
\newtheorem{Theorem}[Proposition]{Theorem}
\newtheorem{Corollary}[Proposition]{Corollary}

\theoremstyle{definition}
\newtheorem{Conjecture}[Proposition]{Conjecture}

\theoremstyle{remark}
\newtheorem{Remark}[Proposition]{Remark}

\def\V{\mathbb V}
\def\O{\mathcal O}

\def\tabA{\mathsf A}
\def\tabB{\mathsf B}
\def\tabC{\mathsf C}
\def\tabD{\mathsf D}

\def\AA{A}
\def\BB{B}
\def\CC{C}

\def\sub{\subseteq}
\newcommand{\Tab}{\operatorname{Tab}}
\newcommand{\col}{\operatorname{col}}

\newcommand{\row}{\operatorname{row}}

\newcommand{\g}{\mathfrak{g}}
\newcommand{\rr}{\mathfrak{r}}
\newcommand{\n}{\mathfrak{n}}
\newcommand{\h}{\mathfrak{h}}
\newcommand{\gl}{\mathfrak{gl}}
\newcommand{\m}{\mathfrak{m}}
\newcommand{\p}{\mathfrak{p}}

\def\qdim{\operatorname{dim}_q}
\def\upit{\downtouparrow}
\def\HC{\operatorname{HC}}
\def\hc{\operatorname{hc}}
\def\lex{\operatorname{lex}}

\def\rsmod{\operatorname{smod-}\!}
\def\rsmodchi{\operatorname{smod_\chi-}\!}
\def\lmod{\!\operatorname{-mod}}
\def\lsmod{\!\operatorname{-smod}}
\def\qdim{\operatorname{dim}_q}
\newcommand{\sqbinom}[2]{\genfrac{[}{]}{0pt}{}{#1}{#2}}
\def\op{\operatorname{op}}

\def\pr{\operatorname{pr}}
\def\prr{\mathrm{p}}

\def\lmof{\!\operatorname{-mod_{\operatorname{fd}}}}
\def\grlmof{\!\operatorname{-grmod_{\operatorname{fd}}}}
\def\lsmof{\!\operatorname{-smod_{\operatorname{fd}}}}

\def\mat{\operatorname{def}}
\def\atyp{\operatorname{atyp}}
\def\coloneqq{:=}
\def\deg{\operatorname{deg}}
\def\parity{\operatorname{par}}
\def\Hom{\operatorname{Hom}}

\def\id{\operatorname{id}}

\def\wt{{\operatorname{wt}}}
\def\End{{\operatorname{End}}}
\def\Rep{{\operatorname{Rep}}}
\def\ostar{\circledast}
\def\Q{{\mathbb Q}}
\def\C{{\mathbb C}}
\def\Z{{\mathbb Z}}
\def\N{{\mathbb N}}
\def\0{{\bar 0}}
\def\1{{\bar 1}}
\def\eps{{\varepsilon}}
\def\sig{{\sigma}}
\def\bsig{{\text{\boldmath$\sigma$}}}
\def\b{\mathfrak b}

\def\g{{\mathfrak g}}
\def\gl{\mathfrak{gl}}
\def\h{\mathfrak h}

\def\m{\mathfrak m}
\def\n{\mathfrak n}
\def\p{\mathfrak p}
\def\t{\mathfrak t}
\def\cH{\mathcal H}

\def\cS{\mathcal S}

\def\roweq{\sim}
\def\linked{\approx}

\title[Whittaker coinvariants]{Whittaker coinvariants for $\bm{\mathrm{GL}(m|n)}$}
\author{Jonathan Brundan and Simon M.~Goodwin}

\address{Department of Mathematics, University of Oregon, Eugene, OR 97403, USA}
\email{brundan@uoregon.edu}

\address{School of Mathematics,
University of Birmingham,
Birmingham, B15 2TT,
UK}
\email{s.m.goodwin@bham.ac.uk}

\thanks{2010 {\it Mathematics Subject Classification}: 17B10, 17B37.}

\thanks{First author supported in part by NSF grant nos. DMS-1161094
  and DMS-1700905.}
\thanks{Second author supported in part by EPSRC grant no. EP/R018952/1.}

\begin{document}

\begin{abstract}
Let $W_{m|n}$ be
the (finite) $W$-algebra attached to the
principal
nilpotent orbit in the general linear Lie superalgebra
$\mathfrak{gl}_{m|n}(\mathbb{C})$.
In this paper we study the {\em Whittaker coinvariants functor}, which
is an exact functor from
category $\mathcal O$ for
$\mathfrak{gl}_{m|n}(\mathbb{C})$ to a certain category of
finite-dimensional modules over $W_{m|n}$.
We show that this functor has
properties similar to Soergel's
functor $\V$ in the setting of category $\mathcal O$ for a semisimple
Lie algebra.
We also use it to compute the center of $W_{m|n}$ explicitly, and deduce
consequences for
the classification of blocks of $\mathcal O$ up to Morita/derived equivalence.
\end{abstract}

\maketitle

\section{Introduction}

This article is a sequel to \cite{BBG}, in which we began a study of the
{\em principal $W$-algebra} $W = W_{m|n}$ associated to the general linear Lie superalgebra
$\mathfrak{g} = \mathfrak{gl}_{m|n}(\C)$.
This associative superalgebra is a quantization
of the Slodowy slice to the principal nilpotent orbit in $\g$; see
e.g.\ \cite{P, GG, Losev} for more
about (finite) $W$-algebras in the purely even case.

There are several different approaches to the construction of $W$.
We begin by briefly recalling one of these in more detail.
Since
$\mathfrak{gl}_{m|n}(\C) \cong
\mathfrak{gl}_{n|m}(\C)$, there is no loss in generality in assuming
throughout the article that $m \leq n$.
Pick a nilpotent element $e \in \g_{\0}$ with just two Jordan blocks (necessarily of sizes $m$ and $n$), and let
$\g = \bigoplus_{d \in \Z} \g(d)$ be a good grading for $e \in \g(1)$.
Let $\p := \bigoplus_{d \geq 0} \g(d)$ and $\m := \bigoplus_{d < 0} \g(d)$.
We get a generic character
$\chi:\m \to \C$ by taking the supertrace form with $e$.
Setting
$\m_\chi := \{x-\chi(x) \:|\:x \in \m\} \sub U(\m)$, we then have by definition that
$$
W := \{u \in U(\p) \:|\:  u \, \m_\chi \sub \m_\chi U(\g) \}.
$$
In \cite{BBG}, we obtained a presentation for $W$ by generators and relations, showing that it is a certain truncated shifted version of the Yangian
$Y(\gl_{1|1})$. In
particular, it is quite close to being supercommutative.
We also classified its irreducible representations via highest weight theory.
Every irreducible representation arises as a quotient of an appropriately defined Verma module, all of which have dimension $2^m$.
Then there is another more
explicit construction of the irreducible representations, implying that
they have dimension
$2^{m-t}$ for some {\em atypicality} $0 \leq t \leq m$.

By a {\em Whittaker vector},
we mean a vector $v$ in some right $\g$-module
such that
$v x = \chi(x) v$ for each $x \in \m$; equivalently, $v\,\m_\chi  = 0$.
This is the appropriate analog for $\g$ of the notion of a Whittaker vector
for a semisimple Lie algebra
as studied in
Kostant's classic paper \cite{K}.  From the definition of $W$, we see that the space
of Whittaker vectors, which we denote by $H^0(M)$, is a right $W$-module.
We refer to $H^0$ as the {\em Whittaker invariants} functor.
On the other hand, for a left $\g$-module $M$, it is clear from the definition of $W$ that the space
$H_0(M):= M / \m_\chi M$
of {\em Whittaker coinvariants} is a left $W$-module.
The restriction of this functor to
the
BGG category $\O$ for $\g$
(defined with respect to the
standard Borel subalgebra $\b$ of $\g$ such that $\b_{\0} = \p_{\0}$)
gives an exact functor
from $\O$ to the
category of finite-dimensional left $W$-modules.

The main goal in the first part of this article is to
describe the effect of $H_0$ on
various natural families of modules in $\mathcal O$.
In particular, in
Theorem~\ref{T:main}, we show that it sends
Verma modules in $\mathcal O$ (induced from the standard Borel $\b$) to
the Verma modules for $W$.
Our proof of this is elementary but surprisingly technical, and
it turns out
to be the key
ingredient needed for many things after that. We use it to
show that $H_0$ sends
irreducible modules in $\mathcal O$ of maximal Gelfand--Kirillov dimension
to irreducible $W$-modules, and it sends all other irreducibles in
$\mathcal O$ to zero.
Moreover, every irreducible $W$-module arises in this way.
We also compute
the composition multiplicities of Verma modules for $W$.
They always have composition length $2^t$ where
$t$ is the atypicality mentioned earlier, but are not necessarily
multiplicity-free.
As another more surprising application, we deduce that the
center of $W$ is canonically isomorphic to the center of $U(\g)$; see Theorem~\ref{zonto}.
 Thus central characters for $\g$ and $W$ are identified.

After that, we restrict
attention just to the subcategory
$\mathcal O_\Z$ of $\mathcal O$ that is the sum of all of its blocks
with
integral central character.
Let $\overline{\mathcal O}_\Z$ be the full subcategory of
$W\lmod$ consisting of the $W$-modules
isomorphic to $H_0(M)$ for
$M \in \mathcal O_\Z$.
We show that $\overline{\mathcal O}_\Z$ is Abelian,
and
the Whittaker coinvariants functor restricts to an
exact functor
$$H_0:
\O_\Z
\to\overline{\mathcal O}_\Z
$$
which
satisfies the universal property of the quotient of $\mathcal
O_\Z$ by the Serre subcategory $\mathcal T_\Z$
consisting of all the modules of less
than maximal Gelfand--Kirillov dimension; see Theorem~\ref{qmain}.
Thus,
$\overline{\mathcal O}_\Z$
is an explicit realization of the
Serre quotient $\mathcal O_\Z / \mathcal T_\Z$.
By
\cite[Theorem 4.10]{BLW}, the quotient functor
$\O_\Z \rightarrow \O_\Z / \mathcal T_\Z$
is fully faithful on projectives, hence, so too is $H_0$.
This is reminiscent of a result of Backelin \cite{Back} in the setting
of category $\O$ for a semisimple Lie algebra. Backelin's result
was based ultimately on
the {\em Struktursatz} from \cite{Soergel}.
In that case,
Soergel's
{\em Endomorphismensatz}
shows moreover
that the blocks of the quotient category
can be realized explicitly in terms of the
cohomology algebras of some underlying partial flag varieties.

It would be very interesting to
establish some sort
of analog of Soergel's Endomorphismensatz in the super case.
Ideally, this would
give an
explicit
combinatorial description (e.g.\ by quiver and relations) of the basic algebras
$\BB_\xi$ that are
Morita equivalent to the various blocks $\overline{\mathcal O}_\xi$
of our category $\overline{\mathcal O}_\Z$. Note these
algebras are not commutative in general;
e.g.\ see \cite[Example 4.7 and Remark 4.8]{Bsurvey} for some baby
examples.
In Soergel's proof of
the Endomorphismensatz,
the cohomology algebras of
partial flag varieties
arise as quotients of $Z(\g)$,
which is also
the principal $W$-algebra in that setting
according to \cite{K}.
Paralleling this in the super case, we show that all the maximally atypical $\BB_\xi$'s can be
realized as quotients of a certain idempotented form
$\dot W$ of $W$; see Theorems~\ref{landing} and \ref{moreispossible}.
We also compute explicitly the Cartan matrix of
$\BB_\xi$; see Theorem~\ref{crazy} for an elementary proof based on
properties of the Whittaker coinvariants functor, and Theorem~\ref{crazier}
for a proof based on the super Kazhdan-Lusztig conjecture of
\cite{CLW, BLW}
(which has the advantage of incorporating the natural grading).

We end the article by
discussing some applications to the classification of blocks of
$\O_\Z$, both
up to Morita equivalence and up to gradable derived equivalence in the
sense of \cite[Definition 4.2]{CM}; see Theorems~\ref{freebie} and
\ref{moritatheorem} and Conjectures~\ref{firstc} and \ref{crossbow}.

Finally in the introduction, we draw attention to the work of Losev in \cite{Losev2}.
This paper includes a study of Whittaker coinvariants functors associated to
arbitrary nilpotent orbits in semisimple Lie algebras. Lie superalgebras are not considered in detail,
though some remarks are made about how the theory may apply in this situation in \cite[\S6.3.2]{Losev2}.
His theory includes many of the features
discussed above. In particular, he also views these functors as some generalized Soergel functors.
The approach in this paper is quite different and leads to more explicit results, which we require
for the subsequent applications.

\vspace{2mm}
\noindent
{\em Acknowledgements.}
This article was written in part during  the programme ``Local
Representation Theory and Simple Groups''
held at the Bernoulli Center, EPFL, Lausanne, Switzerland in Autumn 2016.
We also thank Kevin Coulembier for pointing out the counterexample
mentioned after Conjecture 4.37.

\vspace{2mm}
\noindent
{\em Notation.}
We fix once and for all some choice of {\em parity function}
$\parity:\C \rightarrow \Z/2$
such that $\parity(0) = \0$ and $\parity(z+1) = \parity(z)+\1$ for all $z \in \C$.
Also, $\leq$ denotes the partial order on $\C$
defined by $z \leq w$ if $w-z \in \N$.

\section{Category \texorpdfstring{$\mathcal O$}{O} for the general linear Lie superalgebra}\label{sprelim}
In this section, we set up our general combinatorial notation, then
review various standard facts about category
$\mathcal O$ for $\mathfrak{gl}_{m|n}(\C)$.
We assume from the outset that $m \leq n$ as this will be essential
when we introduce the $W$-algebra in the next section, but note
that the general results in this section do not depend on this hypothesis.

\subsection{Combinatorics}\label{scomb}
We fix integers $0 \leq m \leq n$
and a two-rowed {\em pyramid} $\pi$
with $m$ boxes in the first
(top) row and $n$ boxes in the second (bottom) row.
We require that the top row does not jut out past the bottom row.
For example, here are the possible pyramids for $m = 2$ and $n = 5$:
$$
\begin{array}{c}
\begin{picture}(60,24) \put(0,0){\line(1,0){60}}
\put(0,12){\line(1,0){60}} \put(0,24){\line(1,0){24}}
\put(0,0){\line(0,1){24}} \put(12,0){\line(0,1){24}}
\put(24,0){\line(0,1){24}} \put(36,0){\line(0,1){12}}
\put(48,0){\line(0,1){12}}
\put(60,0){\line(0,1){12}}
\put(3,14.5){\hbox{1}}
\put(15,14.5){\hbox{2}}
\put(3,2.5){\hbox{3}}
\put(15,2.5){\hbox{4}}
\put(27,2.5){\hbox{5}}
\put(39,2.5){\hbox{6}}
\put(51,2.5){\hbox{7}}
\end{picture}
\end{array}
,\:\:\:\:
\begin{array}{c}
\begin{picture}(60,24) \put(0,0){\line(1,0){60}}
\put(0,12){\line(1,0){60}} \put(12,24){\line(1,0){24}}
\put(0,0){\line(0,1){12}} \put(12,0){\line(0,1){24}}
\put(24,0){\line(0,1){24}} \put(36,0){\line(0,1){24}}
\put(48,0){\line(0,1){12}}
\put(60,0){\line(0,1){12}}
\put(15,14.5){\hbox{1}}
\put(27,14.5){\hbox{2}}
\put(3,2.5){\hbox{3}}
\put(15,2.5){\hbox{4}}
\put(27,2.5){\hbox{5}}
\put(39,2.5){\hbox{6}}
\put(51,2.5){\hbox{7}}
\end{picture}
\end{array},\:\:\:\:
\begin{array}{c}
\begin{picture}(60,24) \put(0,0){\line(1,0){60}}
\put(0,12){\line(1,0){60}} \put(24,24){\line(1,0){24}}
\put(0,0){\line(0,1){12}} \put(12,0){\line(0,1){12}}
\put(24,0){\line(0,1){24}} \put(36,0){\line(0,1){24}}
\put(48,0){\line(0,1){24}}
\put(60,0){\line(0,1){12}}
\put(27,14.5){\hbox{1}}
\put(39,14.5){\hbox{2}}
\put(3,2.5){\hbox{3}}
\put(15,2.5){\hbox{4}}
\put(27,2.5){\hbox{5}}
\put(39,2.5){\hbox{6}}
\put(51,2.5){\hbox{7}}
\end{picture}\end{array}
,\:\:\:\:
\begin{array}{c}
\begin{picture}(60,24) \put(0,0){\line(1,0){60}}
\put(0,12){\line(1,0){60}} \put(36,24){\line(1,0){24}}
\put(0,0){\line(0,1){12}} \put(12,0){\line(0,1){12}}
\put(24,0){\line(0,1){12}} \put(36,0){\line(0,1){24}}
\put(48,0){\line(0,1){24}}
\put(60,0){\line(0,1){24}}
\put(39,14.5){\hbox{1}}
\put(51,14.5){\hbox{2}}
\put(3,2.5){\hbox{3}}
\put(15,2.5){\hbox{4}}
\put(27,2.5){\hbox{5}}
\put(39,2.5){\hbox{6}}
\put(51,2.5){\hbox{7}}
\end{picture}
\end{array}.
$$
As in these examples, we number the boxes of $\pi$ by $1,\dots,m+n$, so that the boxes in the first (resp.\ second) row are indexed $1,\dots m$ (resp.\ $m+1,\dots,m+n$) from left to right.
Then we write $\row(i)$ and $\col(i)$ for the row and column numbers
of the $i$th box of $\pi$, numbering columns by $1,\dots,n$ in order from left to
right.
Also we denote the number of columns of height $1$ on the
left (resp. right) side of $\pi$ by $s_-$ (resp. $s_+$);
in the degenerate case $m=0$,
one should instead pick any $s_-,s_+\geq 0$ with $s_-+s_+ = n$.

A {\em $\pi$-tableau} is a filling of the boxes of
the pyramid $\pi$ by
complex numbers.
Let $\Tab$ denote the set of all such $\pi$-tableaux.
Sometimes we will represent $\tabA \in \Tab$
as an array
$\tabA
 = \substack{a_1 \cdots a_{m} \\ b_1 \cdots b_{n}}$ of complex numbers.
Here are a few combinatorial notions
about tableaux.
\begin{itemize}
\item
For $\tabA = \substack{a_1 \cdots a_{m} \\ b_1 \cdots b_{n}}$,
we let $a(\tabA) := a_1+\cdots+a_m$ and $b(\tabA) := b_1+\cdots+b_n$
be the sum of the entries on its top and bottom rows, respectively.
\item We say that
$\tabA= \substack{a_1 \cdots a_{m} \\ b_1 \cdots b_{n}}$
is {\em dominant} if $a_1 > \cdots > a_m$ and $b_1 < \cdots < b_n$.
\item We say that
$\tabA= \substack{a_1 \cdots a_{m} \\ b_1 \cdots b_{n}}$
is {\em anti-dominant} if $a_i \not > a_j$ for each $1 \leq
i < j \leq m$ and $b_i \not < b_j$ for each $1 \leq i < j
\leq n$.
\item
A {\em matched pair} in $\tabA$ is a pair of equal entries
from the same column.
\item The {\em defect} $\mat(\tabA)$
is the number of matched pairs in $\tabA$.
\item
Two $\pi$-tableaux $\tabA, \tabB$ are {\em row equivalent},
denoted $\tabA \roweq \tabB$, if $\tabB$ can be obtained from $\tabA$ by rearranging
entries within each row.
\item
The {\em degree of atypicality} $\atyp(\tabA)$
is the maximal defect of any $\tabB \roweq \tabA$.
\item
We write $\tabB \upit \tabA$ if $\tabB$ can be obtained from $\tabA$
by picking several of the matched pairs in $\tabA$ and subtracting $1$ from each of them.
There are $2^{\mat(\tabA)}$ such tableaux $\tabB$.
\item
The {\em Bruhat order} $\preceq$ on $\Tab$ is the smallest partial order
such that $\tabB \prec \tabA$ whenever one of the following holds:
\begin{itemize}
\item
$\tabB$ is obtained from $\tabA= \substack{a_1 \cdots a_{m} \\ b_1 \cdots b_{n}}$ by interchanging
$a_i$ and $a_j$, assuming
$a_i > a_j$
for some $1 \leq i < j \leq m$;
\item
$\tabB$ is obtained from $\tabA= \substack{a_1 \cdots a_{m} \\ b_1 \cdots b_{n}}$ by interchanging
$b_i$ and $b_j$, assuming
$b_i < b_j$
for some $1 \leq i < j \leq n$;
\item
$\tabB$ is obtained from $\tabA=
\substack{a_1 \cdots a_{m} \\ b_1 \cdots b_{n}}$ by subtracting one from both $a_i$ and $b_j$,
assuming $a_i = b_j$ for some
$1 \leq i \leq m$ and $1 \leq j \leq n$.
\end{itemize}
\item
Let $\linked$ be the equivalence relation generated by the Bruhat order
$\preceq$.
We refer to the $\linked$-equivalence classes as {\em linkage
  classes}.
All $\pi$-tableaux in a given linkage class $\xi$ have the same atypicality.
\end{itemize}
The representation theoretic significance of these definitions
will be made clear later in the article.

\subsection{Modules and supermodules}\label{itwas}
In the introduction we have ignored the distinction between modules
and supermodules.
We will be more careful in the remainder of the article.
For an associative algebra $A$, we write $A\lmod$ or $A\lmof$ for the categories of left
$A$-modules or finite-dimensional left $A$-modules, respectively.

Superalgebras
and supermodules are
objects in the symmetric monoidal
category of vector superspaces.
We denote the parity of a homogeneous vector $v$ in a vector superspace
by $|v| \in
\Z/2$, and
recall that the tensor flip  $V \otimes W \stackrel{\sim}{\rightarrow} W \otimes V$
 is given on homogeneous vectors by $v\otimes w \mapsto (-1)^{|v||w|} w \otimes v$.
The notation $[.\,,.]$ always denotes the {\em supercommutator}
$[x,y] = xy - (-1)^{|x||y|} yx$ of homogeneous elements of
a superalgebra.

Let $A$ be an associative
superalgebra.  A left {\em $A$-supermodule}
is a superspace $M = M_\0 \oplus M_\1$ equipped with a linear
left action of $A$ such that $A_i M_j \subseteq M_{i+j}$.
A {\em supermodule homomorphism} is a parity-preserving linear map
that is a
homomorphism
in the usual sense.
We write $A\lsmod$ for the Abelian category of all left
$A$-supermodules and supermodule homomorphisms,
and $A\lsmof$ for the subcategory of finite-dimensional ones.
We denote the usual {\em parity switching functor} on all of these categories
by $\Pi$.

\subsection{Super category \texorpdfstring{$\mathcal O$}{O}}\label{so}
Let $\mathfrak g$ be the Lie superalgebra
$\mathfrak{gl}_{m|n}(\C)$. We write $e_{i,j}$ for the $ij$-matrix unit in $\g$, which is
of parity $|i|+|j|$ where
$$
|i| = \begin{cases} \0 & \text{for $1 \leq i \leq m$}, \\
\1 & \text{for $m+1 \leq i \leq m+n$}.
\end{cases}
$$
Let $\t$ be the Cartan subalgebra of $\g$ consisting of all diagonal
matrices
and $\{\delta_i\}_{1\leq i \leq m+n}$ be the basis for $\t^*$
dual to the basis $\{e_{i,i}\}_{1 \leq i \leq m+n}$ of $\t$.
The usual {\em supertrace form}
$(.\,,.)$ on $\g$
induces a non-degenerate super-symmetric bilinear
form $(.\,,.)$ on $\t^*$
such that
$(\delta_i,\delta_j) = (-1)^{|i|} \delta_{i,j}$.

Now suppose that $\lhd$ is a total order on the set
$\{1,\dots,m+n\}$.
Let $\mathfrak{b}^{\lhd}$ be the Borel subalgebra of $\g$ spanned by
$\{e_{i,j}\}_{i \unlhd j}$.
Then define $\mathcal O^{\lhd}$ to be the full category of $U(\mathfrak{g})\lsmod$
consisting of all
$\mathfrak{g}$-supermodules $M$ such that
\begin{itemize}
\item
$M$ is
finitely generated over $\g$;
\item
$M$ is locally finite-dimensional over
$\mathfrak{b}^{\lhd}$;
\item $M$ is semisimple over $\mathfrak{t}$;
\item
the $\lambda$-weight space $M_\lambda$
of $M$ is concentrated in parity\footnote{We have made this particular
 choice so that (\ref{e:espressoclub}) holds; it is important also in
 the proof of Lemma \ref{bookkeeping}.}
\begin{align}\label{parity2}
\parity(\lambda) :=
\parity((\lambda,\delta_{m+1}+\cdots+\delta_{m+n})) + \lceil (n-m)/2
  \rceil+m\,s_- \in \Z/2.
\end{align}
\end{itemize}
The parity assumption
means that one can simply forget
the $\Z/2$-grading on
objects
of $\mathcal O^{\lhd}$,
since it can be recovered uniquely from the weights.
The reader should not be concerned about the dependence on the
choice
of the function $\parity :\C\rightarrow\Z/2$ (which was made at the end of the introduction):
the categories $\mathcal O^\lhd$ arising from two different
choices are obviously equivalent.
We note that all objects of $\mathcal O^{\lhd}$ are of finite length.

Introduce the weight $\rho^\lhd \in \t^*$ so that
\begin{equation} \label{e:rho}
(\rho^\lhd,\delta_j) =
\#\left\{i \unlhd j\:\big|\:|i| = \1\right\} - \#\left\{i \lhd
  j\:\big|\: |i| = \0\right\}
\end{equation}
for $j=1,\dots,m+n$.
For $\tabA \in \Tab$, let
$\lambda^\lhd_\tabA \in \t^*$ be the unique weight such that
$(\lambda^\lhd_\tabA+\rho^\lhd,\delta_j)$
is the entry in the $j$th box of $\tabA$.
Then we let $M^{\lhd}(\tabA)$ denote the
{\em Verma supermodule} of $\mathfrak{b}^{\lhd}$-highest weight
$\lambda^\lhd_\tabA$, i.e.\
\begin{equation}\label{verma}
M^{\lhd}(\tabA) := U(\g) \otimes_{U(\b^{\lhd})}\C^{\lhd}_\tabA
\end{equation}
where $\C^{\lhd}_\tabA$ is a one-dimensional
$\mathfrak{b}^{\lhd}$-supermodule of weight $\lambda_\tabA^\lhd$
concentrated in parity $\parity(\lambda_\tabA^\lhd)$. Note the parity
choice here is
forced upon us since we want $M^{\lhd}(\tabA)$ to belong to $\mathcal O^{\lhd}$.
The Verma supermodule $M^{\lhd}(\tabA)$ has a unique irreducible quotient $L^{\lhd}(\tabA)$,
and the supermodules
$\{L^{\lhd}(\tabA)\:|\:\tabA \in \Tab\}$ give a complete
set of inequivalent irreducible objects in $\mathcal O^{\lhd}$.

By a {\em normal order} we mean a total order $\lhd$
such that $1 \lhd \cdots \lhd m$ and $m+1 \lhd \cdots
\lhd m+n$.
For any normal order $\lhd$,
the underlying even subalgebra $\b^{\lhd}_\0$
is equal simply to the usual standard Borel subalgebra
of $\mathfrak{g}_\0 =
\mathfrak{gl}_m(\C) \oplus \mathfrak{gl}_n(\C)$ consisting of upper
triangular matrices.
Observing that a $\mathfrak{g}$-supermodule is locally finite over
$\b^{\lhd}$
if and only if it is locally finite over
$\b^{\lhd}_\0$, it is clear
for any two normal orders $\lhd$ and $\LHD$
that $\mathcal O^{\lhd} = \mathcal O^{\LHD}$. Henceforth, we denote
this category coming from a normal order simply by
$\mathcal O$.

Let us explain how to translate between the various labellings
of the irreducible objects of $\mathcal O$ arising from different
choices of normal order.
The basic technique
to pass from $\lhd$ to $\LHD$
is to apply a sequence of
{\em odd reflections} connecting $\lhd$ to $\LHD$.
A single odd reflection connects
normal orders
$\lhd$ and $\LHD$ which agree
except at
$i \in \{1,\dots,m\}$ and $j \in \{m+1,\dots,m+n\}$, with $i,j$
being consecutive in both orders.
Assuming that $i \lhd j$,
we have that
$L^{\lhd}(\tabA) \cong L^{\LHD}(\tabA')$
where
$\tabA'$ is obtained from $\tabA$ by adding 1 to its $i$th and $j$th entries if these entries are equal, or
$\tabA' := \tabA$ if these entries are different.
This was observed
originally by Serganova in her PhD thesis.

The following fundamental lemma
is well known, e.g.\ see \cite[Lemma 6.1]{CLW}.  The proof
involves noting that it suffices to consider the case when $\lhd$ and $\LHD$ are connected by a
single odd reflection, and then it can be observed from explicit
formulas for the characters.

\begin{Lemma}\label{vermasequal}
For any two normal orders $\lhd$ and $\LHD$
and $\tabA \in \Tab$,
the formal characters of $M^{\lhd}(\tabA)$ and $M^{\LHD}(\tabA)$ are equal.
Thus the symbols $[M^{\lhd}(\tabA)]$ and $[M^{\LHD}(\tabA)]$ are equal in
the Grothendieck group of $\mathcal O$.
\end{Lemma}

We will mostly work just with the {\em natural order} $<$
on $\{1,\dots,m+n\}$,  meaning of course that $1 < 2 < \cdots < m+n$;
for this order
we denote $\mathfrak{b}^{\!<}, \lambda_\tabA^{\!<}, \rho^{\!<}, M^{\!<}(\tabA)$
and $L^{\!<}(\tabA)$ simply by
$\mathfrak{b}, \lambda_\tabA, \rho, M(\tabA)$ and $L(\tabA)$. In particular
$\mathfrak{b}$ is the {\em standard Borel subalgebra}
of upper triangular matrices in $\mathfrak{g}$.
The resulting labelling of the irreducible objects of $\mathcal O$ is
the best choice for several other purposes.
For example,
the irreducible object $L(\tabA)$ is {\em finite-dimensional} if and only if $\tabA$ is {\em dominant};
hence, the irreducible objects $L(\tabA)$ for all dominant
tableaux $\tabA$ give a complete set of inequivalent finite-dimensional
irreducible $\mathfrak{g}$-supermodules.
This was established originally by
Kac in \cite{Kac2} by
an argument involving parabolic induction from $\g_{\0}$.
In a similar way, one sees
that $L(\tabA)$ is of {\em maximal Gelfand--Kirillov dimension}
amongst all supermodules in $\mathcal O$
if and only if $\tabA$ is {\em anti-dominant}.

The natural order on $\{1,\dots,m+n\}$
corresponds to the ordering of
the boxes of the pyramid $\pi$ induced by the lexicographic order of
coordinates $(\row(i), \col(i))$, i.e.\ $i < j$ if and only if
$\row(i) < \row(j)$, or $\row(i) = \row(j)$ and $\col(i) < \col(j)$.
There is another normal order which plays a significant role for us,
namely, the order $<'$
arising from the reverse lexicographic order on coordinates, i.e.\
 $i <' j$ if and only if $\col(i) < \col(j)$, or $\col(i) = \col(j)$ and
$\row(i) < \row(j)$. For this order, we denote $\mathfrak{b}^{\!<'},
\lambda_\tabA^{\!<'},
\rho^{\!<'},
M^{\!<'}(\tabA)$ and $L^{\!<'}(\tabA)$ by $\b',
\lambda_\tabA'$, $\rho'$, $M'(\tabA)$ and $L'(\tabA)$.
Note that in \cite{BBG}
the weight $\rho'$ was denoted $\widetilde\rho$.

Whereas the Borel subalgebra $\b$ arising from the natural ordering
has a unique odd simple root, the Borel subalgebra $\b'$ has a maximal
number of odd simple roots.
This leads to some significant differences when
working with the ordering $<'$
compared to the natural ordering. For instance,
it is not so easy to describe
the tableaux $\tabA$ such that $L'(\tabA)$ is either finite-dimensional or of
maximal Gelfand--Kirillov dimension in purely combinatorial terms.

Using the description of $\rho'$ given by (\ref{e:rho}) a direct calculation gives
\begin{equation} \label{e:espressoclub}
(\rho',\delta_{m+1}+\cdots+\delta_{m+n}) \equiv \lceil (n-m)/2 \rceil
+m\,s_-\pmod{2}.
\end{equation}
Hence, recalling (\ref{parity2}), we have that
$\parity(\lambda_\tabA') = \parity(b(\tabA))$ for $\tabA \in \Tab$.

\subsection{The Harish-Chandra homomorphism}\label{hchomsec}
Let $Z(\mathfrak{g})$ denote the center of $U(\mathfrak{g})$. This
can be understood via the Harish-Chandra homomorphism, which
gives an isomorphism between $Z(\mathfrak{g})$ and a certain subalgebra $I(\mathfrak{t})$ of $S(\mathfrak{t})$.
Let $x_i := e_{i,i}$ for $i=1,\dots,m$ and $y_j := -e_{m+j,m+j}$ for
$j=1,\dots,n$, so that $S(\mathfrak{t}) = \C[x_1,\dots,x_m,y_1,\dots,y_n]$.
The Weyl group of $\g$ with respect to $\t$
is the product of symmetric groups $S_m \times S_n$, which acts
naturally on $S(\t)$ so that $S_m$ permutes $x_1,\dots,x_m$ and $S_n$ permutes $y_1,\dots,y_n$.
Then
\begin{equation}\label{idef}
I(\mathfrak{t}) := \left\{f \in S(\t)^{S_m \times S_n}\:\bigg|\:
\begin{array}{l}
\frac{\partial f}{\partial x_i} + \frac{\partial f}{\partial y_j}
\equiv 0 \pmod{x_i-y_j}\\
\text{for any $1 \leq i \leq m$, $1 \leq j \leq n$}
\end{array}
\right\}.
\end{equation}
A distinguished set of generators for $I(\t)$ is given by the
{\em elementary supersymmetric polynomials}
\begin{equation}\label{supersymmetric}
e_r(x_1,\dots,x_m / y_1,\dots,y_n):= \sum_{s+t=r} (-1)^t e_s(x_1,\dots,x_m) h_t(y_1,\dots,y_n)
\end{equation}
for all $r \geq 1$, where $e_s(x_1,\dots,x_m)$ is the $s$th elementary symmetric polynomial
and $h_t(y_1,\dots,y_n)$ is the $t$th complete symmetric polynomial;
see e.g.\ \cite[\S0.6.1]{Se}.

To define the Harish-Chandra homomorphism
itself we fix a total order
$\lhd$ on $\{1,\dots,m+n\}$.
Recall that $\b^{\lhd}$ is the Borel subalgebra spanned by $\{e_{i,j}\}_{i \unlhd j}$; let $\n^{\lhd}$ be its nilradical.
Writing $U(\g)_0$ for the centralizer of $\t$ in $U(\g)$,
let $\phi^{\lhd}:U(\g)_0 \to S(\mathfrak{t})$ be the algebra homomorphism defined by the projection
along the direct sum decomposition
$U(\g)_0 = S(\t) \oplus (U(\g)_0 \cap U(\g) \n^\lhd)$.
Let
\begin{equation}\label{hchom}
\HC := S_{-\rho^{\lhd}} \circ \phi^{\lhd}:Z(\g) \to S(\t),
\end{equation}
where the shift automorphism $S_{-\rho^{\lhd}}$ is
the automorphism of $S(\t)$ defined by $x \mapsto x - \rho^{\lhd}(x)$ for each $x \in \t$.
Now we can state the key theorem here; see \cite[\S13.2]{Musson} for a recent exposition of the proof.

\begin{Theorem}[Kac, Sergeev]\label{hc}
The homomorphism $\HC$ is an isomorphism between $Z(\g)$
and $I(\t)$.
\end{Theorem}

Our definition of the Harish-Chandra homomorphism involves the choice of the total order $\lhd$. But in fact one obtains the same isomorphism
$Z(\g) \stackrel{\sim}{\to} I(\t)$ no matter which order is chosen:

\begin{Theorem}\label{indep}
The map $\HC:Z(\g) \to S(\t)$ does not depend on the particular choice of the total order $\lhd$ used in its definition.
\end{Theorem}

\begin{proof}
Suppose first that $\lhd$ and $\LHD$ are two orders that are conjugate under $S_m \times S_n$,
i.e.\ there exists a permutation $\sigma \in S_m \times S_n$ such that $$
i \lhd j \Leftrightarrow \sigma(i) \LHD \sigma(j),
$$
where
$S_m$ permutes $\{1,\dots,m\}$ and $S_n$ permutes $\{m+1,\dots,m+n\}$.
Let $\HC$ and $\HC'$ be  the Harish-Chandra homomorphisms defined via
$\lhd$ and $\LHD$, respectively. Take any $z \in Z(\g)$
and write it as $z = z_0 + z_1$ for
$z_0 \in S(\t)$, $z_1 \in U(\g)_0 \cap U(\g) \n^{\lhd}$.
Identifying the Weyl group $S_m \times S_n$ with permutation matrices in the
group $\mathrm{GL}_m(\C) \times \mathrm{GL}_n(\C)$ in the obvious way, we get an action
of $S_m \times S_n$ on $\g$ by conjugation. Since this is an action by inner automorphisms
it fixes $z$, so we have that $z = \sigma(z) =
\sigma(z_0) + \sigma(z_1)$ with $\sigma(z_0) \in S(\t)$
and $\sigma(z_1) \in U(\g)_0 \cap U(\g) \n^{\LHD}$.
Now compute:
$$
\HC'(z) = S_{-\rho^{\LHD}}(\sigma(z_0))
= S_{-\sigma(\rho^{\lhd})}(\sigma(z_0)) = \sigma(S_{-\rho^{\lhd}}(z_0))
=
\sigma(\HC(z)) = \HC(z),
$$
where the last equality follows because $\HC(z)$ is symmetric by
Theorem~\ref{hc}.

Since any order is $S_m \times S_n$-conjugate to a normal order,
we have thus reduced the problem to showing that the Harish-Chandra homomorphisms
arising from any two normal orders $\lhd$ and $\LHD$ are equal.
Again, let $\HC$ and $\HC'$ be the Harish-Chandra homomorphisms defined
from the orders $\lhd$ and $\LHD$, respectively, assuming now that both orders are normal.
For any $z \in Z(\g)$ and $\tabA \in \Tab$,
the element $z$ acts on
the Verma supermodule $M^{\lhd}(\tabA)$ (resp.\ $M^{\LHD}(\tabA)$)
by the scalar
$\HC(z)(\lambda_\tabA)$ (resp.\
$\HC'(z)(\lambda_\tabA)$).
So to prove that $\HC(z) = \HC'(z)$ it suffices to show that $M^{\lhd}(\tabA)$
and $M^{\LHD}(\tabA)$ have the same central character, which follows from
Lemma~\ref{vermasequal}.
\end{proof}

There is an explicit formula for
the elements $z_r \in Z(\g)$ lifting the elementary supersymmetric polynomials $e_r(x_1,\dots,x_m/y_1,\dots,y_n)$.
To formulate this, recall from \cite{GKLLRT} that
the {\em $kk$-quasideterminant} of a $k \times k$-matrix
$M$ is $$
|M|_{k,k} := d - c a^{-1} b,
$$
assuming
$M$ is decomposed into block matrices as
$M=\left(\begin{array}{l|l}
a&b\\\hline
c&d
\end{array}
\right)$
so that $a$ is an invertible
$(k-1) \times (k-1)$ matrix and $d$ is a scalar.
Working in the algebra $U(\g)[[u^{-1}]]$ where $u$ is an indeterminate,
let
$$
\zeta_k(u) :=
u|T_k(u)|_{k,k},
$$
where $T_k(u)$ is the $k \times k$ matrix with $ij$-entry
$\delta_{i,j} + (-1)^{|i|} u^{-1} e_{i,j}$.
The coefficients of these formal Laurent series
for $k=1,\dots,m+n$
generate a commutative subalgebra of $U(\g)$.
Then set
\begin{equation}\label{ze}
z(u) = \sum_{r \geq 0} z_r u^{-r} :=
\prod_{k=1}^m \zeta_k(u+1-k)
\bigg/
\prod_{k=1}^n \zeta_{m+k}(u+k-m).
\end{equation}
This defines
elements $z_1,z_2,z_3,\ldots\in U(\g)$.
For example for $\mathfrak{gl}_{1|1}(\C)$ one gets $z_1 = e_{1,1}+e_{2,2}$.

\begin{Theorem}\label{expcent}
The elements $\{z_r\}_{r \geq 1}$ generate the center $Z(\g)$.
Moreover,
$$
\HC(z_r) = e_r(x_1,\dots,x_m/ y_1,\dots,y_n).
$$
\end{Theorem}

\begin{proof}
Let $Y(\g)$ be the Yangian of $\g$ and
$b_{m|n}(u) \in Y(\g)[[u^{-1}]]$ be Nazarov's quantum Berezinian from \cite{Naz};
see also \cite[Definition 3.1]{Gow}.
In \cite[Theorem 1]{Gow}, Gow establishes a remarkable factorization of this
quantum Berezinian, from which
we see that $z(u)$ is the image of
$(u+1)^{-1} (u+2)^{-1} \cdots (u+n-m)^{-1} b_{m|n}(u)$ under the usual
evaluation homomorphism $Y(\g) \twoheadrightarrow U(\g)$.
The coefficients of $b_{m|n}(u)$ are central in $Y(\g)$ by \cite{Naz}
(or \cite[Theorem 2]{Gow}).
Hence, the coefficients $z_1,z_2,\dots$ of our $z(u)$ are central in $U(\g)$.

It remains to compute $\HC(z_r)$. For this we use the definition of $\HC$
coming from the natural order $<$, since for this order it is clear how to apply the projection $\phi = \phi^{\!<}$ to each of the power series $\zeta_i(u)$.
One gets that
$$
\HC(z(u)) =
\prod_{k=1}^m (1+u^{-1} x_k) \bigg / \prod_{k=1}^n (1+u^{-1} y_k).
$$
The $u^{-r}$-coefficient of this expression is equal to
$e_r(x_1,\dots,x_m / y_1,\dots,y_n)$.
\end{proof}

\begin{Remark}
In fact there exist other factorizations of $z(u)$ analogous to (\ref{ze})
which are adapted to more
general total orders $\lhd$.
To explain, let $\lhd$ be an arbitrary total order on $\{1,\dots,m+n\}$.
Let $\sigma \in S_{m+n}$ be the permutation such that $\sigma(1) \lhd \sigma(2) \lhd \cdots \lhd \sigma(m+n)$.
Let $$
\zeta^{\lhd}_{\sigma(k)}(u) := u |T^{\lhd}_k(u)|_{k,k},
$$
where $T_k^{\lhd}(u)$ is the $k \times k$ matrix with $ij$-entry
$\delta_{i,j} + (-1)^{|\sigma(i)|} u^{-1}e_{\sigma(i),\sigma(j)}$.
Then we have that
$$
z(u) =
\prod_{k=1}^m \zeta_k^{\lhd}(u+(\rho^{\lhd}, \delta_{k}))
\bigg /
\prod_{k=1}^n \zeta_k^{\lhd}(u+(\rho^{\lhd}, \delta_{k})).
$$
This (and some analogous factorizations of the quantum Berezinian in $Y(\g)$)
may be derived from \cite[Theorem 1]{Gow} by some explicit
commutations in the super Yangian; we omit the details.
\end{Remark}

\subsection{Projectives, prinjectives and blocks}
The following {\em linkage principle}
gives some rough
information about the composition multiplicities of the
Verma supermodules $M(\tabA)$.
This involves the Bruhat order $\preceq$ from \S\ref{scomb}.

\begin{Lemma}\label{linkage}
$[M(\tabA):L(\tabB)] \neq 0 \Rightarrow \tabB \preceq \tabA$.
\end{Lemma}

\begin{proof}
This is a consequence of the superalgebra analog of the Jantzen sum formula
from \cite[\S10.3]{Musson} or \cite{Gorelik};
see \cite[Lemma 2.5]{Bsurvey} for details.
\end{proof}

For $\tabA \in \Tab$ we denote the
projective cover of $L(\tabA)$ in $\mathcal O$
by $P(\tabA)$.
The supermodule $P(\tabA)$
has a Verma flag, that is, a finite filtration with sections
of the form $M(\tabB)$ for $\tabB \in \Tab$. Moreover, by {\em BGG reciprocity},
the multiplicity $(P(\tabA):M(\tabB))$ of $M(\tabB)$ in a Verma flag of $P(\tabA)$ is equal to the composition multiplicity $[M(\tabB):L(\tabA)]$; see e.g.\ \cite{Btilt}.
Combined with Lemma~\ref{linkage}, this implies that the category $\mathcal O$
is a {\em highest weight category} with weight poset
$(\Tab, \preceq)$.
Of course its standard objects are the Verma supermodules $\{M(\tabA)\}_{\tabA \in \Tab}$.

\begin{Remark}
In fact each choice of normal order $\lhd$ on $\{1,\dots,m+n\}$
gives rise to a {\em different}
structure of highest weight category on $\mathcal O$,
with standard objects being the corresponding Verma supermodules $M^{\lhd}(\tabA)$.
In this article we only need
the highest weight structure that comes from the natural order.
\end{Remark}

By a {\em prinjective object} we mean one that is both projective and injective.
The following lemma classifies the prinjective objects in $\mathcal O$, showing that they are the projective covers of the irreducible objects of maximal Gelfand--Kirillov dimension.

\begin{Lemma}\label{prinj}
Let $\tabA \in \Tab$.  Then
$P(\tabA)$ is prinjective if and only if $\tabA$ is anti-dominant.
In that case $P(\tabA)$ is both the projective cover and the injective hull of $L(\tabA)$.
\end{Lemma}

\begin{proof}
This is a consequence of \cite[Theorem 2.22]{BLW} (bearing in mind also \cite[Theorem 3.10]{BLW}); see
also \cite[Lemma~4.3, Remark~4.4]{Bsurvey}
and \cite[Corollary 6.2(ii)]{CS}.
\end{proof}

Recall finally that $\linked$ is the equivalence relation on $\Tab$ generated by the Bruhat order.
For a linkage class
$\xi \in \Tab /\!\linked$,
we let $\mathcal O_\xi$ be the Serre subcategory of $\mathcal O$
generated by $\{L(\tabA)\}_{\tabA \in \xi}$.
Lemma~\ref{linkage} implies immediately that this is a sum of blocks of $\mathcal O$. In fact each $\mathcal O_\xi$ is an indecomposable block,
thanks to \cite[Theorem 3.12]{CMW}.
Thus the blocks of $\mathcal O$ are in bijection with the linkage classes.

\begin{Lemma}\label{twist}
Let $S_m \times S_n$ act on $\Tab$
by permuting entries within rows.
Suppose we are given a linkage class $\xi \in \Tab /\!\linked$
and simple transposition $\sigma \in S_m \times S_n$
such that $\sigma(\xi) := \{\sigma(\tabA)\:|\:\tabA \in \xi\}$
is a different linkage class to $\xi$.
Then, there is an equivalence of categories
$T_\sigma: \mathcal O_\xi \to \mathcal O_{\sigma(\xi)}$
such that
$T_\sigma(M(\tabA)) \cong M(\sigma(\tabA))$
and
$T_\sigma(L(\tabA)) \cong L(\sigma(\tabA))$
for each $\tabA \in \xi$.
\end{Lemma}

\begin{proof}
This is a reformulation of \cite[Proposition 3.9]{CMW},
where the equivalence $T_\sigma$ is constructed explicitly as a certain twisting functor.
\end{proof}

\section{Principal \texorpdfstring{$W$}{W}-algebras and Whittaker functors}\label{swhittaker}

After reviewing some basic definitions and results from \cite{BBG},
we proceed to introduce the Whittaker coinvariants functor
$H_0$, which
takes representations of $\g$ to representations of its
{\em principal $W$-algebra}, i.e.\ the (finite) $W$-algebra
associated to a principal nilpotent orbit $e \in \g$.
We will mainly be concerned with the restriction of this functor to
the category $\O$.
The main result of the section shows that $H_0$
sends $M(\tabA)$ to the corresponding Verma supermodule $\overline{M}(\tabA)$
for $W$;
up to a parity shift, the latter was already introduced in \cite{BBG}.
This has several important
consequences: we use it
to determine the composition multiplicities of each $\overline{M}(\tabA)$,
to show that $H_0$ sends irreducibles in $\mathcal O$ to irreducibles or zero,
and to describe the center of $W$ explicitly.

\subsection{The principal \texorpdfstring{$W$}{W}-superalgebra}
We continue with $\g = \mathfrak{gl}_{m|n}(\C)$ as in the previous section.
Consider the principal nilpotent element
\begin{equation*}
e \coloneqq
e_{1,2}+e_{2,3}+\cdots+e_{m-1,m} +
e_{m+1,m+2}+e_{m+2,m+3}+\cdots+e_{m+n-1,m+n}
\in \g.
\end{equation*}
Define a
good
grading $\g = \bigoplus_{r \in \Z} \g(r)$ for $e \in \g(1)$
by declaring that each matrix unit $e_{i,j}$ is of degree
\begin{equation}\label{e:dege}
\deg(e_{i,j}) \coloneqq \col(j)- \col(i).
\end{equation}
Set
$$
\p \coloneqq \bigoplus_{r \ge 0} \g(r), \qquad \h \coloneqq \g(0),\qquad
\m \coloneqq\bigoplus_{r < 0} \g(r).
$$
Let $\chi \in \g^*$ be defined by $\chi(x) \coloneqq (x,e)$.
The restriction of $\chi$ to $\m$ is a character of $\m$.
Then define
$\m_\chi \coloneqq \{x-\chi(x) \:|\: x \in \m\}$,
which a shifted copy of $\m$ inside $U(\m)$.
The {\em principal $W$-superalgebra}
may then be defined as
\begin{equation}  \label{e:wdef}
W \coloneqq \{u \in U(\p) \:|\:  u\, \m_\chi \subseteq \m_\chi U(\g) \},
\end{equation}
which is a subalgebra of $U(\p)$.
Although this definition depends implicitly on the choice of pyramid $\pi$,
the isomorphism type of $W$ depends only on $m$ and $n$ not $\pi$, see \cite[Remark 4.8]{BBG}.
The following theorem shows that $W$ is isomorphic
to a truncated
shifted version of the Yangian $Y(\mathfrak{gl}_{1|1})$.

\begin{Theorem}[{\cite[Theorem 4.5]{BBG}}]\label{wpres}
The superalgebra $W$ contains distinguished even elements
$\{d_1^{(r)}, d_2^{(r)}\}_{r \geq 0}$
and odd elements $\{e^{(r)}\}_{r > s_+} \cup \{f^{(r)}\}_{r > s_-}$.
These elements generate $W$ subject only to the following relations:
\begin{align*}
d_i^{(0)} &= 1,
& d_1^{(r)} &= 0 \text{ for $r > m$},\\
[d_i^{(r)}, d_j^{(s)}] &= 0,
&[e^{(r)}, f^{(s)}] &= \sum_{a=0}^{r+s-1} \tilde d_1^{(a)} d_2^{(r+s-1-a)},\\
[e^{(r)}, e^{(s)}] &= 0,
&[d_i^{(r)}, e^{(s)}] &=
\sum_{a=0}^{r-1} d_i^{(a)} e^{(r+s-1-a)},\\
[f^{(r)}, f^{(s)}] &= 0,
&[d_i^{(r)}, f^{(s)}] &=-
\sum_{a=0}^{r-1} f^{(r+s-1-a)} d_i^{(a)},
\end{align*}
where $\tilde d_i^{(r)}$ is
defined recursively from
$\sum_{a=0}^r \tilde d_i^{(a)} d_i^{(r-a)} = \delta_{r,0}$.
\end{Theorem}

We will occasionally need to appeal to
the
explicit formulae\footnote{There is a typo in \cite[(4.12)--(4.13)]{BBG}; both of these formulae need an extra minus sign. Similarly the formulae for $\tilde{d}_i^{(r)}$ in \cite[(4.19)--(4.20)]{BBG} need to be changed by a sign.}\label{footnotepage}
 for the
generators $d_i^{(r)}, e^{(s)}$ and $f^{(s)}$
from \cite[\S4]{BBG}.
In particular, these formulae show that $d_1^{(1)}$ and $d_2^{(1)}$ are
the unique elements of $S(\t)=\C[\t^*]$ such that
\begin{align}
d_1^{(1)}(\lambda) &= (\lambda+\rho',\delta_1+\cdots+\delta_m),\label{nobetter}\\
d_2^{(1)}(\lambda) &= (\lambda+\rho',\delta_{m+1}+\cdots+\delta_{m+n}).\label{better2}
\end{align}
 It is also often useful to work with the generating functions
$$
d_i(u) := \sum_{r \geq 0} d_i^{(r)} u^{-r} \in W[[u^{-1}]],
\qquad
\tilde{d}_i(u) := \sum_{r \geq 0} \tilde{d}_i^{(r)} u^{-r} \in W[[u^{-1}]],
$$
so that $\tilde{d}_i(u) = d_i(u)^{-1}$.
Using these, we can define more elements $\{c^{(r)}, \tilde
c^{(r)}\}_{r \geq 0}$
by setting
\begin{equation}\label{ctilde}
c(u) =
\sum_{r \geq 0} c^{(r)} u^{-r} := \tilde d_1(u) d_2(u),
\qquad
\tilde c(u) =
\sum_{r \geq 0} \tilde c^{(r)} u^{-r} := d_1(u) \tilde d_2(u).
\end{equation}
In particular, by the defining relations, we have that $c^{(r+s-1)} =
[e^{(r)}, f^{(s)}]$ for $r > s_+, s > s_-$.
The elements $\{c^{(r)}\}_{r \geq 1}$
are known to belong to the center $Z(W)$; see \cite[Remark 2.3]{BBG}.
Hence, so too do the elements $\{\tilde c^{(r)}\}_{r \geq 1}$.
We will show in Corollary~\ref{bib} below that either of these
families of elements give generators for $Z(W)$.

Recall finally by \cite[Theorem 6.1]{BBG} that $W$
has a triangular decomposition:
let $W^0, W^+$ and $W^-$ be the subalgebras generated by
$\{d_1^{(r)},d_2^{(s)}\}_{1 \leq r \leq m,1 \leq s \leq n}$,
$\{e^{(r)}\}_{s_+ < r \leq s_++m}$
and $\{f^{(r)}\}_{s_- < r \leq s_-+m}$, respectively;
then the multiplication map $W^- \otimes W^0 \otimes W^+ \to W$ is a vector space isomorphism.
Moreover, by the PBW theorem for $W$, the subalgebra
$W^0$ is a free polynomial algebra of rank $m+n$, while $W^+$ and $W^-$
are Grassmann algebras of dimension $2^{m}$.

\subsection{Highest weight theory for \texorpdfstring{$W$}{W}}\label{wo}
Next, we review some results about the representation theory of $W$
established in \cite{BBG}.
The triangular decomposition allows us to define Verma supermodules for $W$ as follows.
Let $W^\sharp := W^0 W^+$.
This is a subalgebra of $W$, and there is a surjective homomorphism $W^\sharp \to W^0$ which is the identity on $W^0$ and zero on each $e^{(r)} \in W^+$.

Given $\tabA = \substack{a_1 \cdots a_m \\ b_1 \cdots b_n} \in \Tab$,
let $\C_\tabA$ be the
one-dimensional $W^0$-supermodule spanned by a vector $\overline 1_\tabA$
of parity $\parity(b(\tabA))$,
such that
\begin{equation}\label{hwa}
d_1^{(r)} \overline 1_\tabA = e_r(a_1,\dots,a_m) \overline 1_\tabA,
\qquad
d_2^{(s)} \overline 1_\tabA = e_s(b_1,\dots,b_n) \overline 1_\tabA
\end{equation}
for $1 \leq r \leq m, 1 \leq s \leq n$.
View $\C_\tabA$ as a
$W^\sharp$-supermodule via the surjection
$W^\sharp \twoheadrightarrow W^0$.
Then induce to form the {\em Verma supermodule}
\begin{equation} \label{e:WVerma}
\overline{M}(\tabA) \coloneqq W \otimes_{W^\sharp} \C_\tabA,
\end{equation}
setting $\overline{m}_\tabA := 1 \otimes \overline 1_\tabA$.
Of course, $\overline{M}(\tabA)$
only depends on the row equivalence class of $\tabA$.
The PBW theorem for $W$ implies that
$\dim \overline{M}(\tabA) = 2^m$.

We say that $M \in W\lsmod$ is
a {\em highest weight supermodule} of {\em highest weight}
$\tabA  = \substack{a_1 \cdots a_m \\ b_1 \cdots b_n} \in \Tab$
if there exists a homogeneous vector $v \in M$ that generates $M$ as a $W$-supermodule with
$e^{(r)} v = 0$ for $r > s_+$, $d_1^{(r)} v = e_r(a_1,\dots,a_m)
v$ for $1 \leq r \leq m$,
and $d_2^{(s)} v = e_r(b_1,\dots,b_n) v$ for $1
\leq s \leq n$.  The Verma supermodule $\overline M(\tabA)$ is the
{\em universal highest weight supermodule} of highest weight $\tabA$: given any highest weight supermodule
$M$ of highest weight $\tabA$ as above, there exists a unique
surjective homomorphism from either $\overline{M}(\tabA)$ or $\Pi
\overline{M}(\tabA)$ onto $M$ such that $\overline{m}_\tabA \mapsto
v$;
the homomorphism is from $\overline{M}(\tabA)$ if and only if
$|v|=\parity(b(A))$.

By \cite[Lemma 7.1]{BBG}, each $\overline{M}(\tabA)$
has a unique irreducible quotient
$\overline{L}(\tabA)$.

\begin{Theorem}[{\cite[Theorem 7.2]{BBG}}]\label{bbbg}
The supermodules $\{\overline{L}(\tabA)\}_{\tabA \in \Tab}$
give all of the
irreducible $W$-supermodules (up to isomorphism and parity switch).
Moreover, $\overline{L}(\tabA) \cong \overline{L}(\tabB)$ if and only if $\tabA \roweq \tabB$.
\end{Theorem}

In particular, the theorem shows that all irreducible
$W$-supermodules are finite-dimensional. Henceforth, we will restrict our
attention to the full subcategory $W\lsmof$ of $W\lsmod$ consisting of
finite-dimensional supermodules.

There is actually a very simple way to realize $\overline{L}(\tabA)$ explicitly.
Recall that
\begin{equation}\label{hstruct}
\mathfrak{h} \cong \mathfrak{gl}_1(\C)^{\oplus s_-} \oplus
\mathfrak{gl}_{1|1}(\C)^{\oplus m} \oplus \mathfrak{gl}_1(\C)^{\oplus
  s_+}.
\end{equation}
For any $\tabA
 = \substack{a_1 \cdots a_{m} \\ b_1 \cdots b_{n}} \in \Tab$, let $K(\tabA)$ be
the
$\h$-supermodule induced from a one-dimensional $\b'\cap\mathfrak{h}$-supermodule
of weight $\lambda'_\tabA$ and parity $\parity(\lambda'_\tabA) = \parity(b(\tabA))$, cf.\ (\ref{e:espressoclub});
we use the letter $K$ here because it is a Kac supermodule for $\mathfrak{h}$ (as well as being a Verma supermodule).
Note that $\dim K(\tabA) = 2^m$.
We denote the highest weight vector in $K(\tabA)$ by $k_\tabA$.
Observe that
\begin{equation}\label{mk}
M'(\tabA) \cong U(\mathfrak{g}) \otimes_{U(\mathfrak{p})} K(\tabA).
\end{equation}
Also let $V(\tabA)$ be the unique irreducible quotient of
$K(\tabA)$. Thus $V(\tabA)$ is an irreducible $\mathfrak{h}$-supermodule
of $\b'\cap \mathfrak h$-highest weight $\lambda_\tabA'$,
and $\dim V(\tabA) = 2^{m - \mat(\tabA)}$.
Finally,
using the (injective!) homomorphism $W \hookrightarrow U(\p) \twoheadrightarrow U(\h)$
derived from the natural inclusion and projection maps,
we restrict these supermodules to $W$ to obtain
\begin{equation}\label{craven}
\overline{K}(\tabA) :=  K(\tabA) \downarrow^{U(\h)}_{W},
\qquad
\overline{V}(\tabA) := V(\tabA) \downarrow^{U(\h)}_{W}.
\end{equation}
We sometimes denote $k_\tabA \in \overline{K}(\tabA)$ instead by $\overline k_\tabA$.

\begin{Theorem}
\label{irco}
If $\mat(\tabA) = \atyp(\tabA)$
then $\overline{L}(\tabA) \cong \overline{V}(\tabA)$.
\end{Theorem}

\begin{proof}
This is essentially \cite[Theorem 8.4]{BBG}, but we should note also
that the isomorphism constructed is necessarily even since
it sends $\overline m_\tabA$ to $\overline k_\tabA$,
which are both of the same parity $\parity(b(\tabA))$.
\end{proof}

\begin{Lemma}\label{music}
For any $\tabA \in \Tab$, we have that
$[\overline{K}(\tabA)] = \sum_{\tabB \upit \tabA} [\overline{V}(\tabB)]$,
equality written in the Grothendieck group $K_0(W\lsmof)$.
\end{Lemma}

\begin{proof}
By
the representation theory of $\mathfrak{gl}_{1|1}(\C)$,
we have that
$$
[K(\tabA)] = \sum_{\tabB \upit \tabA} [V(\tabB)].
$$
The lemma follows from this on restricting to $W$.
\end{proof}

\subsection{Invariants and coinvariants}\label{invandcov}
Given a right $\g$-supermodule $M$,
it is easy to check from \eqref{e:wdef} that the subspace
\begin{equation} \label{e:whittinv}
H^0(M) := H^0(\m_\chi,M) = \{v \in M \:|\: v \, \m_\chi = 0\}
\end{equation}
is stable under right multiplication by elements of $W$.
Hence, we obtain the {\em Whittaker invariants functor}
\begin{equation}
H^0:\rsmod{U(\g)}
\to \rsmod{W}.
\end{equation}
Let $\rsmodchi U(\g)$ be the full subcategory of $\rsmod U(\g)$
consisting of all the
supermodules on which $\m_\chi$ acts locally nilpotently.
The super analog of {\em Skryabin's theorem} asserts that
the restriction of $H^0$ defines an equivalence of categories
from $\rsmodchi U(\g)$ to $\rsmod W$.
Let $Q$ denote the
$(W,U(\g))$-superbimodule
\begin{equation}
Q :=
U(\g)/\m_\chi U(\g),
\end{equation}
denoting the canonical
image of $1 \in U(\g)$ in $Q$ by $1_\chi$.
Then the functor
\begin{equation}\label{gravy}
- \otimes_W Q : \rsmod W \to \rsmodchi U(\g)
\end{equation}
is the inverse functor to $H^0$ in Skrabin's theorem.
As observed already in \cite[Remarks 3.9--3.10]{Zhao},
Skryabin's proof of this result in the purely even setting
from \cite{Skry} extends routinely to the
super case.
Along the way, one sees that $Q$ is a free left $W$-supermodule with an
explicitly constructed basis, from which we see that there exists
a $W$-supermodule
homomorphism
\begin{equation}\label{prmap}
\prr:Q \twoheadrightarrow W
\end{equation}
such that $\prr(1_\chi) = 1$. We fix such a choice for later use.

Instead, suppose that $M$ is a left $\g$-supermodule.
Then again it is clear from \eqref{e:wdef} that the left action
of $W$ leaves the subspace $\m_\chi M$ invariant, hence, we get induced a well-defined left action of $W$ on
\begin{equation} \label{e:whittcoinv}
H_0(M) := H_0(\m_\chi,M) =M/\m_\chi M.
\end{equation}
This gives us the {\em Whittaker coinvariants functor}
\begin{equation}
H_0:U(\g)\lsmod  \to W\lsmod.
\end{equation}
Equivalently, this is the functor $Q \otimes_W-$.

The first lemma below connects Whittaker invariants
and coinvariants. To formulate it we need some duals:
if $M$ is a left supermodule over some superalgebra
then we write $M^*$ for the full linear dual
of $M$ considered as a right supermodule
with the obvious action $(fv)(a) = f(va)$ (no signs!). Similarly,
we write ${^*}M$ for the dual of a right
supermodule, which is a left supermodule.
There are natural
supermodule homomorphisms $M \rightarrow ({^*}M)^*$ and
$M \rightarrow {^*}(M^*)$ (which involve a sign!).
Note also that if $V$ is a finite-dimensional superspace and $M$ is
arbitrary then
the canonical maps
\begin{equation}\label{dualitymap}
M^* \otimes V^* \stackrel{\sim}{\rightarrow} (V \otimes M)^*,
\qquad
{^*}M \otimes {^*}V \stackrel{\sim}{\rightarrow} {^*}(V \otimes M)
\end{equation}
are isomorphisms.

\begin{Lemma} \label{L:equiv}
Let $M$ be a left $\g$-supermodule.  Then there is a functorial isomorphism
$H_0(M)^* \cong H^0(M^*)$.
In particular, if $H_0(M)$ is finite-dimensional, then
$H_0(M) \cong {^*}H^0(M^*)$.
\end{Lemma}

\begin{proof}
The natural inclusion $H_0(M)^* \hookrightarrow M^*$ induced by
$M \twoheadrightarrow H_0(M)$
has image contained in
$H^0(M^*)$. This gives a $W$-supermodule homomorphism
$H_0(M)^* \hookrightarrow H^0(M^*)$.
To see that it is surjective,
we observe that any element
of $H^0(M^*) \subseteq M^*$ annihilates $\m_\chi M$, hence,
comes from an element of $H_0(M)^*$.
\end{proof}

The next lemma is an analog of another well-known result in the
even setting.

\begin{Lemma}\label{exactness}
The functor $H_0$ sends short exact sequences of left $\mathfrak{g}$-supermodules that are finitely generated over $\m$ to short exact sequences of finite-dimensional left $W$-supermodules.
\end{Lemma}

\begin{proof}
For any left $\m$-supermodule $M$, we introduce its $\chi$-restricted dual
$$
M^\# := \{f \in M^*\:|\:f(\m_\chi^r M) = 0\text{ for }r \gg 0\}.
$$
This defines a functor $(-)^\#: U(\m)\lsmod \to \rsmod U(\m)$.
We claim that this functor is exact.
To see this, we note as in the proof of \cite[Lemma 3.10]{Moeglin} that
the functor $(-)^\#$ is isomorphic to
$\Hom_{\m}(-, E_\chi)$,
where $E_\chi :=U(\m)^\#$ viewed  as an $(\m,\m)$-superbimodule in the obvious way.
The proof of \cite[Assertion 2]{Skry} shows that
$E_\chi$ is injective as a left
$\m$-supermodule; this follows ultimately from the non-commutative Artin-Rees lemma.
The desired exactness follows.

If $M$ is a left $\g$-supermodule then $M^\#$ is actually a $\g$-submodule of $M^*$,
and this submodule
belongs to $\rsmodchi U(\g)$.
Hence, $(-)^\#$ can also be viewed as an exact functor
$U(\g)\lsmod\to\rsmodchi U(\g)$.
As in \cite[Lemma~3.11]{Moeglin}, we have
quite obviously for any left $\g$-supermodule that
$H^0(M^*) = H^0(M^\#)$ as subspaces of $M^*$.
Since $H^0$ is exact on $\rsmodchi U(\g)$ by Skryabin's theorem, we have now proved that the functor
$U(\g)\lsmod \to \rsmod W$ given by $ M \mapsto H^0(M^*)$ is exact.
Finally,
if $M$ is a left $\g$-supermodule that is finitely generated over $\m$, then it is clear that $H_0(M)$ is finite-dimensional,
so that
$H_0(M) \cong {^*}H^0(M^*)$ by Lemma~\ref{L:equiv}.
The lemma follows.
\end{proof}

\begin{Corollary}
The restriction of the functor
$H_0$
to the category $\O$ from \S\ref{so}
is {exact} and has image contained in $W\lsmof$.
\end{Corollary}

\begin{proof}
In view of the lemma, it just remains to observe that all supermodules in $\O$ are finitely
generated over $\mathfrak{m}$.
This follows because the Verma supermodules $M(\tabA)$ are finitely
generated over $\mathfrak{m}$, which is easily seen from the
definition since
$\mathfrak{g}_\0 = \mathfrak{b}_\0\oplus \mathfrak{m}_\0$.
\end{proof}

We also need the following lemma which takes care of all the necessary bookkeeping regarding
parities.

\begin{Lemma}\label{bookkeeping}
For $M \in \O$, the elements $d_1^{(1)},  d_2^{(1)}\in W^0$
act semisimply on
$H_0(M)$. Moreover, the $z$-eigenspace of $d_2^{(1)}$
is concentrated in parity $\parity(z)$
for each $z\in \C$.
\end{Lemma}

\begin{proof}
For $M \in \O$ and any vector $v \in M$ of $\t$-weight
$\lambda$, we know from (\ref{nobetter})--(\ref{better2})
that $d_1^{(1)}$ and $d_2^{(1)}$ act on $v$ (hence, $v + \mathfrak{m}_\chi M$)
by the scalars
$(\lambda+\rho', \delta_{1}+\dots+\delta_{n})$
and
$(\lambda+\rho', \delta_{m+1}+\dots+\delta_{m+n})$, respectively.
Hence, they both act semisimply on all of $H_0(M)$.
For the last part, let $z := (\lambda+\rho', \delta_{m+1}+\dots+\delta_{m+n})$.
Then, by (\ref{parity2}) and (\ref{e:espressoclub}),
the vector $v$ is of parity $\parity(\lambda)=\parity(z)$.
\end{proof}

\subsection{On tensoring with finite-dimensional
  representations}\label{pf}

Let $\Rep(\g)$ be the symmetric monoidal category of
{\em rational representations} of $\g$, that is,
finite-dimensional left
$\g$-supermodules
which are semisimple over $\t$ with weights lying in $\t_\Z^* := \bigoplus_{i=1}^{m+n} \Z\delta_i$.
Tensoring with
$V \in \Rep(\g)$
defines a {\em projective functor}
$$
V \otimes - :U(\g)\lsmod
\rightarrow U(\g)\lsmod.
$$
This is a rigid object of the strict monoidal category $\End(U(\g)\lsmod)$
of $\C$-linear
endofunctors of $U(\g)\lsmod$:
it has a biadjoint defined by tensoring with
the usual dual
$V^\vee$ of $V$ in the category
$\Rep(\g)$.
In this subsection, we introduce an analogous biadjoint pair of endofunctors
$V \ostar -$ and $V^\vee \ostar -$ of $W\lsmof$.  We will use the language
of module categories over monoidal categories,
see e.g.\ \cite[Chapter 7]{EGNO}.

It is convenient to start by working with right supermodules.
From the previous subsection, we recall the notation $M^*$ and ${^*}M$ for duals of left (resp. right)
supermodules, which are right (resp. left) supermodules.
In particular, for $V$ as in the previous paragraph,
$V^*$ is a right $U(\g)$-supermodule.
Tensoring with it gives us an exact functor
$- \otimes V^*:\rsmod U(\g) \rightarrow \rsmod U(\g)$.
In fact, writing $\End(\rsmod U(\g))$ for the strict monoidal category
of all $\C$-linear
endofunctors of $\rsmod U(\g)$, this defines a monoidal functor
$$
\Rep(\g)^{\operatorname{op}} \rightarrow \End(\rsmod U(\g)),
\quad
V \mapsto - \otimes V^*.
$$
In other words, $\rsmod U(\g)$ is a right module category over the
monoidal category $\Rep(\g)$.
The main coherence map that is needed for this comes from the natural isomorphisms
$(- \otimes W^*) \circ (- \otimes V^*) \cong - \otimes (V^*
\otimes W^*) \cong - \otimes ((W \otimes V)^*)$.

It is clear that $- \otimes V^*$ takes objects of $\rsmodchi U(\g)$
to objects of $\rsmodchi U(\g)$.
Hence, we can consider
$- \otimes V^*$ also as an endofunctor of  $\rsmodchi U(\g).$
Transporting this through Skryabin's equivalence from (\ref{gravy}),
we obtain an exact functor
$$
- \ostar V^*:=
H^0((- \otimes_W Q) \otimes V^*):
\rsmod W \rightarrow \rsmod W
$$
Like in the previous paragraph, this actually defines a monoidal functor
$$
\Rep(\g)^{\operatorname{op}} \rightarrow \End(\rsmod W),
$$
making $\rsmod W$ into a right module category over $\Rep(\g)$.
To construct the coherence map
$(- \ostar W^*) \circ (- \ostar V^*) \cong
-\ostar ((W \otimes V)^*)$ for this, one needs to
use the canonical adjunction between $- \otimes_W Q$
and $H^0$.
Perhaps the most important fact about this functor is that
there is a  isomorphism of vector superspaces
\begin{equation}\label{most}
M \ostar V^* \stackrel{\sim}{\rightarrow}
M \otimes V^*
\end{equation}
which is natural in both $M$ and $V$.
In particular, $-\ostar V^*$ takes finite-dimensional $W$-supermodules
to finite-dimensional $W$-supermodules.
By definition, the isomorphism (\ref{most}) is defined by the restriction of the map
\begin{align*}
(M \otimes_W Q) \otimes V^* &\rightarrow M \otimes V^*, \\
(m \otimes 1_\chi u) \otimes f &\mapsto m \,\prr(u) \otimes f
\end{align*}
where $\prr$ is the map from (\ref{prmap}). The proof of this assertion
goes back to the PhD thesis of Lynch. For this and
other details about this construction, we refer to
\cite[\S8.2]{BKrep}; the super case is essentially the same.

\begin{Lemma}\label{yu}
There is an isomorphism
$H^0(M) \ostar V^*
\cong H^0(M \otimes V^*)$ which is natural in $M$
and $V$. It makes $H^0:\rsmod U(\g) \rightarrow \rsmod W$
into a morphism of right $\Rep(\g)$-module categories.
\end{Lemma}

\begin{proof}
We start from the canonical isomorphism
$M \cong H^0(M) \otimes_W Q$ defined by the canonical adjunction from
Skryabin's theorem.
Then apply $H^0 \circ (- \otimes V^*)$ to both sides.
\end{proof}

We are ready to switch the discussion to left supermodules.
For $V \in \Rep(\g)$ as before and $M \in W\lsmof$,
we define
\begin{equation}
V \ostar M :=
{^*} (M^* \ostar V^*),
\end{equation}
noting that $M^* \ostar V^*$ is also finite-dimensional thanks to
(\ref{most}).
Again, we have that $(W \ostar -) \circ (V \ostar -)
\cong (W \otimes V) \ostar -$, so that we obtain
a monoidal functor
\begin{equation}\label{itsmonoidal}
\Rep(\g) \rightarrow \End(W\lsmof),
\qquad
V \mapsto V \ostar -
\end{equation}
making $W\lsmof$ into a (left) module category over $\Rep(\g)$.
Also, applying ${^*}(-)$ to (\ref{most}) with $M$ replaced by $M^*$
then using (\ref{dualitymap}), we get a canonical isomorphism
\begin{equation}
V \otimes M
\cong
{^*}(V^*) \otimes {^*}(M^*)
\cong {^*} (M^* \otimes V^*)
\stackrel{\sim}{\rightarrow} {^*}(M^* \circledast V^*) = V \ostar M
\end{equation}
as vector superspaces.

In general, due to the parity condition prescribed by (\ref{parity2}),
the
endofunctor $V \otimes -$ does not leave $\O$
invariant.
However, it does providing the $\lambda$-weight space of $V$ is concentrated in parity
$\parity((\lambda,\delta_{m+1}+\cdots+\delta_{m+n}))$
for all $\lambda \in \t_\Z^*$.
Let $\Rep_0(\g)$ be the full monoidal subcategory of
$\Rep(\g)$
consisting of all such $V$.
Then, for $V \in \Rep_0(\g)$, we do get a monoidal functor
\begin{equation}
\Rep_0(\g) \rightarrow \End(\mathcal O),
\qquad
V \mapsto V \otimes -.
\end{equation}
So $\mathcal O$ is a module category over $\Rep_0(\g)$.

\begin{Theorem}\label{crappy}
There is a natural isomorphism
$H_0(V \otimes M) \cong V \ostar H_0(M)$
making $H_0:\O \rightarrow W \lsmof$ into a morphism of $\Rep_0(\g)$-module categories.
\end{Theorem}

\begin{proof}
Take $M \in W\lsmof$. Since $H_0(V \otimes M)$ and $V$ are finite-dimensional, we have from
Lemmas~\ref{L:equiv} and \ref{yu} that
\begin{align*}
H_0(V \otimes M)
&\cong {^*}H^0((V \otimes M)^*)
\cong {^*} H^0(M^* \otimes V^*)\\
&\cong {^*} (H^0(M^*) \ostar V^*)
=V \ostar H_0(M).
\end{align*}
Everything else is purely formal; see \cite[\S8.4]{BKrep} for
further discussion.
\end{proof}

Since $V^\vee$ is both a left dual and right dual to $V \in
\Rep_0(\g)$, it is automatic that $V^\vee \ostar -$ is both left and
right adjoint to $V \ostar -$. Moreover, the monoidal isomorphism
described in Theorem~\ref{crappy} intertwines the canonical adjunctions
between
$V \otimes -$ and $V^\vee \otimes -$ with
the ones between
$V \ostar -$ and $V^\vee \ostar -$.

\subsection{Whittaker coinvariants of \texorpdfstring{$M'(\tabA)$}{M'(\tabA)}}
The following theorem
will allow us to determine
the effect of the
Whittaker coinvariants functor on the Verma supermodule $M'(\tabA)$.

\begin{Theorem}\label{hrest}
The map $U(\p) \rightarrow Q, u \mapsto 1_\chi u$ is an isomorphism of
$(W, U(\p))$-superbimodules.
Hence, for any left $\mathfrak{p}$-supermodule $M$,
there is an isomorphism
$$
M\!\downarrow^{U(\p)}_W \,\stackrel{\sim}{\to} H_0(U(\g) \otimes_{U(\p)} M),
 \qquad
v \mapsto 1 \otimes v + \mathfrak{m}_\chi(U(\g) \otimes_{U(\p)} M)
$$
\end{Theorem}

\begin{proof}
The first assertion is immediate from the PBW theorem and the
definition of $Q$. Hence,
$$
H_0(U(\g)\otimes_{U(\p)} M)
\cong Q \otimes_{U(\g)}
U(\g) \otimes_{U(\p)} M
\cong Q \otimes_{U(\p)} M \cong U(\p)\otimes_{U(\p)} M \cong M,
$$
which translates into the given isomorphism.
\end{proof}

Recall the definition of the $W$-supermodule $\overline{K}(\tabA)$ from \S\ref{wo}.

\begin{Corollary}\label{hrestc}
For any $\tabA \in \Tab$, we have that $H_0 (M'(\tabA)) \cong
\overline{K}(\tabA)$.
\end{Corollary}

\begin{proof}
Apply Theorem~\ref{hrest} to the $\mathfrak{p}$-supermodule
obtained by inflating $K(\tabA)$ through
$\mathfrak{p} \twoheadrightarrow \mathfrak{h}$
and use (\ref{mk}).
\end{proof}

\begin{Corollary}\label{hrestd}
For any $\tabA \in \Tab$, there exists a supermodule $M \in \mathcal O$
such that $H_0(M)\cong\overline{L}(\tabA)$.
\end{Corollary}

\begin{proof}
Since $\overline{L}(\tabA)$ only depends on the row equivalence class of $\tabA$,
we may assume that $\atyp(\tabA) = \mat(\tabA)$.
Then apply Theorem~\ref{hrest} to the $\mathfrak{p}$-supermodule
obtained by inflating $V(\tabA)$ through
$\mathfrak{p} \twoheadrightarrow \mathfrak{h}$
and use Theorem~\ref{irco}.
\end{proof}

\subsection{Whittaker coinvariants of \texorpdfstring{$M(\tabA)$}{M(\tabA)}}
We regard the following theorem as one of the central results of this article.

\begin{Theorem}\label{T:main}
For any $\tabA \in \Tab$, we have that $H_0(M(\tabA)) \cong \overline{M}(\tabA)$.
\end{Theorem}

\begin{proof}
See Appendix~\ref{appendix}.
\end{proof}

\begin{Corollary}\label{mek}
For any $\tabA \in \Tab$, we have that
$[\overline{M}(\tabA)] = [\overline{K}(\tabA)]$
in the Grothendieck group $K_0(W\lsmof)$.
\end{Corollary}

\begin{proof}
By Lemmas~\ref{vermasequal} and \ref{exactness}
we have that $[H_0 (M(\tabA))] = [H_0 (M'(\tabA))]$.
Now apply Theorem~\ref{T:main} and Corollary~\ref{hrestc}.
\end{proof}

\begin{Corollary}\label{vermascomp}
Suppose $\tabA \in \Tab$ is chosen so that $\mat(\tabA) = \atyp(\tabA)$.
Then
$$
[\overline{M}(\tabA)]
= \sum_{\tabB \upit \tabA} [\overline{L}(\tabB)].
$$
\end{Corollary}

\begin{proof}
This is immediate from Corollary~\ref{mek},
Lemma~\ref{music} and Theorem~\ref{irco}.
\end{proof}

\subsection{Whittaker coinvariants of \texorpdfstring{$L(\tabA)$}{L(\tabA)}}
Next we describe the effect of $H_0$ on the irreducible objects of
$\mathcal O$.

\begin{Theorem} \label{wirr}
Let $\tabA \in \Tab$.  Then
$$
H_0(L(\tabA)) \cong
\begin{cases}
\overline{L}(\tabA) &\text{if $\tabA$ is anti-dominant}, \\
      0 &\text{otherwise.}
  \end{cases}
$$
\end{Theorem}

\begin{proof}
We first show that $H_0(L(\tabA)) = 0$ if $\tabA$ is not anti-dominant.
Let $\p'$ be the parabolic subalgebra of $\g$ spanned by
$\{e_{i,j}\}_{\row(i) \leq \row(j)}$, i.e.\ $\p'  = \g_{\0} + \b$.
Let $M_{ev}(\tabA) :=
U(\g_\0) \otimes_{U(\b_\0)} \C_{\lambda_\tabA}$
be the Verma module
for $\g_{\0}$ of $\mathfrak{b}_{\0}$-highest weight $\lambda_\tabA$.
We view it also as a $\p'$-supermodule concentrated in parity $\parity(\lambda_\tabA)$ via the natural projection
$\p' \twoheadrightarrow \g_\0$.
Then we have obviously that
$$
M(\tabA) \cong U(\g) \otimes_{U(\p')} M_{ev}(\tabA).
$$
Assume that $\tabA$ is not anti-dominant, and let $\tabB$ be the unique anti-dominant
tableau such that $\tabA \roweq \tabB \linked \tabA$.
By classical theory, the Verma module $M_{ev}(\tabB)$ embeds into $M_{ev}(\tabA)$.
Hence, applying $U(\g) \otimes_{U(\p')} -$, we
see that
$M(\tabB)$ embeds into $M(\tabA)$ too.
Now apply the exact functor $H_0$ to the resulting short exact sequence
$0 \to M(\tabB) \to M(\tabA) \to C \to 0$
using Theorem~\ref{T:main},
to obtain an exact sequence
$0 \to \overline{M}(\tabB) \to \overline{M}(\tabA)
\to H_0(C) \to 0.$
But $\overline{M}(\tabB) \cong \overline{M}(\tabA)$ as $\tabB \roweq \tabA$,
hence, we must have that $H_0(C) = 0$. Since $C \twoheadrightarrow L(\tabA)$,
this implies that $H_0(L(\tabA)) = 0$.

We next show for anti-dominant $\tabA$ that
\begin{equation}\label{handy}
[H_0(L(\tabA))] = [\overline{L}(\tabA)]+\text{(a sum of
$[\overline{L}(\tabB)]$ for
$\tabB \in \Tab$ with $a(\tabB) < a(\tabA)$)}.
\end{equation}
To see this, note
since $L(\tabA)$ is a quotient of $M(\tabA)$ that
$H_0(L(\tabA))$ is a quotient of $H_0(M(\tabA)) \cong \overline{M}(\tabA)$.
Applying Corollary~\ref{vermascomp},
we deduce either that (\ref{handy}) holds or that $H_0(L(\tabA)) = 0$.
Also by Lemma~\ref{linkage} we know (as $\tabA$ is anti-dominant) that
$$
[M(\tabA)] = [L(\tabA)] + \text{(a sum of $[L(\tabB)]$ for $\tabB \in \Tab$ with $a(\tabB) < a(\tabA)$)}.
$$
Applying $H_0$ and using (\ref{handy}) whenever $H_0(L(\tabB)) \neq 0$,
we deduce that
$$
[\overline{M}(\tabA))] =
[H_0(L(\tabA))] + \text{(a sum of $[\overline{L}(\tabB)]$
for $\tabB \in \Tab$ with $a(\tabB) < a(\tabA)$)}.
$$
Since
this
definitely involves $[\overline{L}(\tabA)]$,
we must have that $H_0(L(\tabA)) \neq 0$, and we have established
(\ref{handy}).

Now we claim for any $\tabA \in \Tab$ that there exists some
anti-dominant $\tabB \roweq \tabA$ such that $H_0(L(\tabB)) \cong \overline{L}(\tabA)$.
To see this, we know by Corollary~\ref{hrestd}
that there exists some $M \in \mathcal O$ with
$H_0(M) \cong \overline{L}(\tabA)$.
Say we have that $[M] = [L(\tabB_1)]+\cdots+[L(\tabB_k)]$ in the Grothendieck
group
for some $\tabB_1,\dots,\tabB_k\in \Tab$.
In view of (\ref{handy}), only one of $\tabB_1,\dots,\tabB_k$ can be anti-dominant,
and this $\tabB_i$ must satisfy $H_0(L(\tabB_i)) \cong \overline{L}(\tabB_i) \cong \overline{L}(\tabA)$.
This proves the claim.

We have now shown in any row equivalence class of $\pi$-tableaux that there exists at least one anti-dominant $\tabA$ with $H_0(L(\tabA)) \cong \overline{L}(\tabA)$.
Suppose that $\tabB$ is a different
anti-dominant tableau in the same row equivalence class as $\tabA$. We need to show that $H_0(L(\tabB)) \cong \overline{L}(\tabB)$ too.
To prove this we may assume that $\tabB = \sigma(\tabA)$ for some simple transposition $\sigma \in S_m \times S_n$.
Let $\xi$ be the linkage class containing $\tabA$, so that $\sigma(\xi)$ is the
linkage class containing $\tabB$.
By Lemma~\ref{twist}, there is an equivalence $T_\sigma:\mathcal O_\xi \to
\mathcal O_{\sigma(\xi)}$ such that $T_\sigma(L(\tabA)) \cong L(\tabB)$
and $T_\sigma(M(\tabC)) \cong M(\sigma(\tabC))$ for each $\tabC \linked \tabA$.
The $\Z$-linear maps
$[H_0]:K_0(\mathcal O_\xi) \to K_0(\overline{\mathcal O})$
and $[H_0 \circ T_\sigma]:K_0(\mathcal O_\xi) \to K_0(\overline{\mathcal O})$
are equal; this follows because they are equal on
$[M(\tabC)]$ for each $\tabC \linked \tabA$
as $\overline{M}(\tabC) \cong \overline{M}(\sigma(\tabC))$.
Hence, we get that
$$
[H_0]([L(\tabB)]) = [H_0\circ T_\sigma]([L(\tabA)]) =
[H_0]([L(\tabA)]) = [\overline{L}(\tabA)] = [\overline{L}(\tabB)].
$$
This implies that $H_0(L(\tabB)) \cong \overline{L}(\tabB)$ as required.
\end{proof}

\begin{Corollary}
The full subcategory of $\mathcal O$
consisting of all objects annihilated by $H_0$ consists of all the supermodules in $\mathcal O$
of strictly less than maximal Gelfand--Kirillov dimension.
\end{Corollary}

\begin{proof}
This follows from Theorem~\ref{wirr} on recalling that
$\tabA$ is anti-dominant if and only if
$L(\tabA)$ is of maximal Gelfand--Kirillov dimension.
\end{proof}

\subsection{The center of \texorpdfstring{$W$}{W}}\label{scent}
In this subsection
we determine the center of $W$. The argument here is similar in spirit to the proof of an
analogous result in the purely even setting from \cite[\S6.4]{BKrep};
it depends crucially on Corollary~\ref{mek}.
Let $\pr:U(\g) \twoheadrightarrow U(\p)$
be the projection along the direct sum decomposition
$U(\g) = U(\p) \oplus \m_\chi U(\g)$.
It is easy to see from (\ref{e:wdef}) that
the restriction of $\pr$ defines an algebra homomorphism
\begin{equation}\label{eve}
\pr:Z(\g) \to Z(W).
\end{equation}
The goal is to show that this map is actually an {\em isomorphism}.

Consider the Harish-Chandra homomorphism
$\HC:Z(\g) \stackrel{\sim}{\to} I(\t)$ from Theorem~\ref{hc}.
Recalling Theorem~\ref{indep}, we adopt
the definition of $\HC$ that is adapted to the Borel subalgebra $\b'$,
i.e.\ we view $\HC$ as the restriction of the map
\begin{equation}\label{reno}
S_{-\rho'} \circ \phi'
:U(\g)_0 \to S(\t),
\end{equation}
where $\phi':U(\g)_0 \to S(\t)$ is projection along
$U(\g)_0 = S(\t) \oplus (U(\g)_0 \cap U(\g) \n')$
and $\n'$ is the nilradical of $\b'$.
The restriction of (\ref{reno}) to
$Z(\h)$ also gives a conveniently normalized Harish-Chandra homomorphism
for the Lie superalgebra $\h$, that is, an isomorphism
$\hc:Z(\h)\stackrel{\sim}{\to} J(\t)$
where
\begin{equation}\label{jdef}
J(\t) :=
\left\{f \in S(\t)\:\bigg|\:
\begin{array}{l}
\frac{\partial f}{\partial x_i} + \frac{\partial f}{\partial y_j}
\equiv 0 \pmod{x_i-y_j}\\
\text{for $1 \leq i \leq m$ and $j = i+s_-$}
\end{array}
\right\}.
\end{equation}
Also let $\pi:U(\p) \twoheadrightarrow U(\h)$ be the usual projection,
so that $\ker \pi = U(\p) \rr$ where $\rr$ is the nilradical of $\p$.
We have now set up all of the notation
to make sense of the following diagram:
\begin{equation}\label{comd}
\begin{diagram}
\node{U(\g)_0}\arrow{s,l}{\pr}\arrow{e,t}{S_{-\rho'}\circ\phi'}\node{S(\t)}\\
\node{U(\p)_0}\arrow{e,b,A}{\pi}\node{U(\h)_0}\arrow{n,r}{S_{-\rho'}\circ\phi'}
\end{diagram}
\end{equation}
Moreover, this diagram commutes.
The final important point is that the restriction of $\pi$
to $W$ is injective; this is equivalent to the injectivity of the Miura transform in \cite[Theorem 4.5]{BBG}.

\begin{Lemma}\label{theclaim}
The images of
$d_1^{(r)}$ and $\tilde{d}_2^{(r)}$
under the map $S_{-\rho'} \circ \phi' \circ \pi$
are equal to
$e_r(x_1,\dots,x_m)$ and $(-1)^{r} h_r(y_1,\dots,y_n)$, respectively.
\end{Lemma}

\begin{proof}
This
depends on the explicit formulae
for these elements of $W$
from \cite[Section 4]{BBG}.
For example, for $\tilde{d}_2^{(r)}$
remembering the typo pointed out in the footnote on p.\ \pageref{footnotepage}, we have that
\begin{equation*}
\pi(\tilde{d}_2^{(r)}) =
S_{\rho'}\left(
  \sum
(-1)^{r + |i_1|+\cdots+|i_r|}
e_{i_1,j_1} \cdots e_{i_r,j_r}\right),
\end{equation*}
summing over all $1 \le i_1, \dots , i_r, j_1, \dots , j_r
\le m+n$ such that
\begin{itemize}
\item
$\row(i_1) = \row(j_r) = 2$;
\item
$\col(i_s)=\col(j_s)$ for each $s$;
\item
$\row(i_{s+1})=\row(j_{s})$ and
$\col(i_{s+1}) \leq\col(j_s)$ for $s=1,\dots,r-1$.
\end{itemize}
To apply $\phi'$ to this, note
for one of the monomials
$e_{i_1,j_1} \cdots e_{i_r,j_r}$ that $\phi'$
gives zero if $\row(i_r) = 1$, hence,
we may assume that $i_r = j_r$; then we get zero if $\row(i_{r-1}) = 1$,
hence, $i_{r-1}=j_{r-1}$; and so on. We deduce
$S_{-\rho'} (\phi' (\pi(\tilde{d}_2^{(r)})))
 = (-1)^{r} h_r(y_1,\dots,y_n)$
as claimed.
\end{proof}

\begin{Lemma}\label{step}
We have that $\pi(Z(W)) \subseteq Z(\h)$.
\end{Lemma}

\begin{proof}
We must show for $z \in Z(W)$ and $u \in U(\h)$ that $[\pi(z), u] = 0$.
If $\tabA \in \Tab$ is any typical tableau, i.e.\ $\atyp(\tabA) = 0$,
then $K(\tabA)$ is an irreducible $\h$-supermodule
which remains irreducible (with one-dimensional endomorphism algebra) on restriction to $W$, as follows
from Corollaries~\ref{mek} and \ref{vermascomp}.
Hence, $\pi(z)$ acts as a scalar on $K(\tabA)$,
implying that $[\pi(z),u] \in \operatorname{Ann}_{U(\h)} K(\tabA)$.
It remains to observe that
$$
\bigcap_{\substack{\tabA\in\Tab\\\atyp(\tabA) = 0}}
\operatorname{Ann}_{U(\h)} K(\tabA) = 0.
$$
This follows because $\{\tabA\in\Tab\:|\:\atyp(\tabA) = 0\}$ is Zariski dense in $\Tab$
(identified with $\mathbb{A}^{m+n}$ in the obvious way).
Now we can apply the standard fact that the annihilator of any dense set of Verma supermodules is zero,
see for example the proof of \cite[Lemma 13.1.4]{Musson}\footnote{It is
easy to supply a direct proof of this statement in the present situation
since $\h$ is a direct sum of copies of
$\gl_1(\C)$ and $\gl_{1|1}(\C)$.}.
\end{proof}

\begin{Theorem}\label{zonto}
The homomorphism
$\pr:Z(\g) \to Z(W)$ from (\ref{eve})
is an algebra isomorphism.
Moreover, we have that $\pr(z_r) = \tilde{c}^{(r)}$, where
$z_r\in Z(\g)$ and $\tilde{c}^{(r)} \in Z(W)$ are defined by (\ref{ze})
and (\ref{ctilde}), respectively.
\end{Theorem}

\begin{proof}
We observe to start with that $Z(W) \subseteq U(\h)_0 \oplus U(\p) \rr$.
To see this,
note that $U(\p) = U(\h) \oplus U(\p) \rr$.
Hence, we can write $z \in Z(W)$ as $z_0 + z_1$
with $z_0 \in U(\h)$ and $z_1 \in U(\p) \rr$.
Applying $\pi$ and using Lemma~\ref{step},
we get that $z_0 = \pi(z) \in Z(\h) \subseteq U(\h)_0$,
as required.
Hence, it makes sense to restrict all the maps in the commutative diagram
(\ref{comd}) to obtain another commutative diagram
\begin{equation}\label{subd}
\begin{diagram}
\node{Z(\g)}\arrow{s,l}{\pr}\arrow{e,t,J}{\HC}\node{S(\t)}\\
\node{Z(W)}\arrow{e,b,J}{\pi}\node{Z(\h)}\arrow{n,r,J}{\hc}
\end{diagram}
\end{equation}
Since $\HC$ is injective, so too is the map $\pr$.
Since $\tilde{c}^{(r)}
= \sum_{s+t=r} d_1^{(s)} \tilde{d}_2^{(t)}$,
we get from Lemma~\ref{theclaim} and Theorem~\ref{expcent} that
$$
\hc (\pi(\tilde{c}^{(r)}))
= e_r(x_1,\dots,x_m / y_1,\dots,y_n) = \HC(z_r).
$$
Hence, $\pr(z_r) = \tilde{c}^{(r)}$.

To complete the proof of the theorem, we must show that
$\pr$ is surjective.
As $\hc \circ \pi$ is injective and $\HC(Z(\g)) = I(\t)$, this follows
if we can show that $\hc(\pi(Z(W))) \subseteq I(\t)$.
Since $\overline{M}(\tabA) = \overline{M}(\tabB)$ for $\tabA \roweq \tabB$,
Corollary~\ref{mek} implies that
the generalized central character
of the $W$-supermodule
$\overline{K}(\tabA)$ depends only on the row equivalence class of $\tabA$.
Hence, for $z \in Z(W)$ we deduce that
$\pi(z)$ acts by the same scalar on the
$\h$-supermodules
$K(\tabA)$ for all $\tabA$ in the same row equivalence class.
In other words,
$\hc(\pi(Z(W))) \subseteq S(\t)^{S_m\times S_n}$.
We also have that
$\hc(\pi(Z(W))) \subseteq
\hc(Z(\h)) =
J(\t)$.
It remains to observe
by the definitions (\ref{idef}) and (\ref{jdef}) that
$I(\t) = S(\t)^{S_m \times S_n} \cap  J(\t)$.
\end{proof}

\begin{Corollary}\label{bib}
The center $Z(W)$ is generated by the elements $\{\tilde{c}^{(r)}\}_{r
  \geq 1}$;
equivalently, it is generated by the elements $\{c^{(r)}\}_{r \geq 1}$.
\end{Corollary}

\begin{proof}
This follows from Theorems~\ref{zonto} and \ref{expcent}.
\end{proof}

\section{The quotient category \texorpdfstring{$\overline{\O}_\Z$}{O\_Z}}\label{sactions}

For the remainder of the article, we
restrict attention to integral central characters, denoting the
corresponding subcategory of $\O$ by $\O_\Z$.
We introduce an Abelian subcategory $\overline{\mathcal O}_\Z$
of $W\lsmof$ such that the Whittaker coinvariants functor
restricts to a quotient functor
$$
H_0 : \O_\Z \to \overline{\O}_\Z.
$$
We show that this functor satisfies the double centralizer property,
i.e.\ it is fully faithful on projectives.
Then we discuss the locally unital
(``idempotented'') algebras that are Morita equivalent to the blocks of
$\overline{\O}_\Z$, and give some applications to the
classification of blocks of $\O_\Z$.

\subsection{Categorical actions}\label{Catact}
Let $\Tab_\Z$ be the subset of $\Tab$ consisting of the tableaux all of whose entries are integers.
Let
$\mathcal O_\Z$ be the Serre subcategory of $\mathcal O$
generated by the $\mathfrak{g}$-supermodules $\{L(\tabA)\}_{\tabA \in \Tab_\Z}$.
It is a sum of blocks of $\mathcal O$:
\begin{equation}\label{ozblocks}
\mathcal O_\Z = \bigoplus_{\xi \in \Tab_\Z / \linked} \mathcal O_\xi.
\end{equation}
In particular, $\mathcal O_\Z$
is itself a highest weight category with weight poset $(\Tab_\Z, \preceq)$.

Adopting some standard Lie theoretic notation,
let $\mathfrak{sl}_\infty$ be the Kac-Moody algebra of type $A_\infty$
(over $\C$), with Chevalley generators $\{E_i, F_i\}_{i \in \Z}$,
weight lattice $P := \bigoplus_{i \in \Z} \Z \eps_i$,
simple roots $\alpha_i = \eps_{i}-\eps_{i+1}$, etc..
We denote its natural module by
 $V^+$ and the dual by $V^-$. These have standard bases $\{v_j^+\}_{j \in \Z}$
and $\{v_j^-\}_{j \in \Z}$, respectively.
The vector $v_j^{\pm}$ is
of weight $\pm \eps_j$,
and
the Chevalley generators act by
\begin{align}\label{e}
F_i v^+_j &= \left\{\begin{array}{ll}
v^+_{j+1}&\text{if $j = i$}\\
0&\text{otherwise,}
\end{array}\right.
&E_i v^+_j &= \left\{\begin{array}{ll}
v^+_{j-1}&\text{if $j = i+1$}\\
0&\text{otherwise},
\end{array}\right.\\
F_i v^-_j &= \left\{\begin{array}{ll}
v^-_{j-1}&\text{if $j = i+1$}\\
0&\text{otherwise},
\end{array}\right.
&E_i v^-_j &= \left\{\begin{array}{ll}
v^-_{j+1}&\text{if $j = i$}\\
0&\text{otherwise}.
\end{array}\right.\label{f}
\end{align}
As goes back to \cite{Bkl} (or \cite[\S7.4]{CR} in the purely even
case), there is a categorical action of
$\mathfrak{sl}_\infty$ on $\O$ in the sense of Rouquier \cite[Definition 5.32]{Rou};
see also \cite[Definition 2.6]{BLW} for our precise conventions. We
just give a brief summary of the construction,
referring to the proof of \cite[Theorem 3.10]{BLW} for details.
\begin{itemize}
\item
The required biadjoint endofunctors $F$ and $E$ are the functors
\begin{equation}
F := U \otimes -,
\qquad E := U^\vee \otimes -,
\end{equation}
where $U$ is the
natural
$\mathfrak{g}$-supermodule of column vectors and
$U^\vee$ is its dual.
\item
The natural transformations $x:F \Rightarrow F$ and $s:F^2 \Rightarrow
F^2$
are defined so that $x_M: U \otimes M \rightarrow U \otimes M$
is left multiplication by the Casimir tensor
\begin{equation}\label{casten}
\Omega :=
\sum_{i,j=1}^{m+n} (-1)^{|j|} e_{i,j} \otimes e_{j,i} \in \g \otimes
\g,
\end{equation}
and $s_M: U \otimes U \otimes M \rightarrow U \otimes U \otimes M$ is
induced by the tensor flip $U \otimes U \rightarrow U \otimes U, u
\otimes v \mapsto (-1)^{|u||v|} v \otimes u$.
\item
Let $F_i$ be the summand of $F$ defined by taking the generalized
$i$-eigenspace of $x$, and $E_i$ be the unique summand of $E$ that is
biadjoint to it.
Let $\O_\Z^\Delta$ be the exact subcategory of $\O_\Z$ consisting of
all supermodules admitting a Verma flag, and
$K_0(\O^\Delta_\Z)_\C$ be its complexified Grothendieck group.
Let $T^{m|n} :=
(V^+)^{\otimes m} \otimes (V^-)^{\otimes n}$, and set
\begin{equation}\label{monomials}
v_\tabA := v^+_{a_1} \otimes \cdots v^+_{a_m} \otimes v^-_{b_1} \otimes
\cdots v^-_{b_n} \in T^{m|n}
\end{equation}
for each
$\tabA  = \substack{a_1 \cdots a_{m} \\ b_1 \cdots b_{n}}\in
\Tab_\Z$.
Then, there is a vector space isomorphism
\begin{equation}\label{ggg}
K_0(\O^\Delta_\Z)_\C \stackrel{\sim}{\rightarrow}
T^{m|n},
\quad
[M(\tabA)] \mapsto v_\tabA.
\end{equation}
Moreover, this map intertwines the operators
induced by the endofunctors $F_i$ and $E_i$ on the left hand space
with the actions of the Chevalley generators of $\mathfrak{sl}_\infty$ on the
right.
\item
Under the isomorphism
from (\ref{ggg}),
the Grothendieck groups $K_0(\O_\xi^\Delta)_\C$ of the blocks
correspond to the weight spaces
of $T^{m|n}$.
\end{itemize}
In fact, $\O_\Z$ is a {\em tensor product categorification}
of $T^{m|n}$ in the general sense
of \cite[Definition 2.10]{BLW}.

\vspace{2mm}

In the rest of the subsection, we are going to formulate an analogous
categorification theorem at the level of $W$.
Observe that a $\pi$-tableau
$\tabA = \substack{a_1 \cdots a_{m} \\ b_1 \cdots b_{n}} \in \Tab_\Z$
is anti-dominant if and only if $a_1 \leq \cdots \leq a_m$ and
$b_1 \geq \cdots \geq b_n$.
Let $\Tab_\Z^\circ$ denote the set of all such tableaux.
It
gives a distinguished set of representatives for $\Tab_\Z
/\!\roweq$.
For a linkage class $\xi \in \Tab_\Z / \!\linked$, we let
$\xi^\circ$ denote the set $\xi \cap \Tab_\Z^\circ$
of anti-dominant tableaux that it contains.

Recall
for $\tabA \in \Tab_\Z$ that $P(\tabA)$ is the projective cover of $L(\tabA)$ in
$\mathcal O_\Z$. Let
\begin{equation}\label{rain}
\overline{P}(\tabA) := H_0(P(\tabA)).
\end{equation}
Then we define $\overline{\O}_\Z$
to be the full
subcategory of $W\lsmof$
consisting of all $W$-supermodules that are isomorphic to subquotients of
finite direct sums of the
supermodules $\{\overline{P}(\tabA)\}_{\tabA \in \Tab_\Z^\circ}$.
This is obviously an Abelian subcategory of $W\lsmof$.
Similarly, given a linkage class $\xi \in \Tab_\Z / \!\linked$, we let
$\overline{\O}_\xi$ be the full
subcategory consisting
of subquotients of finite direct sums of the supermodules
$\{\overline{P}(\tabA)\}_{\tabA \in \xi^\circ}$.

\begin{Lemma}\label{oblock}
The Whittaker coinvariants functor restricts to an exact
functor
$H_0:\mathcal O_\Z \rightarrow \overline{\O}_\Z$
sending each block $\mathcal O_\xi$ to $\overline{\O}_\xi$.
Each $\overline{\mathcal O}_\xi$ is itself a block (i.e.\ it is
indecomposable), and $\overline{\O}_\Z$ decomposes as
$$
\overline{\mathcal O}_\Z = \bigoplus_{\xi \in \Tab_\Z / \linked}
\overline{\mathcal O}_\xi.
$$
Moreover,
the supermodules
$\{\overline{L}(\tabA)\}_{\tabA \in \xi^\circ}$
give a complete set of inequivalent irreducible objects in
each $\overline{\O}_\xi$.
\end{Lemma}

\begin{proof}
We first show that the essential image of $H_0$ is contained in
$\overline{\O}_\Z$.
Any object  $M \in \mathcal O_\Z$ is a quotient of a direct sum of the
projective objects $P(\tabA)$ for $\tabA \in \Tab_\Z$. Since $H_0$ is exact,
we deduce that
$H_0(M)$ is a quotient of a direct sum of the objects $\overline{P}(\tabA)$ for $\tabA \in \Tab_\Z$.
Since, by definition, $\overline{\mathcal O}_\Z$ is closed under
taking quotients and direct sums,
we are thus reduced to showing that each $\overline{P}(\tabA)$ for $\tabA \in \Tab_\Z$
belongs to $\overline{\O}_\Z$.
This is immediate by the definition of
$\overline{\O}_\Z$ if $\tabA$ is anti-dominant.
So suppose that $\tabA$ is not anti-dominant.
Then $P(\tabA)$ has a Verma flag, and the socle of any Verma is anti-dominant,
hence, the injective
hull of $P(\tabA)$ is a direct sum of $P(\tabB)$ for $\tabB \in \Tab_\Z^\circ$; see \cite[Theorem 2.24]{BLW}.
Applying $H_0$ we deduce that $\overline{P}(\tabA)$ embeds into a direct sum of
$\overline{P}(\tabB)$ for $\tabB \in \Tab_\Z^\circ$. Since $\overline{\mathcal
  O}_\Z$ is closed under taking submodules, this implies that $\overline{P}(\tabA)$
belongs to $\overline{\mathcal O}_\Z$.

Thus, $H_0$ restricts to a well-defined exact functor
$\mathcal O_\Z \rightarrow \overline{\O}_\Z$.
The same argument at the level of blocks shows that $H_0$ maps
$\mathcal O_\xi$ to $\overline{\mathcal O}_\xi$, and clearly
$\overline{\O}_\Z$
decomposes as the direct sum of the $\overline{\mathcal O}_\xi$'s.
The irreducible objects in $\overline{\mathcal O}_\xi$ are just the
irreducible objects of $W\lsmod$ that it contains,
so they are represented by $\{\overline{L}(\tabA)\:|\:\tabA \in \xi^\circ\}$
thanks to Theorem~\ref{wirr}.

It remains to show that each $\overline{\mathcal O}_\xi$ is indecomposable.
Corollary~\ref{vermascomp} implies for any tableaux $\tabA, \tabB$ with $\tabB \upit \tabA$ that the irreducible
supermodules $\overline{L}(\tabA)$ and $\overline{L}(\tabB)$ are both composition factors of
the indecomposable object $\overline{M}(\tabA)$. Hence, $\overline{L}(\tabA)$
and $\overline{L}(\tabB)$ belong to the same block of $\overline{\mathcal O}$.
Now observe that
the equivalence relation $\linked$ on $\Tab_\Z$ is generated by the
relations $\sim$ and $\upit$.
\end{proof}

\begin{Remark}\label{recoering}
By Lemma~\ref{bookkeeping} and the definition of $\overline{\O}_\Z$,
the elements $d_1^{(1)}$ and $d_2^{(1)}$
act semisimply on any
object  $M \in \overline{\O}_\Z$.
Lemma~\ref{bookkeeping} shows moreover that the $z$-eigenspace of
$d_2^{(1)}$
is concentrated in
parity $\parity(z)$, i.e.\ the $\Z/2$-grading is determined by the eigenspace
decomposition of $d_2^{(1)}$. This is a
similar situation to category $\O$ itself, where the $\Z/2$-grading
was determined by the weight space decomposition.
\end{Remark}

Next we introduce endofunctors
$\overline{F}$ and $\overline{E}$ of $\overline{\O}_\Z$.
Consider the biadjoint endofunctors $\overline{F} := U \ostar -$ and
$\overline{E} := U^\vee \ostar -$
of $W\lsmof$
from \S\ref{pf} (where $U$ is still the natural
$\g$-supermodule).
By Theorem~\ref{crappy}, we have canonical isomorphisms of functors
\begin{equation}\label{worse}
\overline{F} \circ H_0 \stackrel{\sim}{\Rightarrow}
 H_0 \circ F,
\qquad
\overline{E} \circ H_0 \stackrel{\sim}{\Rightarrow} H_0 \circ E,
\end{equation}
going from $\O$ to $W\lsmof$.
It follows immediately that $\overline F (\overline{P}(\tabA))
\cong H_0(F P(\tabA))$, hence, it is in the subcategory $\overline{\O}_\Z$.
Since $\overline F$ is exact, it follows that $\overline F$ leaves the subcategory
$\overline{\O}_\Z$ of $W\lsmof$ invariant. Similarly, so does
$\overline E$. Hence, we can restrict these endofunctors to obtain
a biadjoint pair of endofunctors
\begin{equation}
\overline F:\overline{\O}_\Z \rightarrow
\overline{\O}_\Z,
\qquad
\overline E:\overline{\O}_\Z \rightarrow
\overline{\O}_\Z.
\end{equation}

\begin{Theorem}\label{bigcat}
There are natural transformations $\bar x:\overline{F} \Rightarrow
\overline{F}$
and $\bar s:\overline{F}^2 \Rightarrow \overline{F}^2$
making
$\overline{\O}_\Z$ into an integrable
$\mathfrak{sl}_\infty$-categorification.
Moreover, the Whittaker coinvariants functor
$
H_0:\O_\Z \rightarrow \overline{\O}_\Z
$
is strongly equivariant in the usual sense of categorical actions (e.g.\ see
\cite[Definition 2.7]{BLW}).
\end{Theorem}

\begin{proof}
First, we go through the construction of $\bar x$.
For a left $U(\g)$-supermodule $M$, we already have $x_M:U \otimes M
\rightarrow U \otimes M$ defined by left multiplication by the tensor
$\Omega$ from (\ref{casten}).
Under the isomorphism (\ref{dualitymap}),
the dual map $(x_M)^*:(U \otimes M)^* \rightarrow (U \otimes M)^*$
is
the map $x_{M^*}:M^* \otimes U^* \rightarrow M^* \otimes U^*$
defined by right multiplication by $\Omega$.
Now suppose that $M \in W \lsmof$.
Applying $H^0$ to $x_{M^* \otimes_W Q}$ gives us an endomorphism
$\bar{x}_{M^*}$
of
$M^* \circledast U^* =
H^0((M^* \otimes_W Q) \otimes U^*)$.
Finally, taking the left dual gives us an
endomorphism $\bar x_M := {^*}({\bar x}_{M^*})$ of $U
\circledast M$.

The definition of $\bar s$ can be obtained in a very similar way, but it
is easier to define this using the
coherence isomorphism $U \ostar (U \ostar M) \cong (U \otimes U) \ostar M$
coming from the monoidal functor (\ref{itsmonoidal}), starting from
the
endomorphism of $(U \otimes U) \ostar M$ obtained by applying $-
\ostar \id_M$ to the tensor flip $U \otimes U \mapsto U
\otimes U$.

The fact that $\bar x$ and $\bar s$ satisfy the appropriate degenerate affine
Hecke algebra relations is just a formal consequence of the fact that
$x$ and $s$ do on $U(\g)\lsmod$.
Also, we've already constructed $\overline{F}$ and $\overline{E}$ so
that they are canonically biadjoint.

Next we show that $H_0$ is a strongly equivariant functor.
We have already constructed the required data of an isomorphism
$\zeta:\overline{F} \circ H_0 \stackrel{\sim}{\Rightarrow} H_0 \circ
F$
on the left hand side of  (\ref{worse}).
We next have to check that $x$ and $s$ are intertwined with $\bar x$
and $\bar s$ in the appropriate sense (we need the $F$-version of
\cite[5.2.1(5)]{CR} as recorded in \cite[Definition 2.7(E2)--(E3)]{BLW}).
This is a formal exercise from the definitions (which were set up
exactly for this purpose). Finally, we must check the $F$-version of \cite[5.1.2(4)]{CR}
(which is \cite[Definition 2.7(E1)]{BLW}). This asserts that a certain
natural transformations $H_0 \circ E \Rightarrow \overline{E} \circ
H_0$ constructed from $\zeta$ using the adjunction
is an isomorphism. In fact, one shows that it is the inverse of the
right hand side
of (\ref{worse}). We omit the details here.

Then we decompose $\overline{F}$ into its $\bar x$-generalized
eigenspaces $\overline{F}_i$
as before, and let $\overline{E}_i$ be the adjoint summands of
$\overline{E}$.
Finally, we need to show that the induced actions of
$[\overline{F}_i]$ and $[\overline{E}_i]$ make
$K_0(\overline{\O}_\Z)_\C$ into an integrable representation of $\mathfrak{sl}_\infty$.
This follows from the equivariance of $H_0$:
we already know that $K_0(\O_\Z)_\C$ is integrable upstairs,
and the $\mathfrak{sl}_\infty$-equivariant
map $[H_0]:K_0(\O_\Z)_\C \rightarrow K_0(\overline\O_\Z)_\C$
is surjective according to Theorem~\ref{wirr} and the description of
irreducible objects in Lemma~\ref{oblock}.
\end{proof}

The
Grothendieck group $K_0(\overline{\O}_\Z)_\C$  may be understood from
the point of view of
this categorification theorem as follows.

\begin{Lemma}\label{bells}
Let $S^{m|n} := S^m V^+ \otimes S^n V^-$ (tensor product of symmetric powers).
Then,  there is a unique injective linear map $j$
making the following into a commutative diagram of
$\mathfrak{sl}_\infty$-module homomorphisms:
\begin{equation*}
\begin{diagram}
\node{\:\:T^{m|n}}\arrow{s,r,J}{i}\arrow{r,t,A}{\operatorname{can}}\node{\:\:S^{m|n}}\arrow{s,r,J}{j}\\
\node{K_0(\mathcal
  O_\Z)_\C}\arrow{r,t,A}{[H_0]}\node{K_0(\overline{\mathcal O}_\Z)_\C}
\end{diagram}
\end{equation*}
Here,
 the top map is the canonical map from
tensor powers to symmetric powers,
and $i$ is the composition of the inverse of (\ref{ggg}) with the
natural inclusion
 $K_0(\O_\Z^\Delta)_\C \hookrightarrow
K_0(\O_\Z)_\C$.
\end{Lemma}

\begin{proof}
To see this, one just has to observe that $H_0(M(\tabA)) \cong H_0(M(\tabB))$
for all $\tabA
\sim \tabB$ thanks to Theorem~\ref{T:main}. Moreover, the classes of the Verma
supermodules
$\{[\overline{M}(\tabA)]\}_{\tabA \in \Tab^\circ}$ are linearly independent in
$K_0(\overline{\O}_\Z)_\C$
by the classification of irreducible objects.
\end{proof}

\subsection{Serre quotients and the double centralizer
  property}\label{catr}
Throughout the subsection, we often appeal to Theorem~\ref{wirr} and
the exactness of $H_0$ from Lemma~\ref{exactness}.
Although it is immediate from the definition that $\overline{\O}_\Z$ is an Abelian category,
we do not yet
know that it has enough projectives or injectives.
We proceed to establish this, essentially mimicking the proof
of \cite[Lemma 5.7]{BKschur}.

\begin{Lemma}\label{projc}
For each $\tabA \in \Tab^\circ_\Z$, the supermodule $\overline{P}(\tabA)$ is both
the projective cover and the injective hull of $\overline{L}(\tabA)$ in $\overline{\O}_\Z$.
\end{Lemma}

\begin{proof}
We need the following fact established in \cite[Theorem 2.24]{BLW}: for any $\tabA \in
\Tab_\Z^\circ$, the prinjective supermodule $P(\tabA)$ is a direct summand
of
$F^d P(\tabB)$ for some $d \geq 0$ and some $\tabB \in \Tab^\circ_\Z$ of the special form
$\tabB =  \substack{a \cdots a \\ b \cdots b}$ with $a \neq b$.
Define $d(\tabA)$ to be the smallest $d$ such that this is the case.

We'll prove the lemma by induction on $d(\tabA)$.
For the base case $d(\tabA) = 0$, we have that
$\tabA$ is the only $\pi$-tableau in its linkage class, so that
$P(\tabA) = L(\tabA)$. Hence, $\overline{P}(\tabA) = H_0(P(\tabA)) = H_0(L(\tabA)) =
\overline{L}(\tabA)$. We deduce immediately from its definition that
$\overline{\O}_\xi$ is simple (i.e.\ equivalent to the category of
finite-dimensional vector spaces). Now the conclusion is
trivial
in this case.

For the induction step, take $\tabA\in \Tab_\Z^\circ$ with $d(\tabA) > 0$.
The functors $F$ and $\overline{F}$ both have biadjoints, hence,
they send prinjectives to
prinjectives.
Using Lemma~\ref{prinj} and the definition of $d(\tabA)$,
we can find
some $\tabC \in \Tab^\circ_\Z$ with $d(\tabC) = d(\tabA)-1$ such that $P(\tabA)$ is a
summand of $F P(\tabC)$.
By induction, $\overline{P}(\tabC)$ is both the projective cover and the
injective hull of $\overline{L}(\tabC)$.
So we have that
\begin{align*}
F P(\tabC) &\cong
\bigoplus_{\tabB \in \Tab_\Z^\circ} P(\tabB)^{\oplus m_\tabB},&
\overline{F} \,\overline{P}(\tabC) &\cong
\bigoplus_{\tabB \in \Tab_\Z^\circ} \overline{P}(\tabB)^{\oplus m_\tabB},
\end{align*}
for some multiplicities $m_\tabB$ with $m_\tabA > 0$,
and deduce that $\overline{P}(\tabA)$ is prinjective in $\overline{\O}_\Z$.

Let $\tabB\in \Tab_\Z^\circ$.
Since $L(\tabB)$ appears in the head of $P(\tabB)$,
we see that $\overline{L}(\tabB)$ appears in the head of
$\overline{P}(\tabB)$.
So for $D \in \Tab_\Z^\circ$, we have
$\dim \Hom_{W}(\overline{P}(\tabB), \overline{L}(\tabD)) \geq \delta_{\tabB,\tabD}$
and
$$
\dim\Hom_W(
\overline{F} \,\overline{P}(\tabC), \overline{L}(\tabD))
\geq \sum_{\tabB \in \Tab_\Z^\circ} m_\tabB \delta_{\tabB,\tabD} = m_\tabD.
$$
Moreover, the equality holds here if and only if
$\dim \Hom_{W}(\overline{P}(\tabB), \overline{L}(\tabD)) =
\delta_{\tabB,\tabD}
$
for all $\tabB$ with $m_\tabB > 0$.
This is indeed the case thanks to the following calculation:
\begin{align*}
\dim \Hom_W(\overline{F}\, \overline{P}(\tabC), \overline{L}(\tabD))
&=
\dim \Hom_W(\overline{P}(\tabC),\overline{E} \,\overline{L}(\tabD)) \\
&= [\overline{E} \,\overline{L}(\tabD): \overline{L}(\tabC)]\\
&= [E L(\tabD):L(\tabC)] \\
&=
\dim\Hom_{\g}(P(\tabC), E L(\tabD))\\
&=
\dim\Hom_{\g}(F P(\tabC), L(\tabD))
= m_\tabD.
\end{align*}

The previous paragraph
establishes that
$\dim \Hom_{W}(\overline{P}(\tabA), \overline{L}(\tabB)) =
\delta_{\tabA,\tabB}$ for all $\tabB$, so $\overline{P}(\tabA)$ has irreducible head $\overline{L}(\tabA)$.
Thus we have shown that $\overline{P}(\tabA)$ is the projective cover of
$\overline{L}(\tabA)$ in $\overline{\O}_\Z$, as required.
A similar calculation
shows that $\dim \Hom_{W}(\overline{L}(\tabB), \overline{P}(\tabA))
=\delta_{\tabA,\tabB}$,
and $\overline{P}(\tabA)$ is the injective hull of $\overline{L}(\tabA)$ too.
\end{proof}

\begin{Lemma}\label{mac}
For any $\tabA \in \Tab^\circ$ and $M \in \O_\Z$, the functor
$H_0$ induces an isomorphism
$$
\Hom_{\g}(P(\tabA), M) \stackrel{\sim}{\rightarrow}
\Hom_W(\overline{P}(\tabA), H_0(M)).
$$
\end{Lemma}

\begin{proof}
We are trying to show that the natural transformation
$\Hom_{\g}(P(\tabA),-) \Rightarrow \Hom_W(\overline{P}(\tabA), H_0(-))$
induced by the functor $H_0$ is an isomorphism.
Since $H_0$ is exact, it suffices to check this gives an
isomorphism as in the statement
for $M$ an irreducible supermodule in $\O_\Z$.
If $M = L(\tabB)$ for $\tabB \in \Tab_\Z$, then both sides are zero unless $\tabB
= \tabA$, thanks to Theorem~\ref{wirr} and Lemma~\ref{projc}.
If $\tabB=\tabA$ then, by Lemma~\ref{projc}, both sides are one-dimensional.
The left hand side is spanned by an
epimorphism $P(\tabA) \twoheadrightarrow L(\tabA)$, so remains
non-zero when we apply $H_0$.
Hence, $H_0$ does indeed give an isomorphism.
\end{proof}

\begin{Lemma}\label{es}
The functor $H_0$ is essentially surjective.
\end{Lemma}

\begin{proof}
Let $M \in \overline \O_\Z$.
Applying Lemma~\ref{projc}, we can
construct a two-step projective resolution
$$
\overline{P}_1 \stackrel{\overline{f}}{\to} \overline{P}_0 \to M \to 0
$$
in
$\overline{\O}_\Z$. This means that $M \cong \operatorname{coker} \overline{f}$ for projectives
$\overline{P}_1, \overline{P}_0 \in \overline{\mathcal O}_\Z$ and
$\overline{f} \in \Hom_W(\overline{P}_1, \overline{P}_0)$.
Let $P_1, P_0 \in {\mathcal O}_\Z$ be prinjectives such that
$H_0(P_1) \cong \overline{P}_1$ and $H_0(P_0) \cong \overline{P}_0$.
By Lemma~\ref{mac}, the functor $H_0$ defines an isomorphism
$$
\Hom_{\g}(P_1, P_0) \stackrel{\sim}{\rightarrow}
\Hom_W(\overline{P}_1, \overline{P}_0).
$$
Hence,
there exists
$f \in \Hom_\g(P_1, P_0)$ so that $H_0(f)$ identifies with $\overline{f}$.
Then, using exactness, we get that
$H_0 (\operatorname{coker} f) \cong \operatorname{coker} \overline{f}
\cong M$.
 \end{proof}

\begin{Theorem}\label{qmain}
The functor $H_0:\O_\Z \rightarrow \overline{\O}_\Z$ satisfies the
universal property of the Serre quotient of $\O_\Z$ by the Serre
subcategory $\mathcal{T}_\Z$ consisting of all supermodules of less
than maximal Gelfand-Kirillov dimension.
\end{Theorem}

\begin{proof}
Recalling Lemma~\ref{prinj}, $\mathcal{T}_\Z$ is generated by
$\{L(\tabA)\}_{\tabA \in \Tab_\Z \setminus \Tab^\circ_\Z}$.
By Theorem~\ref{wirr},
the exact functor $H_0$ annihilates all of these objects.
Hence, by the universal property of the
Serre quotient functor
$Q:\O_\Z \rightarrow \O_\Z / \mathcal{T}_\Z$, we get an
induced functor $G:\O_\Z / \mathcal{T}_\Z \rightarrow \overline{\O}_\Z$
such that $H_0 = G \circ Q$.
By Lemma~\ref{es}, $G$ is essentially surjective. It just remains to show
that it is fully faithful, i.e.\ for all $M, N \in \O_\Z$ we have that
$
G:\Hom_{\mathcal O_\Z / \mathcal T_\Z}(Q M,Q N)
\stackrel{\sim}{\rightarrow} \Hom_{W}(H_0(M), H_0(N)).
$
This is clear from Lemma~\ref{mac}
in case $M$ is prinjective, since $Q$ satisfies an analogous property
by the general theory of quotient functors.
Take any $M' := Q M$ and $N' := QN$
and
a two-step projective resolution $
P_1' \to P_0' \to M' \to 0
$
in $\mathcal O_\Z / \mathcal T_\Z$.
We get a commuting diagram
$$
\begin{diagram}
\node{\!\!\!\!0\rightarrow\Hom_{\mathcal O_\Z / \mathcal T_\Z}(M',N')}\arrow{r}\arrow{s,r}{G}\node
{\Hom_{\mathcal O_\Z / \mathcal T_\Z}(P_0',N')}\arrow{r}\arrow{s,r}{G}\node{\Hom_{\mathcal O_\Z / \mathcal T_\Z}(P_1',N')}\arrow{s,r}{G}\\
\node{\!\!\!\!0\rightarrow\Hom_W(GM',GN')}\arrow{r}\node{\Hom_W(G P_0',G
  N')}\arrow{r}\node{\Hom_W(G P_1',G N')}\
\end{diagram}
$$
with exact rows.
We've already established that the
last two vertical maps are isomorphisms, hence, so is the first one.
\end{proof}

\begin{Corollary}\label{controls}
The functor $H_0:\O_\Z \rightarrow \overline{\O}_\Z$ is fully faithful
on projectives.
\end{Corollary}

\begin{proof}
Given the above theorem, this follows from \cite[Theorem 4.10]{BLW}.
\end{proof}

We stress that, although $\overline{\O}_\Z$ is a quotient of a highest weight
category, it is not highest weight itself (except in the trivial case $m+n=1$).

\subsection{Parametrization of blocks by core and atypicality}\label{morec}
At this point, it is convenient to switch from using anti-dominant
$\pi$-tableaux as
our preferred index set for the irreducible objects of
$\overline{\O}_\Z$ to some equivalent but more suggestive formalism.
By a {\em composition} $\lambda \vDash n$, we mean
an infinite tuple $\lambda = (\lambda_i)_{i \in \Z}$ of non-negative
integers whose sum is $n$.
The sum of two compositions is
obtained simply by adding their corresponding parts.
The {\em strictification} $\lambda^+$ of $\lambda$ is the strict
composition $(\lambda^+_1,\dots,\lambda^+_\ell)$ of $n$ obtained from
$\lambda$ by discarding all of its parts that equal zero.
The {\em transpose} $\lambda^T$ of $\lambda$ is the partition
$(\lambda_1^T, \lambda_2^T,\dots)$ of $n$ defined from $\lambda_i^T := \#\{j \in
\Z \mid 0 < \lambda_j \leq i\}$.
For example, if $\lambda = (\dots,0,2,4,0,0,1,0,\dots)$
then $\lambda^+ = (2,4,1)$ and
$\lambda^T = (3,2,1,1,0,0,\dots)$.
Also, we say that two compositions $\mu,\nu \vDash n$ are {\em equal
  up to translation and duality} if there exists $s \in \Z$ such that
{\em either}
$\mu_i = \nu_{s+i}$ for all $i \in \Z$
{\em or} $\mu_i = \nu_{s-i}$ for all $i \in \Z$.

Compositions $\lambda \vDash n$ may be
identified with special elements
of the weight lattice $P$
of $\mathfrak{sl}_\infty$
via the dictionary
$\lambda \vDash n \leftrightarrow
\sum_{i \in \Z} \lambda_i \eps_i \in P$.  For example, $t \eps_i$ is the composition whose $i$th part is equal
to $t$, with all other parts being zero.
Then the usual dominance order $<$ on $P$ determined by
the simple roots $\alpha_i=\eps_i-\eps_{i+1}$ corresponds to the partial order on compositions given
by $\lambda < \mu$ if $\sum_{j \leq i} \lambda_j < \sum_{j\leq i} \mu_j$ for all $i$.
If $\lambda \vDash n$
then $\lambda+\alpha_i \in P$ is
a well-defined composition of $n$ if and only if $\lambda_{i+1} > 0$, in
which case it is the composition with $\lambda_i+1$ as its $i$th part,
$\lambda_{i+1}-1$ as its $(i+1)$th part, and all other parts the
same as $\lambda$.

The point of this is that
the set
\begin{equation}\label{Xi}
\Xi(m|n) := \left\{(\mu,\nu;t)\:\bigg|\:
\begin{array}{ll}
0 \leq t \leq m, \mu \vDash
m-t, \nu \vDash n-t\\
\text{such that }\mu_i\nu_i = 0\text{ for
  all $i \in \Z$}
\end{array}\right\}
\end{equation}
is in bijection with the set of linkage classes $\xi \in \Tab_\Z /
\!\linked$.
To understand how this goes, given $(\mu,\nu;t) \in \Xi(m|n)$ and $\lambda
\vDash t$, we define $\tabA(\mu,\nu;\lambda)$ to be the unique
anti-dominant tableau that has $\lambda_i+\mu_i$ entries equal
to $i$ in its top row, and $\lambda_i+\nu_i$ entries equal to $i$ in
its bottom row.
For example, if $m=3,n=3, t=1$ and $\mu \vDash 2, \nu \vDash 2$ and
$\lambda \vDash 1$ are the compositions with $\mu_5 =
2, \nu_3=\nu_4= 1$ and $\lambda_j = 1$ for some $j \in \Z$, then
\begin{equation}\label{palatable}
\tabA(\mu,\nu;\lambda) = \left\{
\begin{array}{ll}
\substack{5\, 5\, j \\ j\, 4\, 3}&\text{if $j \geq 5$,}\\\\
\substack{4\,5\, 5 \\ 4\, 4\, 3}&\text{if $j =4$,}\\\\
\substack{j\, 5\, 5\\ 4\, 3\, j}&\text{if $j \leq 3$.}
\end{array}
\right.
\end{equation}
In general, the set $\{\tabA(\mu,\nu;\lambda)\}_{\lambda \vDash t}$
is equal to $\xi^\circ$ for a unique $\xi \in \Tab_\Z / \!\linked$ of
atypicality $t$,
and all linkage classes arise in this way.

Henceforth, we {\em identify} elements $(\mu,\nu;t)$ of $\Xi(m|n)$ with
linkage classes $\xi \in \Tab_\Z /\!\linked$
via the bijection described in the previous paragraph, denoting
the block decompositions
of $\O_\Z$ and $\overline{\O}_\Z$
instead by
$$
\O_\Z = \bigoplus_{\xi \in \Xi(m|n)} \O_\xi,
\qquad
\overline{\O}_\Z = \bigoplus_{\xi \in \Xi(m|n)} \overline{\O}_\xi,
$$
respectively.
Thus, blocks are parameterized by
an {\em atypicality} $t$
and a {\em core} $(\mu,\nu)$.
As usual, the indecomposable projective, standard and irreducible objects
of $\O_\xi$ are represented by the supermodules
$P(\tabA), M(\tabA)$ and $L(\tabA)$ for $\tabA \in \xi$.
For
$\xi = (\mu,\nu;t) \in \Xi(m|n)$
and $\lambda \vDash t$, we will usually
write $\overline{P}_\xi(\lambda), \overline{M}_\xi(\lambda)$ and
$\overline{L}_\xi(\lambda)$ in place of
$\overline{P}(\tabA)$,
$\overline{M}(\tabA)$ and
$\overline{L}(\tabA)$ for $\tabA := \tabA(\mu,\nu;\lambda)$.
In this way,
these families of objects in
$\overline{\O}_\xi$ are now parameterized by compositions
$\lambda \vDash t$ rather than by anti-dominant tableaux.
For example, the block $\xi$ associated to the anti-dominant
tableau $\tabA =
\substack{1\, 1\, 3\, 3 \\ 4\, 2\, 1\, 1}$ has atypicality 2 and core
$(\mu,\nu)$ where $\mu,\nu\vDash 2$ have
$\mu_3 = 2$ and $\nu_2 = \nu_4 = 1$.
Moreover, $\overline{L}(\tabA) = \overline{L}_\xi(\lambda)$ where
$\lambda \vDash 2$ has $\lambda_1=2$.

\subsection{Formal characters}
Let $\{\chi_i\}_{i \in \Z}$ be indeterminates. Set $\chi^{\alpha_i} :=
\chi_i \chi_{i+1}^{-1}$ and $\chi^\eta := \prod_{i \in \Z} \chi_i^{\eta_i}$
for $\eta \vDash m$.
Let $e_r(\eta)$ be the $r$th elementary symmetric function
$e_r(a_1,\dots,a_m)$ where $a_1,\dots,a_m$ are chosen so that $\eta_i$
of them are equal to $i$ for each $i \in \Z$.
Then, for any finite-dimensional $W$-module $M$,
we define its {\em $\eta$-weight space}
\begin{equation}\label{thewtspace}
M_\eta := \big\{v \in M\:\big|\:(d_1^{(r)} - e_r(\eta))^N v = 0\text{ for }N
\gg 0\big\}.
\end{equation}
The {\em formal character} of $M$ is
\begin{equation}
\operatorname{ch} M := \sum_{\eta \vDash m} (\dim M_\eta) \chi^\eta \in
\Z[\chi_i\:|\:i \in \Z].
\end{equation}

\begin{Theorem}\label{characters}
For $\xi = (\mu,\nu;t) \in \Xi(m|n)$ and $\lambda \vDash t$, we have that
\begin{align}\label{mch}
\operatorname{ch} \overline{M}_\xi(\lambda) &=
\chi^{\lambda+\mu}
\prod_{i \in \Z} (1+\chi^{\alpha_i})^{\lambda_{i+1}+\mu_{i+1}},\\
\operatorname{ch} \overline{L}_\xi(\lambda) &= \chi^{\lambda+\mu}
\prod_{i \in \Z} (1+\chi^{\alpha_i})^{\mu_{i+1}}.\label{lch}
\end{align}
Moreover, for any $M \in \overline{\O}_\xi$,
we have that $M = \bigoplus_{\eta \vDash m} M_\eta$.
\end{Theorem}

\begin{proof}
We first prove (\ref{mch}).
Let $\tabA := \tabA(\mu,\nu;\lambda)$ and
$a_1,\dots,a_m$ be the entries along the top row of $\tabA$.
By Corollary~\ref{mek},
we have that $\operatorname{ch} \overline{M}_\xi(\lambda) =
\operatorname{ch} \overline{K}(\tabA)$, and will compute the latter.
The advantage of this is that $\overline{K}(\tabA)$ is the restriction
of the $\mathfrak{h}$-supermodule $K(\tabA)$.
Recalling (\ref{hstruct}), $K(\tabA)$ possesses
a basis of $\mathfrak{t}$-weight vectors
$\{v_\theta\}_{\theta = (\theta_1,\dots,\theta_m) \in \{0,1\}^m}$
such that $S_{\rho'}(x_{i}) v_\theta = (a_i - \theta_i) v_\theta$ for each
$i=1,\dots,m$ (where $x_i = e_{i,i}$ as in \S\ref{hchomsec}).
Hence,
$$
S_{\rho'}(e_r(x_1,\dots,x_m))v_\theta =
e_r(a_1-\theta_1,\dots,a_m-\theta_m) v_\theta.
$$
For two tuples $\theta,\theta' \in \{0,1\}^m$, we write $\theta'
>_{\operatorname{lex}} \theta$ if
$\theta'_j > \theta_j, \theta_{j+1}'=\theta_{j+1},\dots,\theta'_m =
\theta_m$ for some $1 \leq j \leq m$.
Then we observe that
$$
d_1^{(r)} v_\theta = e_r(a_1-\theta_1,\dots,a_m-\theta_m) v_\theta + (\text{a
  linear combination of $v_{\theta'}$'s for $\theta'
  >_{\operatorname{lex}} \theta$}).
$$
This follows from the explicit formula for
$\pi(d_1^{(r)})$ recorded in the proof of \cite[Lemma 8.3]{BBG};
see also Lemma~\ref{theclaim}.
Hence,
we see that
$v_\theta$ contributes the monomial
$\chi_{a_1-\theta_1}\cdots \chi_{a_m-\theta_m}$ to the formal character of
$\overline{K}(\tabA)$.
We have now shown that
$$
\operatorname{ch} \overline{M}_\xi(\lambda) =
\chi_{a_1}\cdots \chi_{a_m}
\sum_{\theta \in \{0,1\}^m} \bigg[\prod_{{\substack{1\leq i\leq m \\ \theta_i = 1}}} \chi^{\alpha_{a_i-1}}\bigg],
$$
which simplifies to give (\ref{mch})
since $\chi_{a_1}\cdots \chi_{a_m} = \chi^{\lambda+\mu}$.

The proof of (\ref{lch}) is very similar, using instead that
$\operatorname{ch} \overline{L}_\xi(\lambda) =
\operatorname{ch}\overline{V}(\tabB)$ according to Theorem~\ref{irco},
where $\tabB \sim \tabA(\mu,\nu;\lambda)$ is chosen so
that the entries along its top row are
$b_1,\dots,b_{m-t},c_1,\dots,c_t$ and the entry immediately below each
of the $c_i$'s is another $c_i$.
Then $V(\tabB)$ possesses
a basis of $\mathfrak{t}$-weight vectors
$\{v_\theta\}_{\theta = (\theta_1,\dots,\theta_{m-t}) \in \{0,1\}^{m-t}}$
such that
$S_{\rho'}(x_{i}) v_\theta = (b_i - \theta_i) v_\theta$ for
$i=1,\dots,m-t$ and
$S_{\rho'}(x_{i}) v_\theta = c_i v_\theta$
for $i=m-t+1,\dots,m$.
So the same argument as in the previous paragraph gives that
$$
\operatorname{ch} \overline{L}_\xi(\lambda) =
\chi_{b_1}\cdots \chi_{b_{m-t}} \chi_{c_1}\cdots \chi_{c_t}
\sum_{\theta \in \{0,1\}^{m-t}} \bigg[\prod_{{\substack{1\leq i\leq m-t \\ \theta_i = 1}}} \chi^{\alpha_{b_i-1}}\bigg],
$$
which simplifies to give (\ref{lch}).

Finally, to get the last sentence, we just exhibited a basis
showing that it is true for $M
=\overline{L}_\xi(\lambda)$, which is enough to establish it in general.
\end{proof}

\begin{Corollary}\label{injectivityofch}
The map $K_0(\overline{\O}_\xi) \rightarrow \Z[\chi_i\:|\:i \in
\Z]$ given by $[M] \mapsto \operatorname{ch}(M)$ is injective.
\end{Corollary}

\begin{proof}
By Theorem~\ref{characters},
$\operatorname{ch}\overline{L}_\xi(\lambda)$
is equal to $\chi^{\lambda+\mu}$ plus a sum of terms of the form $\chi^\nu$
for $\nu > \lambda+\mu$.
Hence, the formal characters of the irreducible objects in
$\overline{\O}_\xi$ are linearly independent, which implies the corollary.
\end{proof}

The following lemma describes what the $e^{(r)}$'s and $f^{(r)}$'s do
to weight spaces.

\begin{Lemma}\label{jumpers}
For any finite-dimensional $W$-module $M$ and $\eta \vDash m$, we have that
\begin{align}\label{acty1}
f^{(s_-+r_1)}
\cdots f^{(s_-+r_k)}
 M_\eta&\subseteq
\bigoplus_{\substack{\theta \vDash k \\
     \theta_i \leq \eta_{i+1}}}
M_{\eta+\sum_i \theta_i \alpha_i},\\
e^{(s_++r_1)}\cdots e^{(s_++r_k)} M_\eta &\subseteq
\bigoplus_{\substack{\theta \vDash k \\  \theta_i \leq \eta_i}}
M_{\eta-\sum_i \theta_i\alpha_i}\label{acty2}
\end{align}
for all $k \geq 0$ and $r_1,\dots,r_k > 0$.
\end{Lemma}

\begin{proof}
We will prove (\ref{acty1}).
Then (\ref{acty2}) follows
by twisting with the involution $\iota:W \stackrel{\sim}{\rightarrow} W$ given by
$d_i^{(r)} \mapsto (-1)^r d_i^{(r)}$,
$e^{(s_++r)} \mapsto (-1)^r f^{(s_-+r)}$, and $f^{(s_-+r)} \mapsto (-1)^r
e^{(s_++r)}$;
note for this that $\iota^*(M_\eta) = \iota^*(M)_{\eta'}$
where $\eta'_i = \eta_{-i}$.

To establish (\ref{acty1}),
let $W_1^0$ (resp. $W_1^\flat$)
be the subalgebra of
$W$ generated by $d_1^{(1)},\dots,d_1^{(m)}$
(resp. by $d_1^{(1)},\dots,d_1^{(m)}, f^{(s_-+1)},\dots,f^{(s_-+m)}$).
For $\eta \vDash m$,
we define the weight spaces $M_\eta$ of a finite-dimensional $W_1^\flat$-module $M$
by the same formula (\ref{thewtspace}) as before.
Let $\C_\eta$ be a one-dimensional
$W_1^0$-module with basis $\overline{1}_\eta$ such that $d_1^{(r)}
\overline{1}_\eta =
e_r(\eta) \overline{1}_\eta$ for each $r$.
Then form the induced module
$\overline{M}(\eta) := W_1^{\flat} \otimes_{W_1^0} \C_\eta$,
setting $\overline{m}_\eta := 1 \otimes \overline{1}_\eta$.
We claim that
\begin{equation}\label{thing1}
f^{(s_-+r_1)}
\cdots f^{(s_-+r_k)} \overline{m}_\eta\in
\bigoplus_{\substack{\theta \vDash k \\
     \theta_i \leq \eta_{i+1}}}
\overline{M}(\eta)_{\eta+\sum_i \theta_i \alpha_i}.
\end{equation}
To deduce (\ref{acty1}) from this claim, it suffices to show
for any $v \in M_\eta$ that is a simultaneous eigenvector
for all $d_1^{(1)},\dots, d_1^{(m)}$ that
$f^{(s_-+r_1)}
\cdots f^{(s_-+r_k)} v$ belongs to the subspace on the right hand side
of (\ref{acty1}).
This follows from (\ref{thing1}) because there is a unique $W_1^\flat$-module
homomorphism
$\omega: \overline{M}(\eta) \rightarrow M$ such that  $\overline{m}_\eta \mapsto
v$, and $\omega\left(\overline{M}(\eta)_{\eta+\sum_i \theta_i \alpha_i}\right)
\subseteq M_{\eta+\sum_i \theta_i \alpha_i}$.

Finally, to prove (\ref{thing1}),
we pick any block $\xi = (0,\nu;m) \in \Xi(m|n)$ of maximal
atypicality.
Applying the PBW theorem for $W$,
we see that
$W_1^\flat$-module $\overline{M}(\eta)$ may be identified with
the restriction
of the Verma supermodule $\overline{M}_\xi(\eta)$, so that
$\overline{m}_\eta$ is the highest weight vector in $\overline{M}_\xi(\eta)$.
As $\overline{m}_\eta$ is a $d_1^{(1)}$-eigenvector of
eigenvalue
$\sum_{i \in \Z} i \eta_i$, the relations imply that
$f^{(s_-+r_1)} \cdots f^{(s_-+r_k)} \overline{m}_\eta$
is a $d_1^{(1)}$-eigenvector of eigenvalue
$\sum_{i \in \Z} i \eta_i - k$. Hence, this vector
lies in the sum of
the weight spaces $\overline{M}_\xi(\eta)_{\eta'}$ for $\eta' \vDash m$ with $\sum_i
i \eta_i' = \sum_i i \eta_i - k$.
Finally, we apply (\ref{mch}) to see that
$\overline{M}_\xi(\eta)_{\eta'}$ is zero unless $\eta' = \eta + \sum_i
\theta_i \alpha_i$ for $\theta \vDash k$ with $\theta_i \leq
\eta_{i+1}$ for all $i$.
\end{proof}

\subsection{Cartan matrix of \texorpdfstring{$\overline{\O}_\xi$}{O\_xi}}
The next goal is to calculate the Cartan matrix of the
block $\overline{\O}_\xi$.
We will deduce this from the following lemma describing the composition
multiplicities in the Verma supermodules $\overline{M}_\xi(\lambda)$.
This is a reformulation of Corollary~\ref{vermascomp}, but we
will give an alternative proof here using the formal
characters computed in Theorem~\ref{characters}.

\begin{Lemma}\label{crazypre}
For any $\lambda,\kappa \vDash t$,
the Verma multiplicity
$[\overline{M}_{\xi}(\lambda):
\overline{L}_{\xi}(\kappa)]$ is non-zero
if and only if
$\kappa = \lambda +\sum_i \theta_i \alpha_i$ for
$\theta = (\theta_i)_{i \in \Z}$
satisfying
$0 \leq \theta_i \leq \lambda_{i+1}$ for all $i$;
equivalently,
$\lambda = \kappa - \sum_i \theta_i \alpha_i$
for $\theta$ with $0 \leq \theta_i \leq \kappa_i$ for all $i$.
When this holds, we have that
$$
[\overline{M}_{\xi}(\lambda):
\overline{L}_{\xi}(\kappa)]
=
\prod_{i \in \Z}
\tbinom{\lambda_{i+1}}{\theta_i}.
$$
\end{Lemma}

\begin{proof}
In view of Corollary~\ref{injectivityofch}, this follows from the following calculation:
\begin{align*}
\operatorname{ch} \overline{M}_\xi(\lambda)
&=
\chi^{\lambda+\mu}
\prod_{i \in \Z} (1+\chi^{\alpha_i})^{\lambda_{i+1}+\mu_{i+1}}\\
&=
\biggl(\chi^{\lambda}
\prod_{i \in \Z} (1+\chi^{\alpha_i})^{\lambda_{i+1}}\biggr) \biggl(\chi^\mu
\prod_{i \in \Z} (1+\chi^{\alpha_i})^{\mu_{i+1}} \biggr)
\\
&=
\biggl(\sum_{\substack{\theta = (\theta_i)_{i \in \Z}\\0 \leq \theta_i \leq
    \lambda_{i+1}}}
\chi^{\lambda}
\prod_{i \in \Z} \tbinom{\lambda_{i+1}}{\theta_i}(\chi^{\alpha_i})^{\theta_i}
\biggr) \biggl(\chi^\mu \prod_{i \in \Z}
(1+\chi^{\alpha_i})^{\mu_{i+1}}\biggr)\\
&=
\sum_{\substack{\theta = (\theta_i)_{i \in \Z}\\0 \leq \theta_i \leq
    \lambda_{i+1}}}
\biggl(
\prod_{i \in \Z} \tbinom{\lambda_{i+1}}{\theta_i}\biggr)
\operatorname{ch} \overline{L}_\xi\Bigl(\lambda+\sum_{i \in \Z} \theta_i \alpha_i\Bigr).
\end{align*}
Here we have used both of the formulae from Theorem~\ref{characters}.
\end{proof}

\begin{Theorem}\label{crazy}
For $\xi = (\mu,\nu;t) \in \Xi(m|n)$
and any $\lambda,\kappa \vDash t$, the multiplicity
$[\overline{P}_\xi(\lambda)\!:\!\overline{L}_\xi(\kappa)]$ is non-zero if and
only if
$\kappa=\lambda+ \sum_i (\lambda_{i+1}-\rho_{i+1}) \alpha_i$
for $\rho = (\rho_i)_{i \in \Z}$
satisfying $0 \leq \rho_{i+1} \leq \lambda_{i+1}+\min(\lambda_i,\rho_i)$ for all $i$,
in which case
\begin{align*}
[\overline{P}_\xi(\lambda)\!:\!\overline{L}_\xi(\kappa)]
&=
m!\,n!
\hspace{-22mm}
\sum_{\substack{\tau = (\tau_i)_{i \in \Z}\\\max(\lambda_{i+1},\rho_{i+1}) \leq \tau_{i+1} \leq \lambda_{i+1}+\min(\lambda_{i},\rho_i)}}
\hspace{-18mm}
\prod_{i \in \Z}
\frac{\binom{\lambda_{i+1}+\tau_i - \tau_{i+1}}{\tau_{i}-\lambda_{i}}
\binom{\lambda_{i+1}+\tau_i - \tau_{i+1}}{\tau_{i}-\rho_{i}}}
{(\lambda_{i+1}\!+\!\tau_i \!-\! \tau_{i+1})!(\lambda_{i+1}\!+\!\tau_i \!-\!
  \tau_{i+1}\!+\!\gamma_i)!}\,,
\end{align*}
where $\gamma := \mu+\nu$.
Moreover, $[\overline{P}_\xi(\lambda)\!:\!\overline{L}_\xi(\kappa)]
 = [\overline{P}_\xi(\kappa)\!:\!\overline{L}_\xi(\lambda)]$.
\end{Theorem}

\begin{proof}
Let $\tabA := \tabA(\mu,\nu;\kappa)$
and $\tabC := \tabA(\mu,\nu;\lambda)$.
Since these are anti-dominant, Theorem~\ref{wirr} and the exactness of
$H_0$ imply that the multiplicity we are trying to compute is equal to
$[P(\tabC):L(\tabA)]$.
This can be computed by the usual BGG reciprocity formula in the
highest weight category $\O_\Z$:
\begin{equation}\label{upp}
[P(\tabC):L(\tabA)] = \sum_{\tabB \in \xi}
[M(\tabB):L(\tabA)]
[M(\tabB):L(\tabC)].
\end{equation}
In particular, this already establishes the symmetry property at the end of
the statement of the theorem.

For any $\tabB \in \xi$, Theorem~\ref{T:main}
shows that $H_0(M(\tabB))$ is isomorphic to $\overline{M}_\xi(\beta)$
for
$\beta \vDash t$ determined uniquely from $\tabA(\mu,\nu;\beta) \sim \tabB$.
Also, for a given $\beta$, the number of different $\tabB$ satisfying
$\tabB \sim \tabA(\mu,\nu;\beta)$
is
$$
m!n! \Big/  \prod_i (\beta_i+\mu_i)! (\beta_i+\nu_i)!
=
m!n! \Big/  \prod_i \beta_i! (\beta_i+\gamma_i)!.
$$
We deduce from (\ref{upp}) that
$$
[\overline{P}_\xi(\lambda):\overline{L}_\xi(\kappa)]=
m!n!\sum_{\beta \vDash t}
\Big(\prod_i \frac{1}{\beta_i! (\beta_i+\gamma_i)!}
\Big)
[\overline{M}_\xi(\beta):\overline{L}_\xi(\lambda)]
[\overline{M}_\xi(\beta):\overline{L}_\xi(\kappa)].
$$
By Lemma~\ref{crazypre},
$[\overline{M}_\xi(\beta):\overline{L}_\xi(\lambda)]
[\overline{M}_\xi(\beta):\overline{L}_\xi(\kappa)] \neq 0$
only if $\beta  = \kappa -\sum_i \theta_i \alpha_i =  \lambda-
\sum_i
\phi_i \alpha_i$
for $\theta,\phi$ satisfying
$0 \leq \theta_i \leq \kappa_i, 0 \leq \phi_i \leq \lambda_{i}$ for all $i$.
Equivalently, replacing $\phi_i$ by $\tau_{i+1}-\lambda_{i+1}$ and
$\theta_i$ by $\tau_{i+1}-\rho_{i+1}$,
it is non-zero
only if
there exist $\rho = (\rho_i)_{i \in \Z}$ and $\tau = (\tau_i)_{i
  \in \Z}$ such that
$\beta = \lambda + \sum_i (\lambda_{i+1}-\tau_{i+1}) \alpha_i$,
$\kappa = \lambda + \sum_i (\lambda_{i+1}-\rho_{i+1}) \alpha_i$
and $0 \leq \tau_{i+1}-\rho_{i+1} \leq \kappa_i,
0 \leq \tau_{i+1}-\lambda_{i+1} \leq \lambda_i$ for all $i$.
Moreover, when this holds, Lemma~\ref{crazypre} gives that
$$
[\overline{M}_\xi(\beta):\overline{L}_\xi(\lambda)]
[\overline{M}_\xi(\beta):\overline{L}_\xi(\kappa)]
= \prod_i \tbinom{\beta_i}{\tau_i-\lambda_i} \tbinom{\beta_i}{\tau_i-\rho_i}.
$$
In this situation, $\kappa_i = \lambda_{i+1}+\rho_i-\rho_{i+1}$
and $\beta_i = \lambda_i+\tau_i-\tau_{i+1}$, so the inequalities just
recorded may be rewritten as $\max(\lambda_{i+1}, \rho_{i+1}) \leq
\tau_{i+1}
\leq \lambda_{i+1}+\min(\lambda_i,\rho_i)$, and we deduce for
$\kappa = \lambda - \sum_i (\lambda_{i+1}-\rho_{i+1})\alpha_i$ that
$$
[\overline{P}_\xi(\lambda):\overline{L}_\xi(\kappa)]
=
m! n!
\hspace{-22mm}
\sum_{\substack{\tau = (\tau_i)_{i \in \Z}\\\max(\lambda_{i+1},\rho_{i+1}) \leq \tau_{i+1} \leq \lambda_{i+1}+\min(\lambda_{i},\rho_i)}}
\hspace{-18mm}
\prod_i
\frac{\binom{\lambda_{i+1}+\tau_i - \tau_{i+1}}{\tau_{i}-\lambda_{i}}
\binom{\lambda_{i+1}+\tau_i - \tau_{i+1}}{\tau_{i}-\rho_{i}}}
{(\lambda_{i+1}\!+\!\tau_i \!-\! \tau_{i+1})!(\lambda_{i+1}\!+\!\tau_i \!-\!
  \tau_{i+1}\!+\!\gamma_i)!}\,.
$$
Finally, we observe that for this to be non-zero, at least one such
$\tau$ must exist, which exactly requires that
$0 \leq \rho_{i+1} \leq \lambda_{i+1}+\min(\lambda_i,\rho_i)$ for all $i$.
\end{proof}

\begin{Remark}
The multiplicity in Theorem~\ref{crazy}
depends
on $\lambda,\kappa$ and $\gamma = \mu+\nu$, but not directly on $\mu,\nu$.
\end{Remark}

The formula in Theorem~\ref{crazy} is undoubtedly rather cumbersome. Let us
give a small example right away.
Let $(\mu,\nu;t)$ be as in (\ref{palatable}). Since $t = 1$ the Cartan
matrix naturally has its rows and columns indexed by $\Z$.
For $\lambda = \eps_j$,
there are only three possibilities for the composition $\rho$ in
Theorem~\ref{crazy},
namely, $\rho = 0, \eps_j$ or $\eps_j+\eps_{j+1}$. These correspond to
composition factors $\overline{L}_\xi(\kappa)$ of
$\overline{P}_\xi(\lambda)$ with
$\kappa = \eps_{j-1}, \kappa = \eps_j$ and $\kappa = \eps_{j+1}$,
respectively.
Thus, the Cartan matrix is a tri-diagonal matrix. Computing further
from the formula in the theorem one deduces that the Cartan matrix is
$$
\left[
\begin{array}{rrrrrrrrr}
\ddots&\\
&36&18&&&\\
&18&27&9&&\\
&&9&18&9&\\
&&&9&15&6\\
&&&&6&24&18\\
&&&&&18&36\\
&&&&&&&\ddots
\end{array}
\right],
$$
where we have only displayed rows and columns indexed
$1,\dots,6$; the tri-diagonals are constant above and below these
entries.

In the remainder of the subsection, we assume that $t > 0$, and will deduce several more
palatable consequences of the theorem; these will be needed in the
proof of Theorem~\ref{moritatheorem} below.
For $\lambda \vDash t$,
define
$h(\lambda)$ to be the number of compositions $\rho = (\rho_i)_{i
  \in \Z}$ satisfying the
inequalities
\begin{equation} \label{e:rhocond}
0 \leq \rho_{i+1} \leq \lambda_{i+1} + \min(\lambda_i,\rho_i)
\end{equation}
for all
$i \in \Z$. For example, $h(\epsilon_i) = 3$.
Since $\rho_i = \lambda_i = 0$ for $i \ll 0$, this should be
thought of as
a recursive system of inequalities describing
some polytope; we are counting its lattice points.
Theorem~\ref{crazy} tells us that $h(\lambda)$ is the number of
$\kappa \vDash t$ such that $\overline{L}_\xi(\kappa)$ appears as a
composition factor of $\overline{P}_\xi(\lambda)$.

If $\lambda_j = 0$ for some $j$, then we have by the definition that
\begin{equation} \label{e:hsep}
h(\lambda) = h\!\left(\lambda^{\leq j}\right) h\!\left(\lambda^{\geq j}\right),
\end{equation}
where $\lambda^{\leq j}$ (resp. $\lambda^{\geq j}$) is the composition obtained from $\lambda$
by setting all parts in positions $> j$ (resp. $< j$) to zero.  Also
it is clear that $h(\lambda) = h(\mu)$ if $\mu$ is obtained from
$\lambda$ by a translation.
This reduces the problem of computing $h(\lambda)$ to the case that
$\lambda$ is {\em connected}, i.e.\ it is of the form
$$
h(\lambda_1,\dots,\lambda_r) :=
h((\dots,0,\lambda_1,\dots,\lambda_r,0,\dots))$$
for some
$r$ and
$\lambda_1,\dots,\lambda_r > 0$.

\begin{Lemma}\label{step1}
For any $\lambda \vDash t$, we have that
$h(\lambda) \geq \binom{t+2}{2}$, with equality if and only
if $\lambda = t \eps_i$ for some $i \in \Z$.
\end{Lemma}

\begin{proof}
It is easy to check that $h(t \eps_i) = \binom{t+2}{2}$.
Conversely,
we must show that $h(\lambda) > \binom{t+2}{2}$ whenever it has at
least two non-zero parts.
If $\lambda$ is not connected,
the proof is easily computed by induction on $t$,
using the elementary inequality
$\binom{t'+2}{2}\binom{t''+2}{2} > \binom{t+2}{2}$
if $t',t'' \geq 1$ satisfy $t'+t''=t$.
If $\lambda$ is connected, let $\lambda_k,\lambda_{k+1}$
be its rightmost two non-zero parts, so
$\lambda = (\dots,\lambda_{k-1},\lambda_k,\lambda_{k+1})$.
Then the conclusion follows easily from the claim that
$$
h(\dots,\lambda_{k-1},\lambda_k,\lambda_{k+1})>
h(\dots,\lambda_{k-1},\lambda_k+\lambda_{k+1}).
$$
This can be proved by showing
for each choice of the entries of $\rho$ up to $\rho_{k-1}$ that there are more ways
of extending this to a sequence satisfying \eqref{e:rhocond} for $\lambda = (\dots,\lambda_{k-1},\lambda_k,\lambda_{k+1},\dots)$
than for $\lambda' = (\dots,\lambda_{k-1},\lambda_k+\lambda_{k+1},\dots)$.
\end{proof}

\begin{Remark}
We also expect for $\lambda \vDash t$ that $h(\lambda) \leq
3^t$, with equality if and only if $\lambda$ is generic in the sense
that it consists of isolated $1$'s. We won't need this observation here,
but note that $h(\lambda) = 3^t$ for generic $\lambda$ follows from
$h(\eps_i) = 3$ and \eqref{e:hsep}.
\end{Remark}

\begin{Lemma}\label{step2}
The space
$\Hom_{\overline{\O}_\xi}(\overline{P}_\xi(t \eps_j),\overline{P}_\xi(t
\eps_i))$ is non-zero
if and only if $|i-j| \leq 1$.
\end{Lemma}

\begin{proof}
This is the multiplicity
computed in Theorem~\ref{crazy}
for $\lambda = t \eps_i$ and $\kappa = t \eps_j$.
Since the Cartan matrix is symmetric, we may assume that $i < j$.
From the equation $\lambda-\kappa = \sum_i (\rho_i -\lambda_i)
\alpha_{i-1}$,
we deduce that we must have $\rho_i=\rho_{i+1}=\cdots=\rho_j=t$
and all other parts zero.
But this contradicts the system of inequalities $\rho_{k+1} \leq \lambda_{k+1} +
\min(\lambda_k,\rho_k)$
unless we actually have that $j=i+1$.
\end{proof}

\begin{Lemma}\label{step3}
Let $\gamma := \mu+\nu$.
For each $i \in \Z$ we have that
$$
\dim \End_{\overline{\O}_\xi}(\overline{P}_\xi(t \eps_i))
=
\frac{m!n!}{t!\prod_{j} \gamma_j!}
\sum_{r=0}^t
\binom{t}{r}
\frac{\gamma_i! \gamma_{i+1}!}{(\gamma_{i}+t-r)!(\gamma_{i+1}+r)!}.
$$
This is equal to
$\frac{m!n!}{(t!)^2 \prod_j \gamma_j!} \binom{2t}{t}$
whenever $\gamma_i=\gamma_{i+1} = 0$.
\end{Lemma}

\begin{proof}
The first formula is easily derived from Theorem~\ref{crazy};
in the summation over $\tau$ there, one just has
$\tau$'s with $\tau_i = t, \tau_{i+1} = r$ for $0 \leq r \leq t$, and
all other parts equal to zero.
To deduce the second formula, use $\sum_{r=0}^t \binom{t}{r}^2 = \binom{2t}{t}$.
\end{proof}

\subsection{An idempotented form for
  \texorpdfstring{$W$}{W}}\label{lastss}
In general, the blocks $\overline{\O}_\xi$
of the category $\overline{\O}_\Z$ have infinitely isomorphism classes
of irreducible objects. In such situations, it is often appropriate to consider
locally unital rather than unital algebras.
Here, by a {\em locally unital algebra}, we mean an associative algebra $A$
equipped with a distinguished system of mutually orthogonal
idempotents
$\{1_i\}_{i \in I}$ such that $A = \bigoplus_{i,j \in I} 1_i A 1_j$.
By a (left) {module} $M$ over a locally unital algebra $A$,
we always mean a
module as usual which is locally unital in the sense that $M =
\bigoplus_{i \in I} 1_i M$.
Let $A\lmof$ be the category of all such modules with $\dim M <
\infty$.
In this subsection,
we
construct a locally unital
algebra $W_\xi$ whose module category is equivalent to
$\overline{\O}_\xi$.
In the next subsection, we will show for maximally atypical blocks
that $W_\xi$ is isomorphic to
the locally unital endomorphism algebra of a minimal projective
generating family in $\overline{\O}_\xi$.
For further discussion of our motivation, see \cite[$\S$4]{Bsurvey}.

To define $W_\xi$, we must first pass from $W$ to an idempotented form
$\dot W$.
By definition, this is the locally unital algebra
with distinguished idempotents $\{1_\eta\}_{\eta\vDash m}$,
generators
\begin{align*}
c^{(r)} 1_\eta &\in 1_\eta \dot W 1_\eta
&&\text{for $\eta \vDash m$ and $r \geq 0$},\\
d^{(r)} 1_\eta &\in 1_\eta \dot W 1_\eta
&&\text{for $\eta \vDash m$ and $r \geq 0$},\\
f_i^{(r)} 1_\eta &\in 1_{\eta+\alpha_i} \dot W 1_\eta
&&\text{for $\eta \vDash m,\:
r > s_-$ and $i \in \Z$ such that $\eta_{i+1} > 0$,}\\
e_i^{(r)} 1_\eta &\in 1_{\eta-\alpha_i} \dot W 1_\eta
&&\text{for $\eta \vDash m,\:r > s_+$ and $i \in \Z$ such that $\eta_{i}>0$,}
\end{align*}
and certain relations.
In order to write these down, we need a couple of conventions.
We interpret the following currently undefined expressions as zero:
$1_\eta$, $c^{(r)} 1_\eta$ and $d^{(r)} 1_\eta$
if $\eta \not\vDash m$;
$e_i^{(r)} 1_\eta$ if either
$\eta \not\vDash m$ or
$\eta-\alpha_i \not\vDash m$;
$f^{(r)}_i 1_\eta$ if either $\eta \not\vDash m$
or $\eta+\alpha_i \not\vDash m$.
This means
that for given $\eta \vDash m$, we have $e^{(r)}_i 1_\eta = f^{(r)}_i 1_\eta = 0$ for all but
finitely many $i$.
Also
 we will omit idempotents from the middles of monomials when they are
clear from the context.
Then, the relations are as follows:
\begin{align}
c^{(0)} 1_\eta = d^{(0)} 1_\eta = 1_\eta,\quad
d^{(1)} 1_\eta &= \sum_{i} i \eta_i 1_\eta,\quad
d^{(r)} 1_\eta = 0
\text{ for $r > m$},\\
c^{(r)}1_\eta\text{ is central},\qquad
 d^{(r)} d^{(s)} 1_\eta
&=
 d^{(s)} d^{(r)}1_{\eta},\\
d^{(r)}
e^{(s)}_i 1_\eta
-
 e^{(s)}_i
 d^{(r)}1_\eta
&=
\sum_{a=0}^{r-1} d^{(a)} e^{(r+s-1-a)}_i  1_\eta,\\
d^{(r)} f^{(s)}_i 1_\eta -
f^{(s)}_i d^{(r)} 1_\eta
&=-
\sum_{a=0}^{r-1} f^{(r+s-1-a)}_i d^{(a)} 1_\eta,\\
e^{(r)}_ie^{(s)}_j 1_\eta+
e^{(r)}_j e^{(s)}_i 1_\eta
+&e^{(s)}_i e^{(r)}_j
1_\eta
+e^{(s)}_j e^{(r)}_i
1_\eta=0
,\\
f^{(r)}_i f^{(s)}_j 1_\eta+
f^{(r)}_j f^{(s)}_i 1_\eta
+&f^{(s)}_i f^{(r)}_j
1_\eta
+f^{(s)}_j f^{(r)}_i
1_\eta=0,\\
e^{(r)}_i  f^{(s)}_j 1_\eta
+
f^{(s)}_j e^{(r)}_i 1_\eta
&= 0\:\text{for $i \neq j$,}\label{fourth}\\
\sum_i (e^{(r)}_i  f^{(s)}_i 1_\eta
+
f^{(s)}_i e^{(r)}_i 1_\eta)
&=
c^{(r+s-1)} 1_\eta.\label{fifth}
\end{align}

We leave it as an exercise for the reader to check in the degenerate
case $m=0$ that $\dot W = \C[c^{(1)},\dots,c^{(n)}]$, i.e.\ it is a polynomial
algebra in $n$ variables.
The case $m=1$ is also particularly easy to understand;
in this case, we
let $1_i := 1_{\eps_i},
e 1_i := e_i^{(s_++1)} 1_{\eps_i},
f 1_i := f_{i-1}^{(s_-+1)} 1_{\eps_i}$
and $z 1_i := c^{(n)} 1_{\eps_i}$ for short.

\begin{Lemma}\label{swimming}
In the case $m=1$, the algebra $\dot W$ is
$\C[c^{(1)},\dots,c^{(n-1)}] \otimes T$
where $T$ is the path algebra of the infinite
quiver
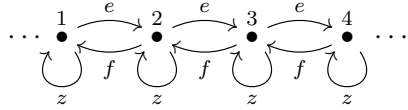
\begin{displaymath}
\:\cdots
\xymatrix{
\ar@(dr,dl)[]^{z}
\stackrel{1}{\bullet}
\ar@/^/[r]^{e}&\ar@/^/[l]^{f}
\ar@(dr,dl)[]^{z}
\stackrel{2}{\bullet}
\ar@/^/[r]^{e}&\ar@/^/[l]^{f}
\ar@(dr,dl)[]^{z}
\stackrel{3}{\bullet}
\ar@/^/[r]^{e}&\ar@/^/[l]^{f}
\ar@(dr,dl)[]^{z}
\stackrel{4}{\bullet}}
\:\cdots
\end{displaymath}
subject to the relations $e^2 1_i = f^2 1_i = 0$ and
$ef1_i + fe 1_i = z 1_i$ for all $i \in \Z$.
\end{Lemma}

\begin{proof}
It is easy to check from the defining relations for $\dot W$ that
there
is a locally unital algebra homomorphism
$f:\C[c^{(1)},\dots,c^{(n-1)}] \otimes T \rightarrow \dot W$
sending $1_i, c^{(r)} 1_i, e 1_i, f 1_i$ and $z 1_i$ to
the corresponding
elements of $\dot W$.
To show this is an isomorphism, we
define
$f^{-1}:\dot W \rightarrow
\C[c^{(1)},\dots,c^{(n-1)}] \otimes T$
by sending
\begin{align*}
1_{\eps_i}&\mapsto 1_i,\\
c^{(r)} 1_{\eps_i} &\mapsto
\left\{
\begin{array}{ll}
c^{(r)} 1_i&
\text{if $r < n$,}\\
(-1)^{r-n} ((i+1)^{r-n+1}-i^{r-n+1}) z 1_i&\text{if $r \geq n$,}
\end{array}\right.\\
d^{(r)} 1_{\eps_i} & \mapsto
\delta_{r,1} i 1_{i},\\
e_j^{(s_++r)} 1_{\eps_i} & \mapsto \delta_{i,j} (-i-1)^{r-1}
e 1_i,\\
f_j^{(s_-+r)} 1_{\eps_i} & \mapsto \delta_{i,j-1} (-i)^{r-1}
f 1_i,
\end{align*}
for each $r \geq 1$ and $i,j \in \Z$.
This is well defined by another (longer) relation check.
Finally, it is obvious that $f^{-1} \circ f = \operatorname{id}$.
To see that $f \circ f^{-1} = \operatorname{id}$, note from the
relations that the
elements
$1_i, e 1_i, f 1_i$ and $c^{(1)} 1_i,\dots,c^{(n)} 1_i$ for all $i \in \Z$
already suffice to generate $\dot W$.
\end{proof}

\begin{Lemma}\label{factors}
The following two categories may be identified:
\begin{enumerate}
\item
The full subcategory of $W\lmof$
consisting of all modules $M$ such that
$M$ is equal to the direct sum
$\bigoplus_{\eta \vDash m} M_\eta$
of its weight spaces from
(\ref{thewtspace}) and the endomorphism defined by the action
of $d_1^{(1)}$ is diagonalizable;
\item
The full subcategory of $\dot W\lmof$
consisting of all modules $M$ such that
$d^{(r)}1_\eta - e_r(\eta) 1_\eta$
acts nilpotently for all $\eta \vDash m$ and $r > 1$.
\end{enumerate}
\end{Lemma}

\begin{proof}
Take $M$ belonging to the category (1).
We make it into a $\dot W$-module by declaring that
\begin{itemize}
\item
$1_\eta$ acts as the projection $\pr_\eta$
along the weight space decomposition;
\item
$c^{(r)} 1_\eta$ acts as $c^{(r)} \circ \pr_\eta$ and $d^{(r)} 1_\eta$ acts as $ d_1^{(r)} \circ \pr_\eta$; and
\item $e^{(r)}_i 1_\eta$ acts as $\pr_{\eta-\alpha_i} \circ e^{(r)}
\circ \pr_\eta$ and
$f^{(r)}_i1_\eta$ acts as $\pr_{\eta+\alpha_i} \circ f^{(r)}
\circ \pr_\eta$.
\end{itemize}
To see that this makes sense,
we need to verify that the defining relations of $\dot W$ are satisfied. This
follows from the relations for $W$ in Theorem~\ref{wpres}.
For example, to check (\ref{fourth})--(\ref{fifth}),
we have by the definition of the action of the generators of $\dot W$ on $v \in M_\eta$ that
\begin{align*}
\sum_{i,j} (e_i^{(r)} &f_j^{(s)}
1_\eta +
f_j^{(s)} e_i^{(r)} 1_\eta)v\\
&=
\sum_{i,j} (\pr_{\eta+\alpha_j-\alpha_i} \circ e^{(r)} \circ \pr_{\eta+\alpha_j}
\circ f^{(s)} + \pr_{\eta-\alpha_i+\alpha_j} \circ f^{(s)} \circ \pr_{\eta-\alpha_i}
\circ e^{(r)})v\\
&=(e^{(r)} f^{(s)} + f^{(s)} e^{(r)})v = \delta_{i,j} c^{(r+s-1)} v =
c^{(r+s-1)} 1_\eta v,
\end{align*}
using Lemma~\ref{jumpers} with $k=1$ for the second equality.
Applying $\pr_{\eta+\alpha_j-\alpha_i}$ to both sides, this establishes
(\ref{fourth}) when $i \neq j$ and (\ref{fifth}) when $i=j$.

Conversely, take $M$ belonging to the category (2). Then
we make it into a $W$-module by declaring that
\begin{itemize}
\item
$c^{(r)}, d_1^{(r)}, e^{(r)}$ and $f^{(r)}$
act on $1_\eta M$ as $c^{(r)} 1_\eta, d^{(r)} 1_\eta,
\sum_i e_i^{(r)} 1_\eta$ and $\sum_i f_i^{(r)} 1_\eta$;
\item
$d_2^{(r)}$ acts via the formula $d_2^{(r)} = \sum_{s=0}^r d_1^{(s)} c^{(r-s)}$
from (\ref{ctilde}).
\end{itemize}
Then one needs to verify that the relations for $W$ from
Theorem~\ref{wpres} are satisfied.
Moreover, we have that $M_\eta = 1_\eta M$.

Finally, the two constructions just explained are inverses of
each other. Again, this uses the $k=1$ case of Lemma~\ref{jumpers}.
\end{proof}

\begin{Theorem}\label{landing}
For $\xi = (\mu,\nu;t) \in \Xi(m|n)$, the category $\overline{\O}_\xi$
is isomorphic to the category $W_\xi\lmof$, where
$$
W_\xi := \dot W \Big/\bigcap_{\lambda \vDash t} \operatorname{Ann}_{\dot
  W}(\overline{P}_\xi(\lambda)).
$$
This is a locally unital algebra with distinguished
idempotents $\{1_\eta\}_{\eta\vDash m}$ that are the images of the
ones in $\dot W$.
Moreover:
\begin{enumerate}
\item
All of the left ideals $W_\xi 1_\eta$ and right ideals $1_\upsilon
W_\xi$ are finite-dimensional.
\item
We have that $c^{(1)} 1_\eta = \sum_i i(\nu_i-\mu_i) 1_\eta$ in $W_\xi$.
\end{enumerate}
\end{Theorem}

\begin{proof}
Set
$V := \bigoplus_{\lambda \vDash t} \overline{P}_\xi(\lambda)$.
For each $\eta \vDash m$, there are only finitely many $\kappa \vDash
t$ such that $1_\eta \overline{L}_\xi(\kappa) \neq 0$. This follows from
Theorem~\ref{characters}: we must have that $\kappa = \eta - \mu - \sum_i \theta_i
\alpha_i$ for $\theta = (\theta_i)_{i \in \Z}$ with $0 \leq \theta_i
\leq \mu_{i+1}$, and there are only finitely many such $\theta$'s.
Moreover, for each $\kappa \vDash t$, there are only finitely many
$\lambda \vDash t$ such that
$[\overline{P}_\xi(\lambda):\overline{L}_\xi(\kappa)] \neq 0$, as is
clear from Theorem~\ref{crazy}.
Hence, there are only finitely many $\lambda \vDash t$ such that
$1_\eta \overline{P}_\xi(\lambda) \neq 0$. Thus, we have shown that
all of the weight spaces $1_\eta V$ of
the $\dot W$-module $V$ are finite-dimensional.

Consider the locally unital algebra
$$
E := \bigoplus_{\eta,\upsilon\vDash m}
\Hom_\C(1_\eta V, 1_\upsilon V)
$$
with multiplication coming from the usual composition.
This is a simple locally unital matrix algebra with unique (up to isomorphism) irreducible module $V$.
The representation of $\dot W$ on $V$ defines
a locally unital algebra homomorphism $\rho:\dot W \rightarrow E$
sending $a \in 1_\upsilon \dot W 1_\eta$ to the linear map
$1_\eta V \rightarrow 1_\upsilon V$ defined by left multiplication by
$a$.
We have that $\ker \rho = \bigcap_{\lambda \vDash t} \operatorname{Ann}_{\dot
  W} \overline{P}_\xi(\lambda)$, hence, $\rho$ induces an embedding
$W_\xi \hookrightarrow E$.
Since each $1_\upsilon E 1_\eta =
\Hom_\C(1_\eta V, 1_\upsilon V)
$ is finite-dimensional,
we deduce
that each $1_\upsilon W_\xi 1_\eta$ is finite-dimensional.
Also $1_\upsilon W_\xi 1_\eta$ is non-zero
only if there exists $\lambda \vDash t$ such that $1_\upsilon
\overline{P}_\xi(\lambda) \neq 0 \neq 1_\eta
\overline{P}_\xi(\lambda)$.
For fixed $\eta$, there are only finitely many such $\lambda$, hence,
only finitely many such $\upsilon$.
This shows that $W_\xi 1_\eta$ is finite-dimensional. Similarly,
so is $1_\upsilon W_\xi$, and we have established (1).

Now we explain how to identify $\overline{\O}_\xi$ with $W_\xi\lmof$.
Given any $M \in \overline{\O}_\xi$, we can forget the $\Z/2$-grading
then apply Lemma~\ref{factors} to view it as a $\dot W$-module,
using also Lemma~\ref{bookkeeping}
and the last part of Theorem~\ref{characters}.
This defines a functor from $\overline{\O}_\xi$ to
the full
subcategory of $\dot W\lmof$ consisting
of subquotients of finite direct sums of the modules
$\{\overline{P}_\xi(\lambda)\}_{\lambda \vDash t}$.
Each $\overline{P}_\xi(\lambda)$ factors through the quotient $W_\xi$,
so the latter category may also be described as
the full subcategory of $W_\xi\lmof$ consisting
of subquotients of finite direct sums of the modules
$\{\overline{P}_\xi(\lambda)\}_{\lambda \vDash t}$.
In fact, this defines an isomorphism of categories, since the
$\Z/2$-grading on $M$
can be recovered uniquely thanks to Remark~\ref{recoering}.

It just remains to observe that {\em every} finite-dimensional
$W_\xi$-module belongs to the full subcategory of $W_\xi\lmof$ consisting
of subquotients of finite direct sums of the modules
$\{\overline{P}_\xi(\lambda)\}_{\lambda \vDash t}$.
To see this, note that any $M \in W_\xi\lmof$ is a quotient of a finite direct
sum of the projective modules $W_\xi 1_\eta$ for $\eta \vDash m$.
Moreover, $W_\xi 1_\eta \hookrightarrow E 1_\eta$,
which is a direct sum of copies of $V$,
so that $W_\xi 1_\eta$ embeds into a (possibly infinite) direct sum of
the modules $\overline{P}_\xi(\lambda)$.
In fact, it embeds into a finite direct sum of these modules
since it is finite-dimensional, as is each
$\overline{P}_\xi(\lambda)$.

To establish (2), note that $c^{(1)} = d_2^{(1)} - d_1^{(1)}$,
so it acts diagonalizably on any object of $\overline{\O}_\xi$
thanks to Lemma~\ref{bookkeeping}.
Moreover,
$d_2^{(1)} - d_1^{(1)}$
acts on $\overline{L}(\tabA)$ as
$b(\tabA)-a(\tabA) = \sum_i i (\nu_i-\mu_i)$
for any $\tabA \in \xi^\circ$.
Thus, $c^{(1)} - \sum_i i (\nu_i - \mu_i)$ annihilates all
$\overline{P}_\xi(\lambda)$, and (2) follows.
\end{proof}

\begin{Remark}
Here is a slightly different construction of the locally unital
algebra $W_\xi$. Given any finite subset $X \subset \xi$,
the $W$-module $\bigoplus_{\lambda \in X} \overline{P}_\xi(\lambda)$
is finite-dimensional. Hence,
$$
W_X := W \Big/ \bigcap_{\lambda \in X}
\operatorname{Ann}_W(\overline{P}_\xi(\lambda))
$$
is a finite-dimensional algebra.
Each $W_X$ possesses a distinguished family of idempotents
$\{1_\eta\}_{\eta\vDash m}$
such that $W_X = \bigoplus_{\upsilon,\eta\vDash m} 1_\upsilon W_X
1_\eta$, namely, $1_\eta$ is the primitive idempotent in
the finite-dimensional commutative subalgebra of $W_X$
generated by
$d_1^{(1)},\dots,d_m^{(1)}$
that projects any module onto its
$\eta$-weight space.
Then
$W_\xi$ is the inverse limit $\varprojlim_X W_X$ over all finite
subsets $X$ of $\xi$, taking the inverse limit in the category of
locally unital algebras with idempotents indexed by compositions of $m$.
\end{Remark}

\begin{Remark}\label{ball1}
The special case $m=n=1$ is particularly trivial.
If $m=n=t=1$ then
$W_\xi$ is the algebra $T$ from Lemma~\ref{swimming} subject to the
additional relations $z 1_i = 0$ for all $i \in \Z$.
If $m=n=1$ and $t=0$,
then we have that $\mu = \eps_i$ and $\nu = \eps_j$ for $i \neq j$,
and
$W_\xi$ is the algebra $T$  with the
additional relations $z 1_i  = (j-i) 1_i, z 1_{i-1} = (j-i) 1_{i-1}$,
and $1_k = 0$ for $k \neq i,i-1$. (In this case, $W_\xi\cong M_2(\C)$.)
\end{Remark}

\subsection{Graded lifts and the Soergel algebra of a block}\label{coffee}
Throughout the subsection, we fix $\xi = (\mu,\nu;t) \in
\Xi(m|n)$
and set $\gamma := \mu+\nu$ for short.
Let
\begin{equation}
\AA_\xi := \bigoplus_{\tabA,\tabB \in \xi} \Hom_{\g}(P(\tabA), P(\tabB)),
\end{equation}
viewed as an algebra with multiplication that is the {opposite} of
composition.
Note $\AA_\xi$ is a locally unital in the sense introduced at the
start of \S\ref{lastss}, with
distinguished idempotents $\{1_\tabA\}_{\tabA \in \xi}$ coming
from the identity endomorphisms of each $P(\tabA)$.
By standard theory, the functor
\begin{equation}\label{terrible}
\bigoplus_{\tabA \in \xi} \Hom_{\g}(P(\tabA),-):
\O_\xi \rightarrow \AA_\xi\lmof
\end{equation}
is an equivalence of categories.
Since we have that
$1_\tabA \AA_\xi 1_\tabB = \Hom_{\g}(P(\tabA), P(\tabB))$,
this functor sends $P(\tabB)$ to the left ideal $\AA_\xi 1_\tabB$.

By \cite[Theorem 5.26]{BLW} (plus \cite[Theorem 3.10]{BLW}), the basic algebra $\AA_\xi$
admits a positive grading $\AA_\xi = \bigoplus_{d \geq 0} (\AA_\xi)_d$ making it into a {Koszul
algebra}.
In view of (\ref{terrible}),
this means that we can introduce a {\em graded lift} of the block $\O_\xi$, namely,
the category $\AA_\xi\grlmof$ of finite-dimensional graded
$\AA_\xi$-modules. This is entirely analogous to the situation in
\cite{Soergel}, where Soergel introduced a graded lift of the category $\O$ for a semisimple
Lie algebra.
For further discussion, see
\cite[$\S\S$5.2--5.5]{BLW}. In particular, \cite[Theorem 5.11]{BLW} shows
that categorical action on $\O_\Z$ discussed in \S\ref{Catact} also
lifts to the graded setting.
Note finally by \cite[Corollary 2.5.2]{BGS} (which extends obviously to the present locally
unital setting) that the Koszul grading on $\AA_\xi$ is unique up to
automorphism. This means that the grading is {canonical}, so it can be used to refine the invariants of
blocks computed already in Theorem~\ref{crazy}; see
Theorem~\ref{crazier} below.

Recalling next that the indecomposable projective objects in the quotient category
$\overline{\O}_\xi$ are
the $W$-supermodules $\left\{\overline{P}_\xi(\lambda)\right\}_{\lambda \vDash
t}$,
we can also consider the locally unital algebra
\begin{equation}\label{cxi}
\BB_\xi := \bigoplus_{\kappa,\lambda \vDash t}
\Hom_{W}(\overline{P}_\xi(\kappa),
\overline{P}_\xi(\lambda)),
\end{equation}
again with
multiplication that is opposite of composition.
Its distinguished idempotents are denoted
$\{1_\lambda\}_{\lambda\vDash t}$, so
that
$1_\kappa \BB_\xi 1_\lambda = \Hom_{W}\left(\overline{P}_\xi(\kappa),
\overline{P}_\xi(\lambda)\right)$.
In view of the double centralizer property of
Corollary~\ref{controls},
we can identify $\BB_\xi$ with the subalgebra $\bigoplus_{\kappa,
  \lambda \vDash t} 1_{\tabA(\mu,\nu;\kappa)} \AA_\xi 1_{\tabA(\mu,\nu;\lambda)}$ of
$\AA_\xi$.
In particular, $\BB_\xi$ inherits
a positive grading $\BB_\xi = \bigoplus_{d \geq 0} (\BB_\xi)_d$
from the Koszul grading on $\AA_\xi$.
We call the graded algebra $\BB_\xi$ the {\em Soergel algebra} of the block
$\O_\xi$.
Just like in (\ref{terrible}), there is an equivalence of categories
\begin{equation}\label{terrible2}
\bigoplus_{\kappa \vDash t} \Hom_{W}\!\left(\overline{P}_\xi(\kappa),-\right):
\overline{\O}_\xi \rightarrow \BB_\xi\lmof,
\end{equation}
so that the category $\BB_\xi\grlmof$ is a graded lift of
$\overline{\O}_\xi$.
The algebra $\BB_\xi$ is Morita equivalent to the algebra
$W_\xi$ from Theorem~\ref{landing}.
The following theorem shows that
these two algebras coincide for maximally atypical blocks.
It gives us hope that the algebra $\BB_\xi$ can be described
in these cases as the path algebra of an infinite quiver with
explicit relations.

\begin{Theorem}\label{moreispossible}
For any block $\xi = (0,\nu;m) \in \Xi(m|n)$ of maximal atypicality,
there is an isomorphism
$W_\xi \stackrel{\sim}{\rightarrow} \BB_\xi$ such that $1_\lambda \mapsto
1_\lambda$
for each $\lambda \vDash m$.
\end{Theorem}

\begin{proof}
Maximal atypicality implies that all of the irreducible $W$-modules
in $\overline{\O}_\xi$ are one-dimensional. By Theorem~\ref{landing}, these are the
irreducible $W_\xi$-modules.
Hence, $W_\xi$ is a locally
unital basic algebra such that $W_\xi\lmof$ is isomorphic to
$\overline{\O}_\xi$.
Since the idempotent $1_\lambda \in W_\xi$ acts as the identity on
$\overline{L}_\xi(\lambda)$ and as zero on all other
$\overline{L}_\xi(\kappa)$'s,
it is a primitive idempotent, and
the left ideal $W_\xi 1_\lambda$ is
isomorphic to the projective cover
$\overline{P}_\xi(\lambda)$ of $\overline{L}_\xi(\lambda)$.
Then comparing with (\ref{cxi}), we get that
$$
\BB_\xi \cong \bigoplus_{\kappa,\lambda\vDash m}
\Hom_{W_\xi}(W_\xi 1_\kappa, W_\xi 1_\lambda)
\cong \bigoplus_{\kappa,\lambda \vDash m}
1_\kappa W_\xi 1_\lambda = W_\xi.
$$
This isomorphism sends $1_\lambda \mapsto 1_\lambda$ for each $\lambda
\vDash m$.
\end{proof}

\begin{Remark}\label{ball2}
Let $\xi$ be as in Theorem~\ref{moreispossible}, so that $B_\xi \cong W_\xi$.
If $m=n=1$, we described this algebra already in
Remark~\ref{ball1}.
If $m=1 < n$, using also Theorem~\ref{landing},
we see that $B_\xi$ is a quotient of
$\C[c^{(2)},\dots,c^{(n-1)}]
\otimes T$, where $T$ is as in
Lemma~\ref{swimming}.
Setting $x_i := ef 1_i$ and $y_i := fe 1_i$, we have that
$1_i T 1_i \cong \C[x_i,y_i]$. It follows that
$1_{\eps_i} B_\xi 1_{\eps_i}$ is a quotient of the polynomial
algebra $\C[c^{(2)},\dots,c^{(n-1)}, x_i, y_i]$.
\end{Remark}

Our next theorem computes the graded dimensions of
the spaces $1_\kappa B_\xi 1_\lambda$. It is a graded analog of Theorem~\ref{crazy}.
To formulate it,
for a positively graded vector space $V = \bigoplus_{n \geq 0} V_n$, we let
$\qdim V := \sum_{n \geq 0} (\dim V_n) q^n$ where $q$ is an indeterminate.
Let $[n]$ be the quantum integer $(q^n -
q^{-n})/ (q-q^{-1})$, let $[n]!$ be the corresponding quantum factorial, and let
$\sqbinom{n}{r}$ be the quantum binomial coefficient.

\begin{Theorem}\label{crazier}
For any $\lambda,\kappa \vDash t$,
the space $1_\kappa \BB_\xi 1_\lambda$
is non-zero if and
only if
$\kappa=\lambda+ \sum_i (\lambda_{i+1}-\rho_{i+1}) \alpha_i$
for a composition $\rho$
with
$0 \leq \rho_{i+1} \leq \lambda_{i+1}+\min(\lambda_{i},\rho_{i})$ for all $i$,
in which case
$$
\qdim
1_\kappa \BB_\xi 1_\lambda
=
[m]![n]!
\hspace{-23mm}
\sum_{\substack{\tau = (\tau_i)_{i \in \Z}
\\
\max(\lambda_{i+1},\rho_{i+1})\leq \tau_{i+1} \leq
\lambda_{i+1}+\min(\lambda_i,\rho_i)
}}
\hspace{-22mm}
 q^{s(\tau)}
\prod_i \frac{\sqbinom{\lambda_{i+1}+\tau_i-\tau_{i+1}}{\tau_{i}-\lambda_{i}}
\sqbinom{\lambda_{i+1}+\tau_i-\tau_{i+1}}{\tau_{i}-\rho_{i}}}
{[\lambda_{i+1}+\tau_i-\tau_{i+1}]!
[\lambda_{i+1}+\tau_i-\tau_{i+1}\!+\!\gamma_i]!}\,,
$$
where
\begin{multline*}
s(\tau) := \binom{m}{2}+\binom{n}{2}
+
\sum_i(2\tau_i-\lambda_i-\rho_i)(\lambda_{i+1}+\tau_i-\tau_{i+1}+\gamma_{i})\\
-\sum_i \binom{\lambda_{i+1}+\tau_i-\tau_{i+1}}{2}
-\sum_i \binom{\lambda_{i+1}+\tau_i-\tau_{i+1}+\gamma_i}{2}.
\end{multline*}
\end{Theorem}

\begin{proof}
See Appendix \ref{appendixb}.
\end{proof}

\begin{Corollary}\label{pling}
For all
compositions
$\lambda\vDash t$ which are generic in the sense that $\lambda_i \neq 0
\Rightarrow \lambda_i +\lambda_{i+1}+\gamma_i + \gamma_{i+1} =
1$, we have that
$$
\qdim
1_\lambda \BB_\xi 1_\lambda
=
q^{\binom{m}{2}+\binom{n}{2}-\sum_i \binom{\gamma_i}{2}}
(1+q^2)^t\:
[m]! [n]! \big/ \textstyle\prod_i [\gamma_i]!.
$$
\end{Corollary}

\begin{proof}
We apply Theorem~\ref{crazier} with $\kappa =\lambda$, hence, $\rho_i =
0$ for all $i$.
Letting
$I := \{i \in \Z\:|\:\lambda_{i-1}=1\}$,
the summation is over the compositions
$\tau^J$ for $J \subseteq I$ defined from
$\tau^J_i := \lambda_i + 1$ if $i \in J$,
 and
$\tau^J_i := \lambda_i$ otherwise.
We have that $s(\tau^J) =
\binom{m}{2}+\binom{n}{2}-\sum_i \binom{\gamma_i}{2}
+2|J|$, hence,
$\sum_{J \subseteq I} q^{s(\tau^J)} =
q^{\binom{m}{2}+\binom{n}{2}-\sum_i \binom{\gamma_i}{2}} (1+q^2)^t$,
while the big product is always equal to $1 \big/ \prod_i [\gamma_i]!$.
\end{proof}

By Lemma~\ref{projc}, the basic algebra $\BB_\xi$ is a Frobenius
algebra. It is also graded and indecomposable.
It is a general fact about such algebras that the top degrees
of the endomorphism algebras of the indecomposable projective objects
are all the
same (and these endomorphism algebras are one-dimensional in this degree); see
e.g.\ \cite[Proposition 5.18]{Rougrad}.
The following describes this top degree explicitly.

\begin{Corollary}\label{zcd}
Let $d := m^2+n^2 - \sum_i \gamma_i^2$.
For every $\lambda \vDash t$, we have that
$\dim (1_\lambda \BB_\xi 1_\lambda)_d = 1$
and
$(1_\lambda \BB_\xi 1_\lambda)_{d'} = 0$ for all $d' > d$.
\end{Corollary}

\begin{proof}
In view of the remarks just made, it suffices to
compute the top degree of $1_\lambda \BB_\xi 1_\lambda$
for generic $\lambda$, which is easily done
using Corollary~\ref{pling}.
(One can also check this directly for arbitrary $\lambda$; in general,
the monomial of top degree
in the
polynomial $\qdim 1_\lambda \BB_\xi 1_\lambda$ as computed by
Theorem~\ref{crazier} comes from the summand with $\tau_{i+1} =
\lambda_{i+1}+\lambda_i$
for each $i$.)
\end{proof}

\begin{Conjecture}
The algebra $\BB_\xi$ is a graded symmetric Frobenius algebra.
\end{Conjecture}

The last important algebra that we associate to the block $\O_\xi$ is its
{\em center}
$\CC_\xi$, i.e.\ the endomorphism
algebra of the identity functor $\operatorname{Id}:\O_\xi \rightarrow
\O_\xi$.
The double centralizer property implies that $\CC_\xi$ is also the center of
the quotient category $\overline{\O}_\xi$.
It can be recovered from the Soergel algebra
$\BB_\xi$. To explain this, we view elements
of $\BB_\xi$ as infinite
matrices of the form
$x = (x_{\kappa,\lambda})_{\kappa,\lambda \vDash t}$ for
$x_{\kappa,\lambda} \in
1_\kappa \BB_\xi 1_\lambda$, all but finitely many
of which are zero.
If we drop this finiteness condition, we obtain a
completion
$\widehat{\BB}_\xi$ of this algebra.
In fact, we have simply that
\begin{equation}
\widehat{\BB}_\xi = \End_W\left(\bigoplus_{\lambda \vDash t}
\overline{P}_\xi(\lambda)\right)^{\op}.
\end{equation}
Moreover, finite-dimensional $\BB_\xi$-modules are the same as finite-dimensional modules over the completion.
This all depends on the fact that $\BB_\xi$
is {\em bounded} in the sense that all of the ideals
$1_\kappa \BB_\xi$ and $\BB_\xi 1_\lambda$
are finite-dimensional.
Finally, we may identify
\begin{equation}\label{gradedcenter}
\CC_\xi = Z(\widehat{\BB}_\xi).
\end{equation}
The grading on $\BB_\xi$ induces a positive grading $\CC_\xi =
\bigoplus_{d \geq 0} (\CC_\xi)_d$.

\begin{Lemma}\label{jedi}
The top graded component of $\CC_\xi$
is $\bigoplus_{\lambda \vDash t} (1_\lambda \CC_\xi
1_\lambda)_d$
where $d := m^2 + n^2 - \sum_i \gamma_i^2$,
with each summand
$(1_\lambda \CC_\xi
1_\lambda)_d$ being one-dimensional.
Also, the Jacobson radical of $\CC_\xi$ is nilpotent of codimension 1.
\end{Lemma}

\begin{proof}
The first assertion follows from Corollary~\ref{zcd}.
 The second assertion then follows because $\O_\xi$, hence,
$\overline{\O}_\xi$,
is indecomposable.
\end{proof}

For the following conjecture, we observe that $B_\xi \cap C_\xi$ is an ideal of $C_\xi$.

\begin{Conjecture}
The image of the canonical map $Z(W) \rightarrow C_\xi$ is isomorphic to $C_\xi/B_\xi \cap C_\xi$.
\end{Conjecture}

\subsection{Morita and derived equivalences between blocks}\label{SMorita}
In the final subsection, we make some remarks about the problem of classifying blocks of
$\O$ up to Morita and/or derived equivalence.
Actually, we just look at the integral blocks, which is
justified thanks to
\cite[Theorem 3.10]{CMW}.
Recall the definitions of $\lambda^+$ and $\lambda^T$ from
 the beginning of \S\ref{morec}.

We begin by discussing derived equivalences.
Following \cite[Definition 4.2]{CM},
we say that $\O_\xi$ and $\O_{\xi'}$ are {\em gradable derived
  equivalent}
if there is a $\C$-linear equivalence $F: D^b(\O_\xi) \rightarrow D^b(\O_{\xi'})$ of
triangulated categories with inverse $G$, such that both $F$ and
$G$ admit graded lifts.

The following theorem gives many examples of gradable derived
equivalences. It comes for free from the theory of braid group actions from
\cite{CR}.

\begin{Theorem}
\label{freebie}
Suppose that $\xi = (\mu,\nu;t) \in \Xi(m|n)$ and $i \in \Z$.
Let $s_i(\xi) := (s_i(\mu), s_i(\nu);t)$, where $s_i(\mu)$ (resp.
$s_i(\nu)$) is obtained by interchanging the $i$th and $(i+1)$th parts
of $\mu$ (resp. $\nu$).
Then, there is a gradable derived
equivalence
$$
\Theta_i:D^b(\O_\xi) \rightarrow D^b(\O_{s_i(\xi)})
$$
inducing a map of the form $[M(\tabA)] \mapsto \pm[M(s_i(\tabA))]$
at the level of Grothendieck groups, where $s_i(\tabA)$ is obtained by
replacing all entries $i$ of $\tabA\in\xi$
by $(i+1)$ and vice versa.
If $t=\mu_i \mu_{i+1} = \nu_{i}\nu_{i+1} = 0$, this functor is induced
by a $\C$-linear equivalence $\Theta_i:\O_\xi \rightarrow
\O_{s_i(\xi)}$ such that $\Theta_i M(\tabA) \cong M(s_i(\tabA))$ for all $\tabA
\in \xi$.
\end{Theorem}

\begin{proof}
Since we have a categorical action of $\mathfrak{sl}_\infty$ on
$\O_\Z$
as described in \S\ref{Catact},
and this categorical action admits a graded lift by
\cite[Theorem 5.26]{BLW},
the existence of $\Theta_i$ follows from
\cite[Theorem 6.4]{CR}.
The functor $\Theta_i$ is defined there
by tensoring with the ``Rickard
complex'', which
admits a graded lift by \cite[\S5.3.2]{Rou}.
When $t=\mu_i\mu_{i+1}=\nu_i\nu_{i+1}=0$, the Rickard complex
collapses to a single term,
hence, it is
a ``Scopes equivalence''.
\end{proof}

Theorem~\ref{freebie} motivates the following conjecture.

\begin{Conjecture}\label{firstc}
Take blocks $\xi = (\mu,\nu;t)\in \Xi(m|n)$ for $0 \leq m \leq n$
and $\xi' = (\mu,\nu;t') \in \Xi(m'|n')$ for $0 \leq m' \leq n'$,
such that
$\O_\xi$ and $\O_{\xi'}$ are non-trivial, i.e.\ they have more than one isomorphism class of
irreducible object.
Then $\O_\xi$ and $\O_{\xi'}$ are gradably derived equivalent if and
only if
$t = t'$, $m=m'$, $n=n'$
and $(\mu+\nu)^T = (\mu'+\nu')^T$.
\end{Conjecture}

Part of the ``if'' implication of this conjecture is implied by Theorem~\ref{freebie}:
$\O_\xi$ and $\O_{\xi'}$ are gradably derived equivalent if $t = t'$, $m=m'$, $n=n'$,
$\mu^T = \mu'^T$ and $\nu^T = \nu'^T$.  Our hope is that there should be some additional gradable derived equivalences
allowing these existing ones to be upgraded to include the case that $(\mu+\nu)^T =
(\mu'+\nu')^T$.

The graded algebra
$\CC_\xi$ from (\ref{gradedcenter}) is an invariant of gradable
derived equivalence thanks to \cite[Lemma 4.6]{CM}.
So, to prove  the ``only if'' direction of Conjecture~\ref{firstc},
one should look for more information about the structure of
$\CC_\xi$ along the lines of Lemma~\ref{jedi}.
At present, we do not even know how to show
that gradably derived equivalent blocks
have the
same atypicality.
We expect that the atypicality of a block should
be related to the dimension of its derived category in the sense of
\cite{Roucrazyone}.

\vspace{2mm}

The remainder of the subsection is concerned with Morita equivalences
between blocks.
We first point out some obvious ones which arise by twisting with automorphisms
of
$U(\g)$.
For $\xi = (\mu,\nu;t)$ and $\xi' = (\mu',\nu';t)$,
the blocks
$\O_\xi$ and $\O_{\xi'}$
are
equivalent as $\C$-linear categories if any of the following hold:
\begin{itemize}
\item
(``Translation'') There exists $s \in \Z$ such that $\mu_i = \mu_{i+s}'$ and
$\nu_i = \nu_{i+s}'$ for all $i$; use the automorphism
$e_{i,j} \mapsto e_{i,j} + (-1)^{|i|} s \delta_{i,j}$.
\item
(``Duality'') We have that
$\mu_i = \mu_{-i}'$ and $\nu_i = \nu_{-i}'$
for all $i$; use the automorphism
$e_{i,j} \mapsto -(-1)^{|i||j|} e_{w_0(j), w_0(i)}$
where $w_0$ is the longest element of $S_m \times S_n$.
\item
We have that $m=n$, $\mu = \nu'$ and $\nu = \mu'$;
use the automorphism
that switches the top left and bottom right blocks and the top right and
bottom left blocks in the standard matrix realization of $\g$.
\end{itemize}
We also get some
more interesting Morita equivalences between
typical blocks from the last part of Theorem~\ref{freebie}:
if $t=t'=0$, $\mu^+ = (\mu')^+$ and $\nu^+ = (\nu')^+$
then $\O_\xi$ and $\O_{\xi'}$ are equivalent.

The following theorem shows that there are very few equivalences
between atypical blocks.
This
was pointed out already in low rank by
Coulembier and Serganova in \cite[\S6.3]{CS}; see also
\cite[Remark 6.6]{CS} which predicts the importance of the
invariants used in the proof of Theorem~\ref{moritatheorem}.

\begin{Theorem}\label{moritatheorem}
Let $\xi = (\mu,\nu;t)\in \Xi(m|n)$ for $0 \leq m \leq n$,
and $\xi' = (\mu,\nu;t') \in \Xi(m'|n')$ for $0 \leq
m' \leq n'$.
Suppose that  $\O_\xi$ and $\O_{\xi'}$ are equivalent as
$\C$-linear categories. Then:
\begin{enumerate}
\item $t=t'$.
\end{enumerate}
Suppose in addition that there is more than one isomorphism class of
irreducible object in
the blocks $\O_\xi$ and $\O_{\xi'}$. Then:
\begin{enumerate}
\item[(2)] $m=m'$ and $n=n'$.
\end{enumerate}
Finally,
assume that there are
infinitely many isomorphism classes of irreducible objects,
so that $t,t' > 0$.
Then:
\begin{enumerate}
\item[(3)] $\mu+\nu$ and $\mu'+\nu'$ are equal up to translation
and duality.
\end{enumerate}
\end{Theorem}

\begin{proof}
In view of Theorem~\ref{qmain},
the assumption that $\O_\xi$ is equivalent to $\O_{\xi'}$
implies that
$\overline{\O}_\xi$ is equivalent to
$\overline{\O}_{\xi'}$.
If either of the blocks $\overline{\O}_\xi$ or $\overline{\O}_{\xi'}$
has a unique irreducible object up to isomorphism,
then so does the other, and we must have that
$t=t'=0$.
Otherwise, we have that $t,t' > 0$ and these blocks have infinitely
many classes of irreducible objects.
To prove (1) and (3) in these cases, we will show that $t$
and $\gamma := \mu+\nu$
can be recovered
uniquely (up to translation and duality)
from the category $\overline{\O}_\xi$
without using any
information about its structure that is external to the abstract
$\C$-linear category.

Starting from $\overline{\O}_\xi$, we can choose a complete set of pairwise
inequivalent irreducible objects $\{L(x)\:|\:x \in X\}$ indexed by
some set $X$.
Since we assumed $t > 0$, the set $X$ is infinite.
Let $P(x)$ be a projective cover of $L(x)$
and $h(x) := \#\left\{y \in X\:|\:[P(x):L(y)] \neq 0\right\}$.
By Lemma~\ref{step1}, we know that the minimal possible value for
$h(x)$ as $x$ ranges over all of $X$ is equal to $\binom{t+2}{2}$
for some $t \geq 1$.
Thus, we have recovered the atypicality $t$ of the block $\overline\O_\xi$ from
the underlying abstract category.

Now that we know $t$, we can define
$X_{\min} := \left\{x \in X\:\big|\:h(x) = \binom{t+2}{2}\right\}$.
By Lemma~\ref{step1} again, we know that $X_{\min}$ is in bijection with
$\Z$. To fix a choice of such a bijection, we arbitrarily pick some
$x_0 \in X$. Now we appeal to Lemma~\ref{step2}. It tells us that
the set $\left\{x \in X_0 \setminus\{x_0\}\:\big|\:[P(x_0):L(x)] \neq
0\right\}$ contains exactly two elements. We arbitrarily call one of these
elements $x_1$ and the other $x_{-1}$.
Then the set $\left\{x \in X_0\setminus\{x_0,x_1\}\:\big|\:[P(x_1): L(x)] \neq 0\right\}$
is a singleton $\{x_2\}$, the set
$\left\{x \in X_0\setminus\{x_1,x_2\}\:\big|\:[P(x_2): L(x)] \neq 0\right\}$
is a singleton $\{x_3\}$, and so on.
Similarly, $\left\{x \in X_0\setminus\{x_0,x_{-1}\}\:\big|\:[P(x_{-1}): L(x)] \neq 0\right\}$
is a singleton $\{x_{-2}\}$, and so on.
In this way, we have enumerated the elements of $X_{\min}$ as
$\{x_i\:|\:i \in \Z\}$.
We have done this in a way that ensures that it agrees with the
canonical labelling
$\{\overline{L}_\xi(t \eps_i)\:|\:i \in \Z\}$, at least up to some duality and
translation which we can simply ignore due to the symmetry of the
invariants that we are about to use.

Next, we explain how to recover the composition $\gamma$
from the dimensions
$\dim \End_{\O_\xi}(P(x_i))$.
By the formula in Lemma~\ref{step3}, these all take the same value
$N := \frac{m!n!}{(t!)^2 \prod_j \gamma_j!} \binom{2t}{t}$
for all but finitely
many $i \in \Z$. Thus, we have recovered the number $N$.
Rescaling, we get the numbers
\begin{equation}\label{demon}
d(i) := \binom{2t}{t}
\dim \End_{\O_\xi}(P(x_i)) \big / N
=
\sum_{r=0}^t
\binom{t}{r}
\frac{t!\gamma_i! \gamma_{i+1}!}{(\gamma_{i}+t-r)!(\gamma_{i+1}+r)!}.
\end{equation}
for each $i \in \Z$, and will explain how to recover the $\gamma_i$'s
uniquely from this sequence.
We have that
$d(i) =\binom{2t}{t}$ with equality if and only if
$\gamma_i = 0 = \gamma_{i+1}$. This already determines all but finitely many
of the $\gamma_i$'s. Observe moreover that the expression on the right
hand side of (\ref{demon}) is monotonic in
$\gamma_i$:
it
gets strictly
smaller if we make the non-negative integer $\gamma_i$ bigger.
So we can use this equation to compute each $\gamma_i$ uniquely,
assuming $\gamma_{i+1}$ has already been determined inductively
(starting from the biggest $i$ such that
$\gamma_i \neq 0$).

At this point, we have established (1) and (3). Our proof of (2)
requires considerably more force as we need to exploit the existence of
the Koszul grading on the basic algebra $\AA_\xi$ discussed in the
previous section.
If $\O_\xi$ and $\O_{\xi'}$ are Morita equivalent, then the algebras
$\AA_\xi$ and $\AA_{\xi'}$
are isomorphic as locally unital graded algebras thanks
to the unicity of Koszul gradings.
Hence, so too are the algebras $\BB_\xi$ and $\BB_{\xi'}$.
Since we already know that $t=t'$,
we can then invoke Corollary~\ref{pling}  to
deduce that
$$
q^{\binom{m}{2}+\binom{n}{2}-\sum_i \binom{\gamma_i}{2}}
[m]! [n]! \big/ \textstyle\prod_i [\gamma_i]!
=
q^{\binom{m'}{2}+\binom{n'}{2}-\sum_i \binom{\gamma'_i}{2}}
[m']! [n']! \big/ \textstyle\prod_i [\gamma'_i]!,
$$
where $\gamma' := \mu'+\nu'$ of course.
When $t,t' > 0$, we already know that $(\gamma)^+ = (\gamma')^+$, so get easily from this that
$m=m'$ and $n=n'$.
If $t = t' = 0$, we need to use also that the blocks are not trivial
(and deduce in addition that $\gamma^T = (\gamma')^T$);
one also finds this argument in the proof of \cite[Lemma 8.2]{CM}.
\end{proof}

\begin{Corollary}
For $\xi = (0,\nu;m) \in \Xi(m|n)$ and $\xi' = (0,\nu';m') \in
\Xi(m'|n')$ with $m,m' > 0$, the blocks
$\O_\xi$ and $\O_{\xi'}$ are equivalent if and only if
$m=m'$
and $\nu$ equals $\nu'$ up to translation and duality.
\end{Corollary}

Evidence for the following conjecture
comes from
Theorem~\ref{crazier}: it shows that
the Soergel algebras $\BB_\xi$ and
$\BB_{\xi'}$ in the statement
have the same graded Cartan matrices.

\begin{Conjecture}\label{crossbow}
Assume that $\xi = (\mu,\nu;t)$ and $\xi'=(\mu',\nu';t)$
for $\mu,\mu'\vDash m-t$ and $\nu,\nu' \vDash n-t$
such that $\mu+\nu$ equals $\mu'+\nu'$ up to translation and duality.
Then
$\BB_\xi \cong \BB_{\xi'}$ as locally unital graded algebras, so that the blocks
$\overline{\O}_\xi$ and $\overline{\O}_{\xi'}$ are equivalent.
\end{Conjecture}

Believing this conjecture, we also thought initially that the blocks $\O_\xi$ and
$\O_{\xi'}$ themselves
should also be equivalent (under the same hypotheses as in the conjecture).
However, this is too optimistic, due to the following
counterexample communicated to us by Coulembier: for $\g =
\gl_{3|4}(\C)$, $t=1$,
and $\xi,\xi'$ defined so that $\mu_1 = 2$, $\mu'_2 = 1$,
$\mu'_3 = 1$, $\nu_2 = 1$, $\nu_3 = 1$, $\nu_4 = 1$, $\nu'_1=2$ and
$\nu'_4=1$, the blocks $\O_\xi$ and $\O_{\xi'}$ are
not equivalent. To establish this, Coulembier shows that they have
different finitistic global dimensions by an application of \cite[Theorem~6.4]{CS}.

\appendix
\section{Proof of Theorem~\ref{T:main}}\label{appendix}

Let notation be as in the statement of the theorem.
We note that the case $m = 0$ follows from Corollary \ref{hrestc} along with
the trivial observations that $M(\tabA) = M'(\tabA)$ and $\overline K(\tabA) = \overline M(\tabA)$
when $m=0$.
So we assume henceforth that $m > 0$.
Set $M := M(\tabA)$ for short.

\begin{Lemma}\label{claim1} $H_0(M)$ is spanned by a subset $\cS$ of size $2^m$.
\end{Lemma}

\begin{proof}
We define
\begin{align*}
I^- &:=
\{(i,j) \:|\: i > j, \, i,j = 1,\dots,m+n \},\\
I^-_\ge &:= \{(i,j) \in I^- \:|\: \col(i) \le \col(j) \},\\
I^-_< &:= \{(i,j) \in I^- \:|\: \col(i) > \col(j) \}.
\end{align*}
Let $\n^-$ be the subalgebra of $\g$ of strictly lower triangular
matrices, so that $\{e_{i,j}\}_{(i,j) \in I^-}$
is a basis of $\n^-$.  Also $\{e_{i,j}\}_{(i,j) \in I^-_\ge}$
is a basis of $\n^- \cap \p$ and $\{e_{i,j}\}_{(i,j) \in I^-_<}$
is a basis of $\n^- \cap \m$.
We note that $\n^- \cap \p \sub \g_\1$, so the elements
of $\{e_{i,j}\}_{(i,j) \in I^-_\ge}$ actually all supercommute.

Fix a
total order on $I^-$ in such a way that $I^-_<$ precedes $I^-_\ge$.
The set of ordered monomials of the form $\left(\prod_{(i,j) \in I^-}
e_{i,j}^{d_{i,j}} \right)m_\tabA$, where
$d_{i,j} \in \Z$ if $\row(i) = \row(j)$ and $d_{i,j} \in \{0,1\}$ if $\row(i) > \row(j)$,
forms a basis of $M$.
For $(i,j) \in I^-_<$, we have $e_{ij} \in \m$, so $e_{ij} -
\chi(e_{ij}) \in \m_\chi$,
and $\chi(e_{ij}) \in \{0,\pm 1\}$.
Hence, the following ordered monomials span
$M/\m_\chi M$:
$$
\bigg\{
\Big(\prod_{(i,j) \in I^-_\ge} e_{i,j}^{d_{i,j}}\Big) m_\tabA + \m_\chi
M
\:\bigg|\:
\begin{array}{ll}
d_{i,j} \in \{0,1\}
\end{array}
\bigg\}.
$$
We are going to cut this
spanning set down
to one of the required size $2^m$.

For $K \sub I^-_\ge$, we use the notation
\begin{equation}
u(K) := \prod_{(i,j) \in K} e_{i,j} \in U(\n^- \cap \p),
\end{equation}
and define $\wt(K) := \sum_{(i,j) \in K} (\delta_i - \delta_j) \in \t^*_\Z$ to be the $\t$-weight of $u(K)$.
In this paragraph, we are going to focus on the monomials
$$
u(K)e_{k,l}u(L) m_\tabA + \m_\chi M \in M/\m_\chi M,
$$
where $K, L \sub I^-_\ge$ and $e_{k,l} \in \m$ or $e_{k,l} \in \b$.
The goal is to describe a {\em straightening process} in order to prove that this monomial can
be expressed as a linear combination of monomials of the form
\begin{equation} \label{e:smon}
u(J) m_\tabA + \m_\chi M,
\end{equation}
for $J \sub I^-_\ge$ such that one of the conditions (S1)--(S3) holds:
\begin{itemize}
\item[(S1)] $|J| = |K|+|L|$ and $\wt(J) = \wt(K)+\wt(L)+\delta_k-\delta_l$;
\item[(S2)] $|J| = |K|+|L|$ and $\wt(J) = \wt(K)+\wt(L)$, this is only possible if $k \ne m+1$ and $l = k-1$, or if $l = k$; or
\item[(S3)] $|J| < |K|+|L|$.
\end{itemize}
We proceed by induction on
$|K|+|L|$.
Our objective is clear when $|K|+|L| = 0$, as
$e_{k,l} m_\tabA + \m_\chi M = \chi(e_{k,l}) + \m_\chi M$ if $e_{k,l} \in \m$,
and $e_{k,l} m_\tabA + \m_\chi M = \lambda_\tabA(e_{k,l}) + \m_\chi M$ if $e_{k,l} \in \b$,
where $\lambda_\tabA$ is viewed as an element of $\b^*$.
Now assume that $|K|+|L| > 0$.

For the case $e_{k,l} \in \m$, we have
\begin{equation} \label{e:cleft}
u(K)e_{k,l}u(L) m_\tabA +\m_\chi M = \pm e_{k,l}u(K)u(L) m_\tabA  +
[u(K),e_{k,l}]u(L) m_\tabA +\m_\chi M.
\end{equation}
We do not need to know the sign in the equation above,
so we won't specify these explicitly; this is also
the case in some other equations below.
The first term  in \eqref{e:cleft} is (up to sign)
$$
e_{k,l}u(K)u(L) m_\tabA + \m_\chi M = \chi(e_{k,l}) u(K)u(L) + \m_\chi M,
$$
which is zero or (up to sign) a monomial as in \eqref{e:smon} satisfying (S2).

For the case $e_{k,l} \in \b$, we have
\begin{equation} \label{e:cright}
u(K)e_{k,l}u(L) m_\tabA +\m_\chi M = \pm u(K)u(L)e_{k,l} m_\tabA  +
u(K)[e_{k,l},u(L)] m_\tabA +\m_\chi M.
\end{equation}
The first term in \eqref{e:cright} is (up to sign)
$$
u(K)u(L)e_{k,l} m_\tabA + \m_\chi M = \lambda_\tabA(e_{k,l}) u(K)u(L) + \m_\chi M,
$$
where $\lambda_\tabA$ is viewed as an element of $\b^*$.
This is zero or (up to sign) a monomial as in \eqref{e:smon} satisfying (S2).

Now we consider the case $e_{k,l} \in \m$ further. Observe that the term $[u(K),e_{k,l}]$ occurring in \eqref{e:cleft} is a sum of terms of the form
$\pm u(K_{i,j})[e_{i,j},e_{k,l}] u(K^{i,j})$,
summed over $(i,j) \in K$ where $K_{i,j}$ is the set of elements of
$K$ before $(i,j)$ in our fixed order of $I^-$ and $K^{i,j}$ is the
set of those after $(i,j)$.
Either $[e_{i,j},e_{k,l}]$ is zero, or an element of one of
$\n^- \cap \p$, $\m$ or $\b$; note that $\m \cap \b \ne \{0\}$ in general,
so $[e_{i,j},e_{k,l}]$ can be an element of both $\m$ and $\b$
but this does not matter.
If $[e_{i,j},e_{k,l}] \in \n^- \cap \p$, then
$u(K_{i,j})[e_{i,j},e_{k,l}] u(K^{i,j})u(L) m_\tabA +\m_\chi M$ is (up to sign) a monomial of the form \eqref{e:smon}
for which (S1) holds.  Whereas if $[e_{i,j},e_{k,l}] \in \m$ or $[e_{i,j},e_{k,l}] \in \b$, then
we can apply induction to deduce that $u(K_{i,j})[e_{i,j},e_{k,l}]
u(K^{i,j})u(L) m_\tabA +\m_\chi M$ is
a sum of monomials as in \eqref{e:smon} that satisfy (S3).

For the case $e_{k,l} \in \b$ we can argue entirely similarly, but working with the term $u(K)[e_{k,l},u(L)]$,
which occurs in \eqref{e:cright}.

We have now established that our straightening process works.

\vspace{1mm}
\smallskip

Now we let
$H := \{(i,j) \in I^-_\ge \:|\: i > m+1\}$, then define
\begin{align}\label{barty}
\cH
&:= \{L \sub I^-_\ge \:|\: H \sub L \},
&\cS &:= \{u(L) + \m_\chi M \:|\: L \in \cH\}.
\end{align}
Let $<_{\lex}$ be the order on $\t_\Z^*$ defined by
$\sum_{i=1}^{m+n} r_i \delta_i <_{\lex} \sum_{i=1}^{m+n} s_i \delta_i$ if
$r_j < s_j$ where $j$ is maximal such that $r_j \ne s_j$.
Take $K \sub I^-_\ge$.
Proceeding by induction on $|K|$ and reverse induction on $\wt(K)$
with respect to $<_{\lex}$,
we are going to prove that $u(K)$ is in the space spanned by $\cS$.
Before proceeding, we note that the case $m =1$ and $\pi$ is left justified
is trivial, because $H = \varnothing$. So we assume that this is not the case.

For the base step we consider the case $|K| = 0$, so that $u(K) = m_\tabA$.
In this case we let $k = 2m + s_-$. Our assumption above implies that $k$ is maximal such that
the column of $k$ in the pyramid $\pi$ has two boxes, and that this is not the leftmost column in $\pi$.
In particular, $k-1$ also lies in the second row of $\pi$.

Now consider $e_{m,k-1}e_{k,m} m_\tabA + \m_\chi M$.
We have that $e_{m,k-1} \in \m$ and that $\chi(e_{m,k-1}) = 0$, so that
$e_{m,k-1}e_{k,m} m_\tabA + \m_\chi M = 0 + \m_\chi M$.
Moreover,
\begin{align*}
e_{m,k-1}e_{k,m} m_\tabA + \m_\chi M &= -e_{k,m}e_{m,k-1} m_\tabA + e_{k,k-1} m_\tabA + \m_\chi M
= - m_\tabA + \m_\chi M,
\end{align*}
where we use that $e_{m,k-1} m_\tabA = 0$, because $e_{m,k-1} \in \n$, and that $e_{k,k-1} m_\tabA + \m_\chi M =
- m_\tabA + \m_\chi M$ because $e_{k,k-1}+1 \in \m_\chi$.
It follows that $m_\tabA \in \m_\chi M$, so $m_\tabA + \m_\chi M$ is certainly
in the span of $\cS$.

Next let $K \sub I^-_\ge$ and assume inductively that $u(L)m_\tabA + \m_\chi M$ is in the
span of $\cS$ whenever $|L| < |K|$ or when $|L| = |K|$ and $\wt(L) >_{\lex} \wt(K)$.
If $H \sub K$, then $u(K) \in \cS$, so we may assume this is not the case.
We choose $(k,l) \in H \, \setminus K$ such that $k$ is maximal and $l$ minimal given $k$.
Consider $e_{l,k-1}e_{k,l}u(K)m_\tabA$.  Since $e_{l,k-1} \in \m$ and $\chi(e_{l,k-1}) = 0$,  we have
$e_{l,k-1}e_{k,l} u(K) + \m_\chi M = 0 + \m_\chi M$.
Moreover,
\begin{align*}
e_{l,k-1}e_{k,l} u(K) m_\tabA + \m_\chi M &= -e_{k,l}e_{l,k-1} u(K) m_\tabA + e_{k,k-1} u(K) m_\tabA + \m_\chi M \\
&= -e_{k,l}e_{l,k-1} u(K) m_\tabA - u(K) + \m_\chi M,
\end{align*}
where we use that $e_{k,k-1} u(K) m_\tabA + \m_\chi M =
- u(K) m_\tabA + \m_\chi M$, because $e_{k,k-1}+1 \in \m_\chi$.
Thus we see that it suffices to show that $e_{k,l}e_{l,k-1} u(K) m_\tabA +\m_\chi M$
is in the span of $\cS$.

To see this, we note that the $\t$-weight of $e_{k,l}e_{l,k-1} u(K)$ is
$\wt(K) + \delta_k - \delta_{k-1} >_{\lex} \wt(K)$.
Next we calculate
\begin{align*}
e_{k,l} e_{l,k-1} u(K) m_\tabA &= (-1)^{|K|} e_{k,l} u(K) e_{l,k-1} m_\tabA + e_{k,l}[e_{l,k-1},u(K)] m_\tabA \\
&= e_{k,l}[e_{l,k-1},u(K)] m_\tabA,
\end{align*}
because $e_{l,k-1} \in \n$ so that $e_{l,k-1} m_\tabA = 0$.
Then we see that $[e_{l,k-1},u(K)]$ is a sum of terms of the form
$(-1)^{|K_{i,j}|} u(K_{i,j})[e_{l,k-1},e_{i,j}] u(K^{i,j})$
over $(i,j) \in K$.   The nonzero possibilities for $[e_{l,k-1},e_{i,j}]$ are
$e_{l,l}+e_{k-1,k-1}$,  $e_{l,j}$ or $e_{i,k-1}$ all of which lie in either $\m$ or $\b$.
Moreover, we note that $e_{k,k-1}$ is not possible, because $e_{k,l} \notin K$, though $e_{l,l-1}$ can occur.
Therefore, using the straightening process, we obtain that
$e_{k,l}u(K_{i,j})[e_{l,k-1},e_{i,j}] u(K^{i,j})+\m_\chi M$
is a sum of monomials of the form
$u(J)m_\tabA + \m_\chi M$
for $J \sub I^-_\ge$.
Moreover, in the present situation the conditions (S1)--(S3) translate to
saying that
\begin{itemize}
\item[(S1$'$)] $|J| = |K|$ and $\wt(J)$ is either $\wt(K) + \delta_k-\delta_{k-1}$ or
$\wt(K) + \delta_k-\delta_{k-1} + \delta_l-\delta_{l-1}$; or
\item[(S2$'$)] $|J| < |K|$.
\end{itemize}
In the first case the possibilities for $\wt(M)$ satisfy $\wt(K) <_{\lex} \wt(M)$.  Hence,
we can apply induction to deduce that $u(K)m_\tabA + \m_\chi M$ is in the
span of $\cS$.
Since $|\cS| = 2^m$,
this completes the proof of the lemma.
\end{proof}

\begin{Lemma}\label{claim2}
 $H_0(M)$ is a highest weight supermodule for $W$.
\end{Lemma}

\begin{proof}
We recall that the grading on $\g$ determined by the pyramid $\pi$ given
by $\deg(e_{i,j}) = \col(j) - \col(i)$ induces a grading
on $U(\p)$.  We refer to this is grading as the {\em Lie grading} as in \cite{BBG}.
Also we require the explicit
formulas for the elements $f^{(r)} \in W \sub U(\p)$ for $r = s_-+1,\dots,s_-+m$ given in \cite[\S4]{BBG}.
We let $\bar e_{i,j} := (-1)^{|\row(i)|} e_{i,j}$ and recall that
$$
f^{(r)} =
S_{\rho'} \bigg(\sum_{s=1}^r (-1)^{r-s}
  \sum_{\substack{i_1,\dots,i_s \\ j_1,\dots,j_s}}
(-1)^{\#\left\{a=1,\dots,s-1\:|\:
\row(j_a) =1
\right\}}
\bar e_{i_1,j_1} \cdots \bar e_{i_s,j_s}\bigg),
$$
where the sum is over all $1 \le i_1, \dots , i_s, j_1, \dots , j_s \le m+n$ such that
\begin{itemize}
\item[(R1)] $\row(i_1) = 2$ and $\row(j_s) = 1$;
\item[(R2)] $\col(i_a) \le \col(j_a)$ ($a=1,\dots,s$);
\item[(R3)] $\row(i_{a+1}) = \row(j_{a})$ ($a=1,\dots,s-1$);
\item[(R4)] if $\row(j_a) = 2$, then $\col(i_{a+1}) > \col(j_a)$
  ($a=1,\dots,s-1$);
\item[(R5)] if $\row(j_a) = 1$, then $\col(i_{a+1}) \leq\col(j_a)$
  ($a=1,\dots,s-1$);
\item[(R6)] $\deg(e_{i_1,j_1}) + \dots + \deg(e_{i_s,j_s}) = r-s$.
\end{itemize}
To spell out the key properties we need from this formula, write
$$
f^{(r)} = \pm \sum_{i=1}^{m-r+1} e_{m+i,i+r-(s_-+1)} + g^{(r)}
$$
where $g^{(r)} \in U(\p)$ is a linear combination of terms of the form
$e_{i_1,j_1} \cdots e_{i_s,j_s} \in U(\p)$
which satisfy the following conditions.
\begin{itemize}
\item[(F1)] The degree of $e_{i_1,j_1} \cdots e_{i_s,j_s}$ in the Lie grading is strictly less than $r-1$.
\item[(F2)] There exists $s'$ such that $\row(i_t),\row(j_t) = 1$ for all $t > s'$ and
$\row(i_{s'}) = 2$, $\row(j_{s'}) = 1$; moreover, if $i_{s'} = m+1$, then
$s' = 1$.
\end{itemize}

Now let $\cH$ and $\cS$ be as in (\ref{barty}).
Let $K \in \cH$, so that $u(K)m_\tabA + \m_\chi M \in \cS$, and let $r \in \{s_-+1,\dots,s_-+m\}$.
We will prove that
$g^{(r)} u(K)m_\tabA + \m_\chi M$ is a linear combination of terms of the form $u(L) m_\tabA + \m_\chi M$,
where $L \in \cH$ such that $u(L)$ has Lie degree strictly less than $u(K) + r-1$.

Let $i,j \in \{1,\dots,m+n\}$ such that
$\row(i) = \row(j) = 1$, $\col(i) \le \col(j)$ and let $L \in \cH$.
Consider $e_{i,j} u(L) m_\tabA + \m_\chi M$.
We have that
\begin{align*}
e_{i,j} u(L) m_\tabA + \m_\chi M &= [e_{i,j},u(L)] m_\tabA + u(L)\lambda_\tabA(e_{i,j})m_\tabA + \m_\chi M.
\end{align*}
Now $[e_{i,j},u(L)]$ is a sum of terms of the form
$\pm u(L_{k,l})[e_{i,j},e_{k,l}] u(L^{k,l})$
over $(k,l) \in L$, where $L_{k,l}$ is the set of elements of
$K$ before $(k,l)$ in our fixed order of $I^-$ and $L^{k,l}$ is the
set of those after $(k,l)$.  We have that $[e_{i,j},e_{k,l}]$ is either zero or equal to $-e_{k,j}$ if $i = l$, and in
this case we have $(k,j) \in I^-_{\ge 0}$.
It follows that $u(L_{i,j})[e_{i,j},e_{k,l}] u(L^{i,j})$ is either zero or equal to
$\pm u(L')$ for some $L' \in \cH$ with the same Lie degree as $e_{i,j} u(L)$.

Next let $i,j \in \{1,\dots,m+n\}$ such that
$\row(i) = 2$, $\row(j) = 1$, $\col(i) \le \col(j)$ and let $L \in \cH$.
We observe that $e_{i,j} u(L) =0$ if $\col(i) > 1$.  Also if $\col(i) = 1$ (so
$i = m+1$), then we have that $e_{i,j} u(L)$ is either zero or equal to
$\pm u(L')$ for some $L' \in \cH$ with the same Lie degree as $e_{i,j} u(L)$.

Combining the discussion in the previous two paragraphs with the
fact that $g^{(r)} u(K)m_\tabA + \m_\chi M$ is a sum of terms of the form
$e_{i_1,j_1} \cdots e_{i_s,j_s} u(K)m_\tabA + \m_\chi M$
subject to conditions (F1) and (F2) we deduce that
$g^{(r)} u(K)m_\tabA + \m_\chi M$ is a linear combination of terms of the form $u(L) m_\tabA + \m_\chi M$,
where $L \in \cH$ such that $u(L)$ has Lie degree strictly less than $u(K) + r-1$.

Now suppose
that $(m+1,r-s_-) \not\in K$.  Then we have that
$e_{m+1,r-s_-}u(K) = \pm u(K \cup \{(m+1,r-s_-)\})$, and that $e_{m+i,i+r-(s_-+1)} u(K) = 0$
for all $i > 1$.  Therefore,
$$
f^{(r)}(u(K)m_\tabA + \m_\chi M) = u(K \cup \{(m+1,r-s_-\}) + g^{(r)} u(K) + \m_\chi M.
$$
We deduce that
$$
\left\{\prod_{r=s_-+1}^{s_-+m} (f^{(r)})^{a_r} u(H)m_\tabA + \m_\chi M \:\: \vline \:\: a_r \in \{0,1\}\right\}
$$
spans $H_0(M)$, because
$(f^{(r)})^{a_r} u(H)m_\tabA + \m_\chi M$ is equal to
$$
u(H \cup \{(m+1,r-s_-) \:|\: a_r = 1\})m_\tabA  + (\text{terms of lower Lie degree)} + \m_\chi M.
$$
Hence, $H_0(M)$ is a highest weight supermodule
as required.
\end{proof}

\proof[Proof of Theorem~\ref{T:main}]
By Lemma~\ref{claim2} and the universal property of Verma supermodules,
there is a surjective homomorphism $\theta:\Pi^p
\overline{M}(\tabB) \twoheadrightarrow H_0(M)$
for some $\tabB \in \Tab$ and some parity $p \in \Z/2$.
By Lemmas \ref{vermasequal}
and \ref{exactness},  we have that
$[H_0(M)] = [H_0(M'(\tabA))]$
in the Grothendieck group $K_0(W\lsmof)$.
By Corollary \ref{hrestc}, $[H_0(M'(\tabA))] = [\overline{K}(\tabA)]$.
These facts imply that $\dim H_0(M) = \dim \overline{K}(\tabA) = 2^m$, so
that the spanning
set from Lemma~\ref{claim1} is actually a basis.
Since $\dim \overline{M}(\tabB) = 2^m$ too, this shows that
$\theta$ is in fact an isomorphism,
and moreover we have established that $[\Pi^p \overline{M}(\tabB)] = [\overline{K}(\tabA)]$.

It remains to show that $p=\0$ and $\tabB=\tabA$.
By their definitions (\ref{craven}) and (\ref{e:WVerma}), the
$W$-supermodules $\overline{K}(\tabA)$ and $\overline{M}(\tabB)$ are both
diagonalizable with respect to $d_2^{(1)}$,
the vectors $\overline k_\tabA$ and $\overline m_\tabB$ are eigenvectors
of $d_2^{(1)}$-eigenvalues $b(\tabA)$ and $b(\tabB)$, and
have parities $\parity(b(\tabA))$ and $\parity(b(\tabB))$, respectively.
Moreover, all other $d_2^{(1)}$-eigenspaces in these supermodules correspond to
strictly smaller eigenvalues.
As $e^{(r)}$ raises $d_2^{(1)}$-eigenvalues by one, $\overline k_\tabA$
must be a highest weight
vector. Hence, there is a non-zero (but not necessarily surjective) homomorphism
$\overline{M}(\tabA) \rightarrow \overline{K}(\tabA), \overline m_\tabA
\mapsto \overline k_\tabA$. This discussion implies that
\begin{align*}
[\overline K(\tabA)] &= [\overline L(\tabA)] + \left(\text{$[\overline
  L(\tabC)]$'s and $[\Pi\overline L(\tabC)]$'s
with $b(\tabC) < b(\tabA)$}\right),\\\
[\Pi^p\overline M(\tabB)] &= [\Pi^p \overline L(\tabB)] + \left(\text{$[\overline
  L(\tabC)]$'s and $[\Pi \overline L(\tabC)]$'s with $b(\tabC) < b(\tabA)$}\right).
\end{align*}
In the previous paragraph,
we
established already that
$[\Pi^p \overline{M}(\tabB)] = [\overline{K}(\tabA)]$.
So we must have that $\tabA=\tabB$ and
$p=0$, and the proof is complete.
\endproof

\section{Proof of Theorem~\ref{crazier}}\label{appendixb}

Fix $N \geq 2$ and let $U_q \mathfrak{sl}_N$ be the usual
quantized enveloping algebra
over the field $\Q(q)$ ($q$ an
indeterminate) that is
associated to the simple Lie algebra $\mathfrak{sl}_N(\C)$.
We denote its standard generators by $\left\{F_i,
  E_i, K_i^{\pm}\right\}_{1 \leq i <
  N}$.
Let $P := \bigoplus_{i=1}^N \Z \eps_i$ be its weight lattice, with
simple roots $\{\alpha_i := \eps_i-\eps_{i+1}\}_{1 \leq i < N}$
and symmetric form $(-,-)$ defined from $(\eps_i,\eps_j) :=
\delta_{i,j}$.
We have the natural $U_q \mathfrak{sl}_N$-module $V^+$ on basis
$\{v_i^+\}_{1 \leq i \leq N}$ and the dual natural module $
V^-$ on basis $\{v_i^-\}_{1 \leq i\leq  N}$. The actions of
the generators on these bases are given by the following formulae:
\begin{align*}
F_i v^{+}_j &= \delta_{i,j} v^{+}_{i+1},
&E_i v^{+}_j &= \delta_{i+1,j} v^{+}_i,
&
K_i  v^+_j &= q^{(\alpha_i,\eps_j)}  v^+_j,
\\
F_i v^{-}_j &= \delta_{i+1,j} v^{-}_i,
&
E_i v^{-}_j &= \delta_{i,j} v^{-}_{i+1},
&K_i  v^-_j &= q^{(\alpha_i,-\eps_j)} v^-_j.
\end{align*}
We'll work with the comultiplication
$\Delta:U_q \mathfrak{sl}_N \rightarrow U_q \mathfrak{sl}_N \otimes U_q \mathfrak{sl}_N$ defined from
$$
\Delta(F_i) =
1 \otimes F_i  + F_i \otimes K_i,
\quad
\Delta(E_i) =
K_i^{-1} \otimes E_i  + E_i \otimes 1,
\quad
\Delta(K_i) = K_i \otimes K_i.
$$
This is not quite the same as the comultiplication in Lusztig's book
\cite[\S3.1.3]{Lubook}: our
$q$ and $K_i$ are Lusztig's $v^{-1}$ and $K_i^{-1}$.
All definitions from \cite{Lubook} cited below should be
modified accordingly.

Let $\Theta$ be the quasi-$R$-matrix from \cite[\S4.1.1]{Lubook} (with $v$ replaced
by $q^{-1}$), and $R=R_{V,W}:V \otimes W \stackrel{\sim}{\rightarrow} W \otimes V$
be the $R$-matrix from \cite[\S32.1.4]{Lubook} for any integrable modules $V, W$. Thus, for
vectors $v \in V, w \in W$ of
weights
$\lambda,\mu \in P$, we have that $R(v \otimes w) =
q^{(\lambda,\mu)} \Theta(w \otimes v)$.
The following explicit formulae
for the action of the inverse of the $R$-matrix
on $V^\pm$ are
derived in \cite[$\S$5]{BSW}:
\begin{align}
R^{-1}(v^+_i \otimes v^+_j) &=
\left\{
\begin{array}{ll}
 v^+_j \otimes v^+_i&\text{if $i > j$},\\
q^{-1}v^+_j \otimes v^+_i&\text{if $i =j$},\\
 v^+_j \otimes v^+_i- (q-q^{-1}) v^+_i \otimes v^+_j\hspace{22.6mm}&\text{if $i < j$};
\end{array}
\right.\notag\\
R^{-1}(v^-_i \otimes v^-_j) &=
\left\{
\begin{array}{ll}
 v^-_j \otimes v^-_i&\text{if $i < j$},\\
 q^{-1}v^-_j \otimes v^-_i&\text{if $i =j$},\\
v^-_j \otimes v^-_i- (q-q^{-1}) v^-_i \otimes v^-_j\hspace{22.8mm}&\text{if $i > j$};
\end{array}
\right.\notag\\
R^{-1}(v^+_i \otimes v^-_j) &=
\left\{
\begin{array}{ll}
v^-_j \otimes v^+_i&\text{if $i \neq j$},\\
\displaystyle q v^-_j \otimes v^+_i + (q-q^{-1})\sum_{r=1}^{N-i} (-q)^{r} v^-_{j+r}\otimes v^+_{i+r}&\text{if $i = j$};
\end{array}
\right.\notag
\\
R^{-1}(v^-_i \otimes v^+_j) &=
\left\{
\begin{array}{ll}
v^+_j \otimes v^-_i&\text{if $i \neq j$},\\
\displaystyle q v^+_j \otimes v^-_i + (q-q^{-1})\sum_{r=1}^{i-1} (-q)^{r} v^+_{j-r}\otimes v^-_{i-r}\:&\text{if $i = j$}.
\end{array}
\right.\notag
\end{align}

For a sign sequence $\bsig = (\sig_1,\dots,\sig_k) \in \{\pm\}^k$,
we have the tensor space $V^{\otimes\bsig} :=
V^{\sig_1}\otimes\cdots\otimes V^{\sig_k}$, with basis
$\{v_{i_1}^{\sig_1}\otimes\cdots\otimes v_{i_k}^{\sig_k}\}_{1 \leq
i_1,\dots,i_k \leq N}$.
Let $(-,-)$ be the symmetric bilinear form on $V^{\otimes \bsig}$
defined by declaring that this basis is orthonormal.
There is an anti-linear
algebra automorphism
$\psi:U_q\mathfrak{sl}_N \rightarrow U_q\mathfrak{sl}_N$ defined by
$$
\psi(F_i) := F_i,\qquad
\psi(E_i) := E_i,\qquad
\psi(K_i) := K_i^{-1}.
$$
The modules $V^{\pm}$ possess anti-linear bar-involutions $\psi$ compatible
with this in the sense that $\psi(u v) = \psi(u) \psi(v)$ for all $u
\in \dot I, v \in V^{\pm}$; these are defined simply so that
$\psi(v_i^{\pm}) = v_i^{\pm}$ for all $1 \leq i \leq N$.
Applying
Lusztig's general construction from \cite[\S27.3.1]{Lubook},
we get also a (highly non-trivial)
compatible bar involution
$\psi:V^{\otimes\bsig} \rightarrow V^{\otimes \bsig}$.
Finally, let $\psi^*:V^{\otimes\bsig}\rightarrow V^{\otimes
  \bsig}$
be the adjoint anti-linear involution to $\psi$ with respect to the
form $(-,-)$, i.e.\
$\overline{(\psi(v),w)} = (v, \psi^*(w))$ for all $v, w \in
V^{\otimes\bsig}$.

\begin{Lemma}\label{getstarted}
Let $w_0$ be the longest element of the symmetric group $S_k$, so that
$w_0(\bsig) = (\sig_k,\dots,\sig_1)$.
For $s_i = (i\:\:i\!+\!1) \in S_k$, let $R_{i}$
be the $R$-matrix $1^{\otimes (i-1)} \otimes R \otimes 1^{\otimes
  (k-i-1)}$.
Then let
$R_{w_0}:
V^{\otimes w_0(\bsig)} \stackrel{\sim}{\rightarrow} V^{\otimes \bsig}$ be
the isomorphism
$R_{i_1} \circ \cdots \circ R_{i_{k(k-1)/2}}$
obtained from any reduced expression $w_0 = s_{i_1}\cdots s_{i_{k(k-1)/2}}$.
Define $R_{w_0}^{-1}: V^{\otimes w_0(\bsig)} \stackrel{\sim}{\rightarrow} V^{\otimes \bsig}$
similarly using the inverse $R$-matrices throughout.
Then, we have that
\begin{align}
\psi(v_{i_1}^{\sig_1}\otimes\cdots\otimes v_{i_k}^{\sig_k})
&= q^{-\sum_{1 \leq r < s \leq k} (\sig_{i_r} \eps_{i_r}, \sig_{i_s} \eps_{i_s})}
R_{w_0}(v_{i_k}^{\sig_k}\otimes\cdots\otimes v_{i_1}^{\sig_1}),\\
\psi^*(v_{i_1}^{\sig_1}\otimes\cdots\otimes v_{i_k}^{\sig_k})
&= q^{\sum_{1 \leq r < s \leq k}  (\sig_{i_r} \eps_{i_r}, \sig_{i_s} \eps_{i_s})}
R^{-1}_{w_0}(v_{i_k}^{\sig_k}\otimes\cdots\otimes v_{i_1}^{\sig_1}),\label{getstartedb}
\end{align}
for any $1 \leq i_1,\dots,i_k \leq N$.
\end{Lemma}

\begin{proof}
The formula for $\psi$ follows immediately from Lusztig's construction
in \cite[\S27.3.1]{Lubook}, plus the formula expressing the
$R$-matrix
in terms of the quasi-$R$-matrix from \cite[\S32.1.4]{Lubook}.
To deduce the formula for the adjoint map $\psi^*$, one reduces to the case that
$k=2$, which may then be checked directly using the formulae for $R$
and $R^{-1}$ displayed above.
\end{proof}

Henceforth, we will be interested just in the spaces
$T^{m|n} := (V^+)^{\otimes m} \otimes (V^-)^{\otimes n}$
for $m,n \geq 0$.
Set
$$
T := \bigoplus_{m,n \geq 0} T^{m|n},
$$
with bar involutions $\psi, \psi^*:T \rightarrow T$ obtained from the
ones on each $T^{m|n}$.
Like in (\ref{monomials}),
we denote the monomial basis of $T^{m|n}$ by $\{v_\tabA\}_{\tabA \in \Tab_{m|n}}$,
where $\Tab_{m|n}$ denotes the set of all
tableaux
$\tabA
= \substack{a_1 \cdots a_{m} \\ b_1 \cdots b_{n}}$ with
entries satisfying $1 \leq a_1,\dots,a_m,b_1,\dots,b_n \leq N$.
Also let $\Tab_{m|n}^\circ$ be the set of all the anti-dominant
tableaux in $\Tab_{m|n}$, i.e.\ the tableaux
$\tabA
= \substack{a_1 \cdots a_{m} \\ b_1 \cdots b_{n}}$
satisfying $1 \leq a_1 \leq \cdots \leq a_m \leq N
\geq b_1 \geq
\cdots \geq b_n \geq 1$.

As in \cite[\S27.3.1]{Lubook},
the bar involutions $\psi$ and $\psi^*$ have the properties
\begin{align}
\psi(v_\tabA) &= v_\tabA + \text{(a $\Z[q,q^{-1}]$-linear combination
of $v_\tabB$'s for $\tabB \succ \tabA$)},\\
\psi^*(v_\tabA) &= v_\tabA + \text{(a $\Z[q,q^{-1}]$-linear combination
of $v_\tabB$'s for $\tabB \prec \tabA$)}.\label{soy}
\end{align}
So we can apply Lusztig's Lemma as in the proof of \cite[Theorem 27.3.2]{Lubook} to introduce
the {\em canonical basis}
$\{b_\tabA\}_{\tabA \in \Tab_{m|n}}$
and {\em dual canonical basis}
$\{b_\tabA^*\}_{\tabA \in \Tab_{m|n}}$ of $T^{m|n}$, which are
the unique bases determined by the following properties:
\begin{align}
\psi(b_\tabA) &= b_\tabA,
&
b_\tabA
&\in v_\tabA +\bigoplus_{\tabB \in \Tab_{m|n}} q \Z[q] v_\tabB,\\
\psi^*(b_\tabA^*) &= b^*_\tabA,&
b_\tabA^*
&\in v_\tabA +\bigoplus_{\tabB \in \Tab_{m|n}} q \Z[q] v_\tabB.
\end{align}
Since $\psi$ and $\psi^*$ are adjoint, the canonical and dual
canonical bases are dual bases with respect to the form
$(-,-)$.

Let $S$ be the $\Q(q)$-algebra defined by generators $\{x_i,y_i\}_{1
  \leq i \leq N}$ subject to the following relations:
\begin{align}\label{rel1}
x_i x_j &= q x_j x_i&\text{if }i > j,\\\label{rel2}
y_i y_j &= q y_j y_i&\text{if }i < j,\\\label{rel3}
y_i x_j &= x_j y_i&\text{if }i \neq j,\\
y_i x_i &= q x_i y_i + (q-q^{-1}) \sum_{r=1}^{i-1} (-q)^r x_{i-r}
y_{i-r}. \label{rel4}
\end{align}
The algebra $S$ admits compatible gradings
$S = \bigoplus_{\gamma \in P} S_\gamma$ and
$S= \bigoplus_{m,n \geq 0} S^{m|n}$,
the first of which is defined
by declaring that $\deg(x_i) := \eps_i$ and $\deg(y_i) := -\eps_i$ for
each $i$,
and the second by declaring that
$S^{m|n}$
is the span of the monomials
\begin{equation}
u_\tabA := x_{a_1}\cdots x_{a_m} y_{b_1}\cdots y_{b_n}
\end{equation}
for all $\tabA =
\substack{a_1 \cdots a_{m} \\ b_1 \cdots b_{n}}\in
\Tab_{m|n}$.
The following theorem shows that each $S^{m|n}$ has two distinguished bases: the
{\em monomial basis}
 $\{u_\tabA\}_{\tabA \in
    \Tab^\circ_{m|n}}$ and the {\em dual canonical basis} $\{d_\tabA\}_{\tabA \in
    \Tab^\circ_{m|n}}$, which is explicitly computed.
The proof is analogous to that of \cite[Theorem 20]{Bdual}.

\begin{Theorem}\label{incredible}
The vectors $\{u_\tabA\}_{\tabA \in \Tab^{\circ}_{m|n}}$ give a
basis for $S^{m|n}$.
Moreover:
\begin{enumerate}
\item The $\Q(q)$-linear map
$\pi: T \twoheadrightarrow S,\: v_\tabA \mapsto u_\tabA$
intertwines the dual bar involution $\psi^*$ on $T$ with the
unique anti-linear involution $\psi^*:S\rightarrow S$
such that $\psi^*(x_i) = x_i, \psi^*(y_i) = y_i$ and
\begin{equation}\label{newbarinv}
\psi^*(u u') = q^{(\gamma,\gamma')-mm'-nn'}\psi^*(u') \psi^*(u)
\end{equation}
for all
$u \in S^{m|n}\cap S_\gamma$ and $u' \in S^{m'|n'} \cap S_{\gamma'}$.
\item
For $\tabA \in \Tab_{m|n}$, we have that $\pi(b_\tabA^*) = 0$ unless $\tabA$ is
anti-dominant, in which case the vector $d_\tabA :=
\pi(b_\tabA^*)$ is characterized uniquely by the following properties:
$\psi^*(d_\tabA) = d_\tabA$,
$d_\tabA \in u_\tabA + \sum_{\tabB \in \Tab_{m|n}^\circ}
q \Z[q] u_\tabB$.
\item
Let $z_0 := 0$ and $z_i := x_i y_i - q  x_{i-1} y_{i-1} + \cdots + (-q)^{i-1} x_1
y_1$ for $1 \leq i \leq N$.
Given $\tabA \in \Tab^\circ_{m|n}$ of atypicality $t$,
choose $\substack{c_1 \cdots c_t a_1 \cdots a_{m -t} \\ c_1 \cdots c_t
 b_1 \cdots b_{n-t}  } \sim \tabA$
such that $a_1 \leq \cdots \leq a_{m-t}$ and $b_1 \geq \cdots \geq
b_{n-t}$.
Then:
$$
d_\tabA = q^{-t(t-1)/2 - \#\{(i,j)\:|\:a_i > c_j\} - \#\{(i,j)\:|\:b_i > c_j\}} x_{a_1} \cdots x_{a_{m-t}} z_{c_1} \cdots z_{c_t} y_{b_{1}}\cdots y_{b_{n-t}}.
$$
\end{enumerate}
The vectors $\{d_\tabA\}_{\tabA \in
\Tab_{m|n}^\circ}$ give another basis for $S^{m|n}$.
\end{Theorem}

\begin{proof}
In this proof, we will cite some results from \cite{Bdual}. The
conventions followed there are consistent with those of \cite{Lubook}, so that one
needs to replace $q$ by $q^{-1}$ and $K_i$ by $K_i^{-1}$ when
translating from \cite{Bdual} to the present setting.
The $R$-matrix $\mathcal R_{V,W}$ in \cite{Bdual} is the same as our inverse
$R$-matrix $R_{V,W}^{-1} := (R_{W,V})^{-1}$ (with $q$ replaced by
$q^{-1}$); hence, in view also of (\ref{getstartedb}),
the bar
involution defined in \cite[(3.2)]{Bdual} corresponds to our $\psi^*$.

We begin by recalling the standard definitions of the quantum symmetric algebras
$S(V^+) = \bigoplus_{m \geq 0} S^m(V^+)$ and $S(V^-) = \bigoplus_{n
  \geq 0} S^n(V^+)$.
As discussed in detail in \cite[\S5]{Bdual}, the former
is the quotient of the tensor algebra $T(V^+)$ by the
two-sided ideal $$
I^+ := \langle v^+_i \otimes v^+_j - q v^+_j \otimes
v^+_i\:|\:i > j\rangle.
$$
It has a basis consisting of the
images
$v_{i_1}^+ \cdots v_{i_m}^+$ of the tensors
$v_{i_1}^+ \otimes \cdots \otimes v_{i_m}^+$ for $m \geq 0$ and $i_1 \leq \cdots \leq i_m$.
Similarly, $S(V^-)$ is the quotient of $T(V^-)$ by
$$I^- :=
\langle v^-_i \otimes v_j^- - q v_j^- \otimes v_i^-\:|\:i < j
\rangle,
$$
and it has basis
$v_{j_1}^- \cdots v_{j_n}^-$ for $n \geq 0$ and $j_1 \geq \cdots \geq j_n$.

Let $\mu^{\pm}:S(V^{\pm}) \otimes S(V^{\pm}) \rightarrow S(V^{\pm})$ be
the multiplications on these two algebras.
Then define a multiplication $\mu$ on the vector space $S(V^+) \otimes S(V^-)$
by the composition
$(\mu^+ \otimes \mu^-) \circ (\operatorname{id}_{S(V^+)}
\otimes R_{S(V^-), S(V^+)}^{-1} \otimes \operatorname{id}_{S(V^-)})$.
Since the $R$-matrix is a braiding, this makes $S(V^+) \otimes S(V^-)$
into an associative algebra. Using the formula for $R^{-1}(v_i^-
\otimes v_j^+)$ displayed above, it is easy to check the relations
(\ref{rel1})--(\ref{rel4}) to
show that there is an algebra homomorphism
$$
f: S \rightarrow S(V^+) \otimes S(V^-),
\qquad
x_i \mapsto v_i^+ \otimes 1,
y_j \mapsto 1 \otimes v_j^-.
$$
Also the relations easily give that the anti-dominant monomials $\{u_\tabA\:|\:\tabA \in
\Tab^\circ_{m|n}\}$
span $S^{m|n}$. Moreover, their images under $f$ are a basis for
$S^m(V^+) \otimes S^n(V^-)$.
This shows that $f$ is an isomorphism, thereby establishing the
first statement of the theorem about the monomial basis.

To prove (1), we consider
the diagram
$$
\begin{diagram}
\node[2]{T}\arrow{sw,t,A}{\pi}\arrow{se,t,A}{\pi' := f \circ\pi}\node{}\\
\node{S}\arrow[2]{e,tb}{\sim}{f}\node[2]{S(V^+)\otimes S(V^-).}
\end{diagram}
$$
It suffices to define an anti-linear map $\psi^*:S(V^+) \otimes S(V^-)
\rightarrow S(V^+)\otimes S(V^-)$ such that
$\psi^* \circ \pi' = \pi' \circ \psi^*$, and then show
that
\begin{equation}\label{pens}
\psi^*((x \otimes y) (x' \otimes y'))
= q^{(\alpha+\beta,\alpha'+\beta')-mm' - nn'} \psi^*(x'\otimes y')
\psi^*(x \otimes y)
\end{equation}
for $x \in S^m(V^+)_\alpha, x' \in S^{m'}(V^+)_{\alpha'}, y \in
S^n(V^-)_\beta$ and $y' \in
S^{n'}(V^-)_{\beta'}$.
The generators of the ideal $I^+$ belong to the dual canonical basis
of $V^+ \otimes V^+$, hence, they are fixed by $\psi^*$. This implies
that $I^+$ is $\psi^*$-invariant, hence, $\psi^*:T(V^+) \rightarrow
T(V^+)$ factors through the quotient $S(V^+)$ to induce
$\psi^*:S(V^+) \rightarrow S(V^+)$. The latter map may be defined
directly: it is the unique anti-linear involution that fixes all the
monomials
$v_{i_1}^+\cdots v_{i_m}^+$ for $m \geq 0$ and $i_1 \leq \cdots \leq
i_m$.
Similarly, $\psi^*:T(V^-) \rightarrow T(V^-)$ induces
$\psi^*:S(V^-)\rightarrow S(V^-)$, which fixes
$v_{j_1}^- \cdots v_{j_n}^-$ for $n \geq 0$ and $j_1 \geq \cdots \geq j_n$.
Then we let $\psi^*:S(V^+) \otimes S(V^-) \rightarrow S(V^+) \otimes S(V^-)$
be defined from
$$
\psi^*(x \otimes y) = q^{(\alpha,\beta)} R^{-1}_{S(V^-), S(V^+)}
(\psi^*(y) \otimes \psi^*(x))
$$
for $x \in S(V^+)$ of weight $\alpha$ and $y \in S(V^-)$ of weight
$\beta$.
It is immediate from this definition and (\ref{getstartedb}) that
$\psi^* \circ \pi' = \pi' \circ \psi^*$.
It remains to establish (\ref{pens}).
Let $\tilde \mu^+:S(V^+) \otimes S(V^+)\rightarrow S(V^+)$ be the
twisted multiplication $m^+ \circ R^{-1}_{S(V^+),S(V^+)}$.
Define $\tilde \mu^-:S(V^-)\otimes S(V^-)\rightarrow S(V^-)$
similarly.
Then let $\tilde \mu := (\tilde \mu^+ \otimes \tilde \mu^-) \circ
(\operatorname{id}_{S(V^+)} \otimes R^{-1}_{S(V^-), S(V^+)} \otimes
\operatorname{id}_{S(V^-)}$.
This gives a twisted multiplication on $S(V^+) \otimes S(V^-)$.
Now let $x,y,x'$ and $y'$ be as in (\ref{pens}).
We apply \cite[Lemma 2]{Bdual} to deduce immediately that
$$
\psi^*((x \otimes y) (x' \otimes y'))
= q^{(\alpha+\beta,\alpha'+\beta')}
\tilde{\mu}(\psi^*(x'\otimes y')
\otimes \psi^*(x \otimes y)).
$$
We are thus reduced to checking that
$$
\tilde{\mu}(\psi^*(x'\otimes y')
\otimes \psi^*(x \otimes y))
= q^{-mm'-nn'}
\mu(\psi^*(x'\otimes y')
\otimes \psi^*(x \otimes y)),
$$
which follows as $\tilde \mu^+(x \otimes x') = q^{-mm'}
\mu^+(x\otimes x')$ and
$\tilde \mu^-(y \otimes y') = q^{-nn'} \mu^-(y\otimes y')$, as is pointed out at the
beginning of the proof of \cite[Theorem 16]{Bdual}.

We turn our attention to (2).
If $\tabA \in \Tab_{m|n}$ is {\em not} anti-dominant, then the defining
relations for $S$ imply that
$u_\tabA = \pi(v_\tabA)= q^k u_\tabB$
for $k > 0$ and $\tabA \prec \tabB \in \Tab^\circ_{m|n}$.
Combined with (\ref{soy}), we deduce for $\tabA \in \Tab^\circ_{m|n}$ that
$$
\psi^*(u_\tabA) = u_\tabA + \text{(a $\Z[q,q^{-1}]$-linear combination of
  $u_\tabB$'s
for $\tabA \prec \tabB \in \Tab^\circ_{m|n}$)}.
$$
Hence, we can apply Lusztig's Lemma once again to deduce that
$S^{m|n}$
has another basis $\{d_\tabA\:|\:\tabA \in \Tab_{m|n}^\circ\}$,
with $d_\tabA$ being determined uniquely by the properties
that
$\psi^*(d_\tabA) = d_\tabA$ and $d_\tabA \in u_\tabA + \sum_{\tabB \in \Tab^\circ_{m|n}}
q\Z[q] u_\tabB$. This is the basis appearing in the final statement of the
theorem.
In view of (1), for $\tabA \in \Tab_{m|n}^\circ$,
the vector $\pi(b_\tabA^*)$ satisfies the defining
properties of $d_\tabA$, hence,
$\pi(b_\tabA^*) = d_\tabA$. To complete the proof of (2), we need to show that
$\pi(b_\tabA^*) = 0$ for $\tabA \in \Tab_{m|n} \setminus
\Tab_{m|n}^\circ$. This follows because in that case
$\pi(b_\tabA^*)$ lies in $\bigoplus_{\tabB \in \Tab_{m|n}^\circ} q \Z[q] u_\tabB$,
which contains no non-zero $\psi^*$-invariant vectors.

Finally, we must establish (3).
For this, we first prove the following commutation formulae involving the
$z_i$'s:
\begin{align}
x_j z_i &= \left\{
\begin{array}{rl}
q z_i x_j&\text{if $j > i$,}\\
q^{-1} z_i x_j&\text{if $j\leq i$;}
\end{array}\right. \label{wine1}\\
y_j z_i &= \left\{
\begin{array}{rl}
q^{-1} z_i y_j&\text{if $j > i$,}\\
q z_i y_j&\text{if $j\leq i$;}
\end{array}\right.\label{wine2}\\
z_j z_i &= z_i z_j.\label{wine3}
\end{align}
Actually, we just prove (\ref{wine1}); then the proof of (\ref{wine2})
is similar, and together they obviously imply (\ref{wine3}).
It is obvious that $x_j z_i = q z_i x_j$ for $j > i$.
Also from the definitions we have that
\begin{align}
x_i y_i &= z_i + q z_{i-1},\label{wine4}\\
y_i x_i &= q z_i + z_{i-1}.\label{wine5}
\end{align}
Hence, $z_i x_i = (x_i y_i - q z_{i-1}) x_i = x_i (y_i x_i - z_{i-1})
= q x_i z_i$.
Finally, to show that $z_i x_j = q x_j z_i$ for $i > j$, we proceed by
induction on $i$:
$z_i x_j = (x_i y_i  - q z_{i-1}) x_j = q x_j (x_i y_i - q z_{i-1}) =
q x_j z_i$.

Next we derive the formula for $d_\tabA$ under the assumption that
$m=n=t$. We need to show simply that
$d_\tabA = q^{-t(t-1)/2} z_{c_1} \cdots z_{c_t}$.
Since the $z$'s commute, we may assume that $c_1 \leq \cdots \leq
c_t$.
We proceed by induction on $t$, leaving the base case $t=1$ to the
reader as an exercise.
For the induction step, we have by induction that
$d_{\overline{\tabA}} = q^{-(t-1)(t-2)/2} z_{c_2} \cdots z_{c_{t}}$
where $\overline{\tabA} :=
\substack{c_2 c_3 \cdots c_{t} \\ c_{t} \cdots c_3  c_2 }$,
and must show that $d_\tabA = q^{-(t-1)} z_{c_1} d_{\overline{\tabA}}$.
Expanding the definition of $z_{c_1}$ then commuting
$y$'s past $d_{\overline{\tabA}}$, we get that
$$
q^{-(t-1)} z_{c_1} d_{\overline{\tabA}}=
x_{c_1} d_{\overline{\tabA}} y_{c_1}
-q x_{c_1-1} d_{\overline{\tabA}} y_{c_1-1}
+\cdots + (-q)^{c_1-1} x_{1} d_{\overline{\tabA}} y_{1}.
$$
It follows easily that this vector lies in $u_\tabA + \sum_\tabB q \Z[q] u_\tabB$,
and it just remains to show that it is $\psi^*$-invariant.
Using (\ref{newbarinv}), we have that
$$
\psi^*(q^{-(t-1)} z_{c_1} d_{\overline{\tabA}})
= q^{t-1} q^{-2(t-1)} d_{\overline{\tabA}} z_{c_1}
= q^{-(t-1)} z_{c_1} d_{\overline{\tabA}}.
$$

To complete the proof of (3), assume first that $m=t < n$.
Let $\overline{\tabA}$ be obtained from $\tabA$ by removing the entry
$b_{n-t}$ from its
bottom row. By induction on $n$, we may assume that
$$
d_{\overline{\tabA}} = q^{-t(t-1)/2-\#\{(i,j)\:|\:b_i > c_j +
    \#\{j\:|\:b_{n-t} > c_j\}}
z_{c_1} \cdots z_{c_t} y_{b_{1}} \cdots
y_{b_{n-t-1}}.
$$
We need to show that
$d_\tabA  = q^{-\#\{j\:|\:b_{n-t} > c_j\}}d_{\overline{\tabA}} y_{b_{n-t}}$.
It is easy to see that it equals $u_\tabA$ plus a $q \Z[q]$-linear
combination of other $u_\tabB$'s. Then one checks that it is
$\psi^*$-invariant by a calculation using the commutation formulae and
(\ref{newbarinv}).
Finally, one treats the case $m > t$ in a very similar way: let $\overline{\tabA}$ be $\tabA$ with the entry $a_1$
removed from its top row; by induction we have a formula for
$d_{\overline{\tabA}}$; then one deduces that $d_\tabA = q^{-\#\{j\:|\:a_1 > c_j\}} x_{a_1}
d_{\overline{\tabA}}$.
\end{proof}

Now we switch to the combinatorial framework of \S\ref{morec},
modified slightly since we are working with $\mathfrak{sl}_N$ rather
than $\mathfrak{sl}_\infty$.
We re-use the notation $\lambda \vDash n$ now to indicate that $\lambda$ is
an $N$-part composition of $n$, i.e.\ a sequence
$\lambda = (\lambda_1,\dots,\lambda_N)$ of non-negative integers
summing to $n$.
Fix for the remainder of the appendix integers $m,n \geq 0$ and
a triple
$(\mu,\nu;t)$
such that $0 \leq t \leq \min(m,n)$, $\mu \vDash m-t$, $\nu \vDash
n-t$,
and $\mu_i \nu_i = 0$ for all $i=1,\dots,N$.
For each $\lambda \vDash t$, let $\tabA(\mu,\nu;\lambda)$
be the unique anti-dominant tableau with $\lambda_i+\mu_i$ entries
equal to $i$ on its top row and $\lambda_i+\nu_i$ entries equal to $i$
on its bottom row, for all $i=1,\dots, N$.
We denote $b_\tabA, v_\tabA, b_\tabA^*, u_\tabA$ and
$d_\tabA$ for $\tabA := \tabA(\mu,\nu;\lambda)$ simply by
$b_\lambda,v_\lambda, b_\lambda^*, u_\lambda$ and $d_\lambda$,
respectively.
Set $\gamma := \mu+\nu$.

\begin{Lemma}\label{edible}
For $\lambda,\kappa \vDash t$, the $d_\kappa$-coefficient of
$u_\lambda$ when expanded in terms of the dual canonical basis for
$S^{m|n}$ is non-zero if and only if $\lambda = \kappa-
\sum_{i=1}^{N-1} \theta_i \alpha_i$ for
$(\theta_1,\dots,\theta_{N-1})$ with $0 \leq \theta_i \leq \lambda_i$
for all $i$, in which case the coefficient equals
$$
\prod_{i=1}^{N-1} q^{\theta_i (\lambda_{i+1}+\gamma_{i+1})}
   \sqbinom{\lambda_{i+1}}{\theta_i}.
$$
\end{Lemma}

\begin{proof}
We first observe by induction on $r \geq 0$ that
\begin{equation}\label{last1}
x_i^r y_i^r =
\sum_{s=0}^r q^{sr-r(r-1)/2} \sqbinom{r}{s}
z_i^{r-s} z_{i-1}^s.
\end{equation}
The base case is trivial, while the induction step follows using
(\ref{wine1}), (\ref{wine3}), (\ref{wine4}) and the usual identity
$\sqbinom{r+1}{s} = q^s \sqbinom{r}{s} + q^{s-r-1}\sqbinom{r}{s-1}$.
Combining (\ref{last1}) with (\ref{wine2}), we get also that
\begin{equation}\label{last2}
(x_i^r y_i^r) y_j^s
= q^{-sr} y_j^s (x_i^r y_i^r)
\end{equation}
whenever $j < i$.

Now take any $\lambda \vDash t$ and
set $x^\lambda := x_1^{\lambda_1} \cdots x_N^{\lambda_N},
y^\lambda := y_N^{\lambda_N} \cdots y_1^{\lambda_1}$
and $z^\lambda := z_1^{\lambda_1} \cdots z_N^{\lambda_N}$.
By (\ref{rel1})--(\ref{rel2}) then (\ref{last2}), we have that
\begin{align*}
u_\lambda &=
q^{-\sum_{i < j} \lambda_i (\mu_j+\nu_j)}
x^\mu x^\lambda y^\lambda y^\nu=
q^{-\sum_{i < j} \lambda_i (\lambda_j+\gamma_j)}
x^\mu (x_1^{\lambda_1} y_1^{\lambda_1})
\cdots (x_N^{\lambda_N} y_N^{\lambda_N})
 y^\nu.
\end{align*}
Expanding each $x_i^{\lambda_i} y_i^{\lambda_i}$ here using
(\ref{last1}), we obtain
$$
u_\lambda =
q^{-\sum_{i < j} \lambda_i (\lambda_j+\gamma_j)}
\sum_{\substack{(\theta_0,\theta_1,\dots,\theta_{N-1}) \\ 0 \leq \theta_i \leq
    \lambda_{i+1}}}
\left(\prod_{i = 1}^N
q^{\theta_{i-1}
  \lambda_i-\lambda_i(\lambda_i-1)/2} \sqbinom{\lambda_i}{\theta_{i-1}} \right)
x^\mu z^{(\lambda,\theta)}
y^\nu
$$
where
$z^{(\lambda,\theta)} :=  z_0^{\theta_0} z_1^{\lambda_1+\theta_1-\theta_0} \cdots
z_{N-1}^{\lambda_{N-1}+\theta_{N-1}-\theta_{N-2}}
z_N^{\lambda_N
  - \theta_{N-1}}$.
Since $z_0^{\theta_0} = 0$ unless $\theta_0 = 0$,
and $\sum_{i < j} \lambda_i \lambda_j +\sum_i
\lambda_i(\lambda_i-1)/2 = t(t-1)/2$, this simplifies to
$$
u_\lambda =
q^{-t(t-1)/2-\sum_{i < j} \lambda_i \gamma_j}
\sum_{\substack{(\theta_0, \theta_1,\dots,\theta_{N-1}) \\ 0 \leq \theta_i \leq
    \lambda_{i+1} \\ \theta_0 = 0}}
\left(\prod_{i = 1}^{N-1}
q^{\theta_{i}
  \lambda_{i+1}} \sqbinom{\lambda_{i+1}}{\theta_{i}} \right)
x^\mu z^{(\lambda,\theta)}
y^\nu.
$$
Now pick some $(\theta_0,\theta_1,\dots,\theta_{N-1})$ appearing in
this summation, and set $\kappa := \lambda + \sum_{i=1}^{N-1} \theta_i
\alpha_i \vDash t$. By the formula from
Theorem~\ref{incredible}(3), we get that
$$
d_\kappa = q^{-t(t-1)/2 - \sum_{i < j} \kappa_i \gamma_j} x^\mu
z^\kappa y^\nu.
$$
Moreover,
$$
\sum_{1 \leq i < j \leq N} (\kappa_i - \lambda_i) \gamma_j
= \!\!\sum_{1 \leq i < j \leq N} (\theta_i - \theta_{i-1}) \gamma_j
= \sum_{i=1}^{N-1}
\theta_i \gamma_{i+1}.
$$
The last three identities displayed combine to show that
$$
u_\lambda = \sum_{\substack{(\theta_1,\dots,\theta_{N-1})\\
0 \leq \theta_i \leq \lambda_{i+1}}}
\left(\prod_{i=1}^{N-1} q^{\theta_i (\lambda_{i+1}+\gamma_{i+1})} \sqbinom{\lambda_{i+1}}{\theta_i}\right) d_{\lambda + \sum_{i=1}^{N-1} \theta_i \alpha_i},
$$
and the lemma follows.
\end{proof}

\begin{Lemma}\label{pumpkins}
For any $\lambda,\kappa \vDash t$,
the inner product $(b_\kappa,b_\lambda)$ is non-zero if and
only if
$\kappa=\lambda+ \sum_{i=1}^{N-1} (\lambda_{i+1}-\rho_{i+1}) \alpha_i$
for $\rho = (\rho_1,\dots,\rho_N)$
with $\rho_1 = \lambda_1$ and
$0 \leq \rho_{i+1} \leq \lambda_{i+1}+\min(\lambda_{i},\rho_{i})$ for all
$i=1,\dots,N-1$.
In that case
\begin{align*}
(b_\kappa,b_\lambda)
&=
[m]![n]!\sum_{\tau}
q^{s(\tau)}
\frac{\prod_{i=2}^{N}
\sqbinom{\lambda_{i+1}+\tau_i-\tau_{i+1}}{\tau_{i}-\lambda_{i}}
\sqbinom{\lambda_{i+1}+\tau_i-\tau_{i+1}}{\tau_{i}-\rho_{i}}}
{\prod_{i=1}^N[\lambda_{i+1}+\tau_i-\tau_{i+1}]!
[\lambda_{i+1}+\tau_i-\tau_{i+1}\!+\!\gamma_i]!}\,,
\end{align*}
where we interpret $\lambda_{N+1}$ as zero,
the summation is over $\tau =
(\tau_1,\dots,\tau_{N+1})$
with $\tau_1=\lambda_1, \tau_{N+1}=0$
and
$\max(\lambda_{i+1},\rho_{i+1})\leq \tau_{i+1} \leq
\lambda_{i+1}+\min(\lambda_i,\rho_i)$
for $i=1,\dots,N-1$, and
\begin{multline*}
s(\tau) := \binom{m}{2}+\binom{n}{2}
+
\sum_{i=2}^N(2\tau_i-\lambda_i-\rho_i)(\lambda_{i+1}+\tau_i-\tau_{i+1}+\gamma_{i})\\
-\sum_{i=1}^N \binom{\lambda_{i+1}+\tau_i-\tau_{i+1}}{2}
-\sum_{i=1}^N \binom{\lambda_{i+1}+\tau_i-\tau_{i+1}+\gamma_i}{2}.
\end{multline*}
\end{Lemma}

\begin{proof}
We have that $b_\lambda = \sum_{\tabB \in \Tab_{m|n}} (b_\lambda, v_\tabB) v_\tabB$.
Hence,
\begin{equation}\label{pho}
(b_\kappa,b_\lambda) = \sum_{\tabB \in \Tab_{m|n}} (b_\kappa,v_\tabB) (b_\lambda,v_\tabB)
=\sum_{\beta \vDash t}
\bigg[\sum_{\tabB \sim \tabA(\mu,\nu;\beta)} (b_\kappa,v_\tabB) (b_\lambda,v_\tabB)\bigg].
\end{equation}
To compute the number $(b_\kappa,v_\tabB)$
appearing on the right hand side of this
formula, we have that $v_\tabB = \sum_{\tabA \in \Tab_{m|n}} (b_\tabA, v_\tabB)
b_\tabA^*$.
Hence, in view of Theorem~\ref{incredible}(2), we can compute $(b_\kappa,v_\tabB)$ by
applying $\pi$: it is the $d_\kappa$-coefficient of $u_\tabB =
\pi(v_\tabB)$ when expanded in terms of the dual canonical basis of $S^{m|n}$.
For $\tabB=
\substack{a_1 \cdots a_{m} \\ b_1 \cdots b_{n}}$,
we let $\ell(\tabB) := \#\{i < j\:|\:a_i > a_j\} + \#\{i < j\:|\:b_i < b_j\}$
so that $u_\tabB = q^{\ell(\tabB)} u_\beta$.
Then Lemma~\ref{edible} shows that $(b_\kappa,v_\tabB)$ is non-zero only if
$\beta = \kappa -\sum_{i=1}^{N-1} \theta_i \alpha_i$ for
$(\theta_1,\dots,\theta_{N-1})$ with $0 \leq \theta_i \leq
\kappa_{i}$ for each $i$, in which case
$$
(b_\kappa,v_\tabB) = q^{\ell(\tabB)} \prod_{i=1}^{N-1} q^{\theta_i(\beta_{i+1}
  + \gamma_{i+1})} \sqbinom{\beta_{i+1}}{\theta_i}.
$$
Similarly, $(b_\lambda,v_\tabB)$ is non-zero only if
$\beta=\lambda -\sum_{i=1}^{N-1} \phi_i \alpha_i$ for
$(\phi_1,\dots,\phi_{N-1})$ with $0 \leq \phi_i \leq
\lambda_{i}$ for each $i$, in which case
$$
(b_\lambda,v_\tabB) = q^{\ell(\tabB)}\prod_{i=1}^{N-1} q^{\phi_i(\beta_{i+1}
  + \gamma_{i+1})} \sqbinom{\beta_{i+1}}{\phi_i}.
$$
Observe also that
\begin{align*}
\sum_{\tabB \sim \tabA(\mu,\nu;\beta)} q^{2 \ell(\tabB)}
&= \frac{q^{\binom{m}{2}}
[m]!q^{\binom{n}{2}}[n]!}{\prod_{i=1}^N q^{\binom{\beta_i+\mu_i}{2}}[\beta_i+\mu_i]!
q^{\binom{\beta_i+\nu_i}{2}}[\beta_i+\nu_i]!}\\
&= \frac{q^{\binom{m}{2}}
[m]!q^{\binom{n}{2}}[n]!}{\prod_{i=1}^N q^{\binom{\beta_i}{2}}[\beta_i]!
q^{\binom{\beta_i+\gamma_i}{2}}[\beta_i+\gamma_i]!}.
\end{align*}
Putting these observations together, we deduce that
the $\beta$th summand on the right hand side of (\ref{pho}) is
non-zero only if
$\beta  = \kappa -\sum_{i=1}^{N-1} \theta_i \alpha_i = \lambda - \sum_{i=1}^{N-1} \phi_i \alpha_i$
for $(\theta_1,\dots,\theta_{N-1})$ and $(\phi_1,\dots,\phi_{N-1})$
satisfying
$0 \leq \theta_i \leq \kappa_i, 0 \leq \phi_i \leq \lambda_{i}$ for all $i=1,\dots,N-1$, in which case it equals
\begin{equation}\label{drips}
q^{\binom{m}{2}}
 [m]!
q^{\binom{n}{2}}
[n]!
\frac{\prod_{i=1}^{N-1} q^{(\phi_i+\theta_i)(\beta_{i+1}
  + \gamma_{i+1})}
\sqbinom{\beta_{i+1}}{\phi_i}
\sqbinom{\beta_{i+1}}{\theta_i}
}
{\prod_{i=1}^N q^{\binom{\beta_i}{2}}[\beta_i]!
q^{\binom{\beta_i+\gamma_i}{2}}[\beta_i\!+\!\gamma_i]!}.
\end{equation}
Now we complete the proof by repeating the last part of the proof
of Theorem~\ref{crazy}:
in the formula (\ref{drips}), we replace $\phi_i$ by $\tau_{i+1}-\lambda_{i+1}$ and
$\theta_i$ by $\tau_{i+1}-\rho_{i+1}$,
to deduce that the $\beta$th summand of (\ref{pho})
gives a non-zero contribution only if
there exist $(\rho_2,\dots,\rho_{N})$ and $(\tau_2,\dots,\tau_N)$
such that
$\beta = \lambda + \sum_{i=1}^{N-1} (\lambda_{i+1}-\tau_{i+1}) \alpha_i$,
$\kappa = \lambda + \sum_{i=1}^{N-1} (\lambda_{i+1}-\rho_{i+1}) \alpha_i$,
and $\max(\lambda_{i+1},\rho_{i+1}) \leq \tau_{i+1} \leq
\lambda_{i+1}+\min(\lambda_i,\rho_i)$
for all $i=1,\dots,N-1$, interpreting $\rho_1$ as $\lambda_1$.
Then we simplify as before.
\end{proof}

\proof[Proof of Theorem~\ref{crazier}]
This follows from Lemma~\ref{pumpkins}
together with
the discussion in \cite[\S5.9]{BLW}, on passing to the limit as $N \rightarrow \infty$.
The main point is that the canonical basis $\{b_\tabA\:|\:\tabA \in
\Tab_{m|n}\}$ corresponds to the indecomposable graded projectives in
the graded lift of $\O_\Z$ constructed in {\em loc. cit.}
thanks to \cite[Corollary 5.30]{BLW} (and \cite[Theorem 3.10]{BLW}).
The bilinear form $(-,-)$
is the same as the pairing from \cite[(5.30)]{BLW}.
\endproof

\end{document}